\documentclass[12pt,twoside]{article}
\usepackage{fullpage}
\usepackage[utf8]{inputenc}
\usepackage{macros}
\usepackage{paperstyle}

\title{Geometric Constructions of Mod $p$ Cohomology Operations}
\author{Herng Yi Cheng}

\begin{document}

\maketitle

\begin{abstract}
    The Brown Representability Theorem implies that cohomology operations can be represented by continuous maps between Eilenberg-Maclane spaces. These Eilenberg-Maclane spaces have explicit geometric models as spaces of cycles on round spheres and spaces of relative cycles on unit disks, due to the Almgren Isomorphism Theorem. A.~Nabutovsky asked what maps between spaces of cycles represent the Steenrod squares.

    In this work we answer this question by constructing maps with explicit formulas from spaces of cycles on spheres to spaces of relative cycles on disks that represent all Steenrod squares, as well as all Steenrod powers and Bockstein homomorphisms on mod $p$ cohomology, for all primes $p$.
\end{abstract}

\section{Introduction}

In this work we will give explicit geometric constructions that represent mod $p$ cohomology operations, chief among them the Steenrod powers and the Bockstein homomorphisms.

For the rest of this article, let $p$ be a given prime, and regard all homology and cohomology to be in $\ZZ_p$ coefficients unless otherwise indicated. By a \emph{mod $p$ cohomology operation} we mean, for some fixed $m, n \geq 0$, a family of homomorphisms $\theta_X : H^m(X) \to H^n(X)$ for each space $X$ which are natural with respect to pullbacks along continuous maps between spaces. The cup powers $\alpha \mapsto \alpha^p$ are immediate examples of mod $p$ cohomology operations. Other examples include the \emph{Bockstein homomorphisms} $\beta : H^n(X) \to H^{n+1}(X)$ for each $n \geq 1$, which are defined as the connecting homomorphisms in the long exact sequence of cohomology groups that arises from the short exact sequence of coefficient groups $0 \to \ZZ_p \to \ZZ_{p^2} \to \ZZ_p \to 0$. Basically, the Bockstein homomorphisms answers the following question about each mod $p$ cocycle: if we treat it as a mod $p^2$ cochain, how many times is its coboundary divisible by $p$?

One final important family of mod $p$ cohomology operations are the \emph{Steenrod powers} $P^i : H^n(X) \to H^{n + 2i(p-1)}(X)$ for $n, i \geq 0$, which were introduced by N.~E.~Steenrod \cite{Steenrod_Squares,Steenrod_ReducedPowers_A,Steenrod_ReducedPowers_B}. Note that $P^0 = \id$. When $p = 2$, $P^i$ and $\beta \circ P^i$ are called \emph{Steenrod squares} and are written using an alternate notation as $\Sq^{2i}$ and $\Sq^{2i+1}$ respectively. As a result, $\Sq^1 = \beta$. The Steenrod powers are characterized by a list of algebraic axioms as explained in \cite[Section~4.L]{Hatcher_AlgTop}, and can be constructed using the equivariant cohomology of $\ZZ_p$ group actions, which we explain in \cref{sec:SteenrodPowersEquivCohom}. On the other hand, it has been challenging to provide geometric intuition for them. When $X$ is a closed manifold and $\alpha \in H^n(X)$ has its Poincar\'e dual represented by an embedded submanifold $\Sigma \subset X$, in some sense $P^i(\alpha)$ detects twisting in the normal bundle of $\Sigma$.

The Steenrod powers and Bockstein homomorphisms generate the mod $p$ cohomology operations, in the sense that every such operation is the $\ZZ_p$-linear combination of cup products of compositions of Steenrod powers and Bockstein homomorphisms. Thus we may restrict our study to the Steenrod powers and Bockstein homomorphisms.

\subsection{Brown representability and the main theorems}

To find geometric representations of the Bockstein homomorphisms and Steenrod powers, our starting point is the \emph{Brown Representability Theorem}, which gives isomorphisms $H^n(X) \iso [X, K(\ZZ_p,n)]$, where $X$ is a cell complex, $[X,Y]$ consists of the homotopy classes of pointed continuous maps $X \to Y$ for some choice of basepoints on $X$ and $Y$, and $K(\ZZ_p,n)$ denotes an Eilenberg-Maclane space. Under this isomorphism, each pointed continuous map $a : X \to K(\ZZ_p,n)$ corresponds to $a^*(\iota_n)$, where $\iota_n \in H^n(K(\ZZ_p,n))$ is the \emph{fundamental cohomology class}. Inspired by this, we will geometrically represent $\alpha \in H^n(X)$ by choosing a space $Y \whe K(\ZZ_p,n)$, where $\whe$ denotes weak homotopy equivalence, so that $Y$ has an explicit geometry. (In contrast, even though $K(\ZZ_p,n)$ can always be represented by a cell complex, the number of cells it has in each dimension is not even precisely known.) Then we will construct a ``geometrically nice'' continuous map $a : X \to Y$ such that $a^*(\iota_n) = \alpha$. (Since $Y \whe K(\ZZ_p,n)$, $\iota_n$ may also be considered as the fundamental cohomology class of $Y$.) Such a map $a : X \to Y$ will be called a \emph{Brown representative} for $\alpha$.

Our choice for $Y$ comes from the field of Geometric Measure Theory: when $M$ is a compact Riemannian manifold with a submanifold $N$, let $\Z{k}(M, N)$ denote the space of \emph{mod $p$ integral (relative) $k$-cycles in $M$.} Intuitively, the elements of this space can be thought of as ``limits'' of mod $p$ singular chains in $M$ whose singular simplices are Lipschitz, and whose boundaries are supported in $N$. We will write $\Z{k}(M)$ to mean $\Z{k}(M, \emptyset)$. This set can be given several natural topologies. A common one is the \emph{flat topology} which is induced by the \emph{flat metric} $\Fl$, where $\Fl(S,T)$ is roughly the area of the smallest ``filling'' of $S - T$. (By a filling of $S - T$, we mean a chain whose boundary is $S - T$.) However, the maps between spaces of cycles that we will construct, including the Cartesian product map $T \mapsto T \times T$, will not be continuous in the flat topology. For this reason, we will instead use the \emph{inductive limit topology} \cite[(1.9)]{Almgren_HomotopyGroupsIntegralCycles}, which is a refinement of the flat topology.

These spaces of cycles have a group structure akin to that of the groups of singular cycles, and as a result they are weakly homotopy equivalent to products of Eilenberg-Maclane spaces \cite[Corollary~4K.7]{Hatcher_AlgTop}. In fact, we will prove the weak homotopy equivalences $K(\ZZ_p,n) \whe \Z{k}(\Sp{n+k}) \whe \Z{k}(\D^{n+k}, \partial\D^{n+k})$ for all $k \geq 0$. More generally, the following identity holds:
\begin{align}
    \label{eq:AlmgrenIsoFMetric}
    \pi_i(\Z{k}(M, \partial M)) \iso H_{i+k}(M, \partial M) && \text{for all } k \geq 0 \text{ and } i \geq 1.
\end{align}
This is an analogue of the \emph{Almgren Isomorphism Theorem}, the same statement for similar spaces of cycles \cite{Almgren_HomotopyGroupsIntegralCycles,GuthLiokumovich_ParamIneq}. It can be viewed as a generalization of the Dold-Thom Theorem. Thus we may find Brown representatives of cohomology classes which are nice maps to spaces of cycles, where the degree of the class is equal to the codimension of the cycle in the ambient sphere or disk.

\begin{remark}
    \Cref{eq:AlmgrenIsoFMetric} follows from the Almgren Isomorphism Theorem for spaces of cycles with the flat topology and the fact that on such a space of cycles, the inductive limit topology is weakly homotopy equivalent to the flat topology. This was stated in \cite{Almgren_HomotopyGroupsIntegralCycles} for spaces of cycles with $\ZZ$ coefficients, but we have not been able to find a proof in the literature. Nevertheless, these weak homotopy equivalences follow immediately from the approximation theorems in \cite{GuthLiokumovich_ParamIneq}, which allow continuous families of cycles parametrized by compact sets to be homotoped to families of cycles that are bounded in mass.
\end{remark}

Consider the following examples of Brown representatives for the generators $H^1(\Sp1)$ and $H^1(\RP^2)$, for $p = 2$. The former is a map $a : \Sp1 \to \Z1(\Sp2)$ whose formula, if we parametrize $\Sp1$ as $[-1,1]$, is $a(t) = \{t\} \times \RR^2 \cap \Sp2$. (It is well-defined because $a(-1)$ and $a(1)$ are both points, which are equal to zero as 1-cycles.) It gives a 1-parameter family of circles that ``sweep out'' $\Sp2$, starting from 0 and ending at 0 (see \cref{fig:BrownRepExamples}(a)). The latter is a map $\RP^2 \to \Z2(\D^3, \partial \D^3)$. If $\RP^2$ is considered as the space of lines through the origin in $\RR^3$, then the map has formula $\ell \mapsto \ell^\perp \cap \D^3$ (see \cref{fig:BrownRepExamples}(b)--(c)).

\begin{figure}[h]
    \centering
    \begin{tabular}{cc}
        (a) & (b) \hspace{13em} (c) \\
        \includegraphics[height=4.5cm]{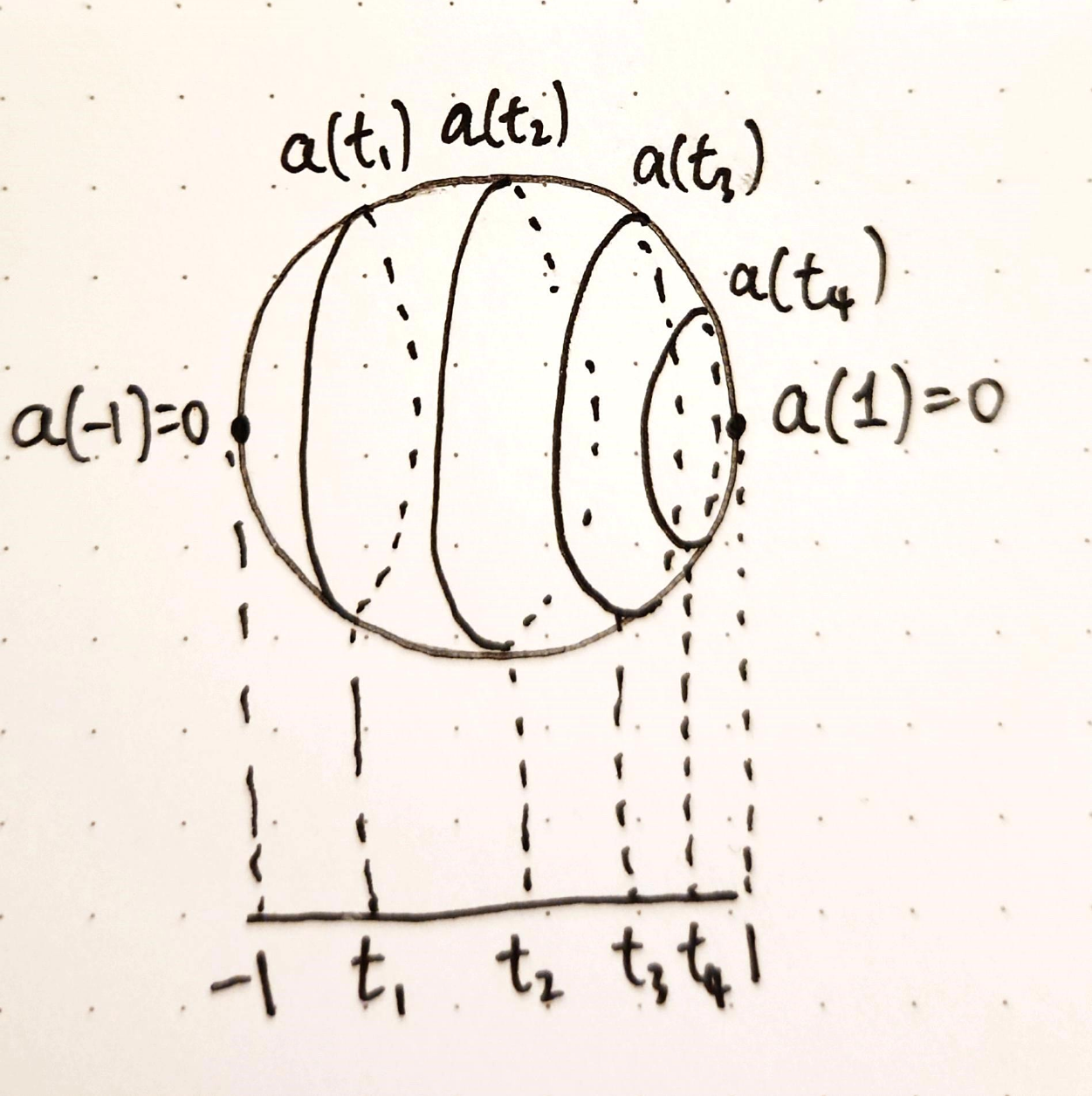} &
        \includegraphics[height=4.5cm]{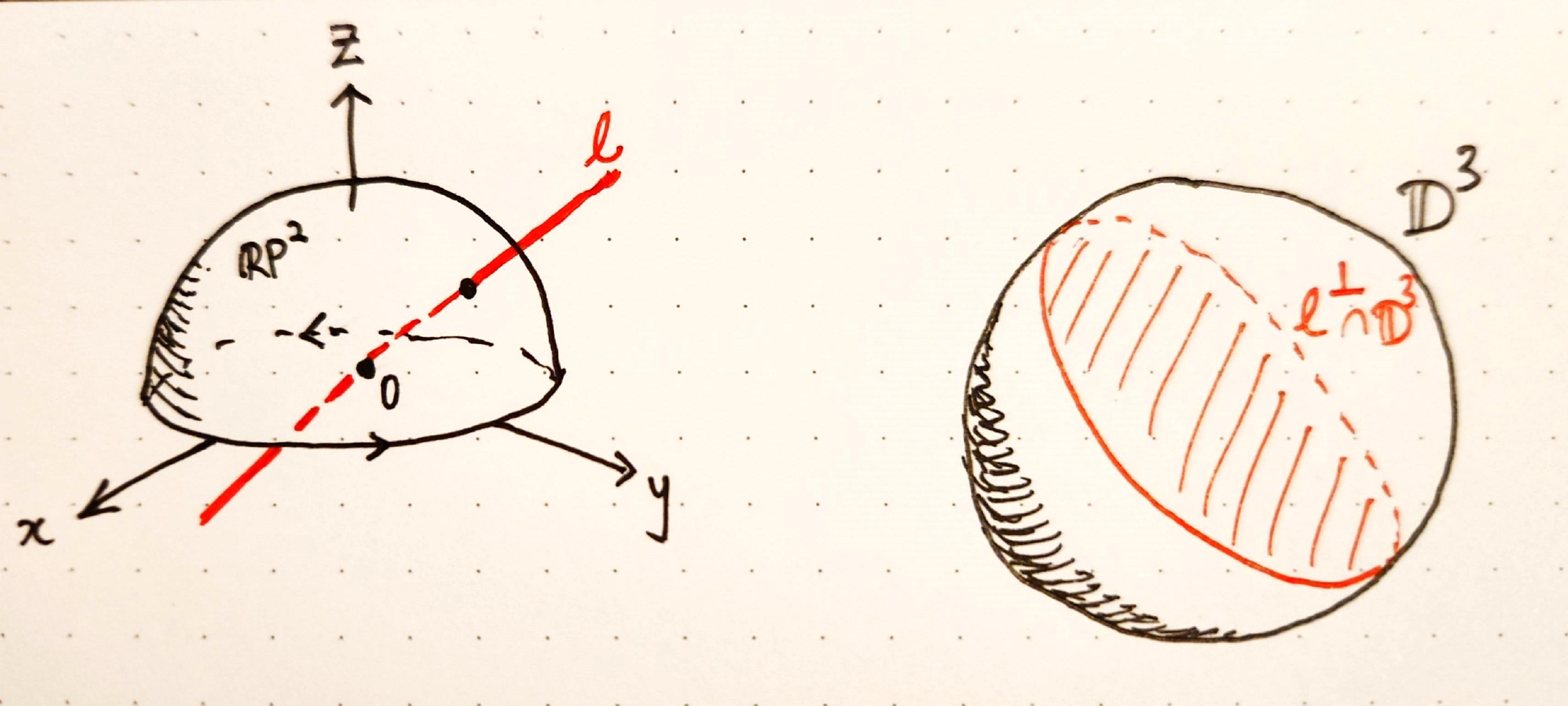}
    \end{tabular}
    
    \caption{(Suppose that $p = 2$.) (a) A Brown representative $a : \Sp1 \to \Z1(\Sp2)$ of the generator of $H^1(\Sp1)$, where $\Sp1$ is parametrized as $[-1,1]$. (b) $\RP^2$ can be visualized as the upper hemisphere of $S^2$, with antipodal points on the equator identified. Each point on the upper hemisphere corresponds to a line $\ell$ through the origin. (c) A Brown representative of the generator of $H^1(\RP2)$ sends $\ell \in \RP^2$ to $\ell^\perp \cap \D^3 \in \Z2(\D^3, \partial\D^3)$.}
    \label{fig:BrownRepExamples}
\end{figure}

Going a step further, the Brown representability theorem also implies that the set of mod $p$ cohomology operations $H^m(-) \to H^n(-)$ bijects with $[K(\ZZ_p,m), K(\ZZ_p,n)]$. A map $f : K(\ZZ_p,m) \to K(\ZZ_p,n)$ corresponds to the cohomology operation $\phi$ so that applying $\phi$ corresponds to composing with $f$: if a class $\alpha \in H^m(X)$ has Brown representative $a : X \to K(\ZZ_p,m)$, then $\phi(\alpha) \in H^n(X)$ has Brown representative $f \circ a$. This is equivalent to the identity $f^*(\iota_n) = \phi(\iota_m)$. In this work, we have geometrically represented the cohomology operation $\phi : H^m(-) \to H^n(-)$, where $\phi$ could be a Bockstein homomorphism or a Steenrod power, by choosing spaces of cycles $Y \whe K(\ZZ_p,m)$ and $Y' \whe K(\ZZ_p,n)$ and constructing geometrically nice maps $f : Y \to Y'$ such that $f^*(\iota_n) = \phi(\iota_m)$. We call $f$ a \emph{Brown representative for $\phi$.} 

We have constructed Brown representatives for the Bockstein homomorphisms as follows. The domains of these maps are $\Z0(\Sp{n})$, which consists of $\ZZ_p$-linear combinations of points in $\Sp{n}$ whose coefficients should sum to zero. (This condition makes it a connected space.) The codomains are $\Z0(\D^n, \partial\D^n)$, which consist of $\ZZ_p$-linear combinations of points in $\D^n$, but with no further restrictions on their coefficients. As these are cycles relative to the boundary, any points on the boundary are considered to vanish. For every set $X$, $\ZZ_p$ acts on the Cartesian power $X^p$ by cyclic permutations which are generated by $(x_1, \dotsc, x_p) \mapsto (x_2, \dotsc, x_p, x_1)$. Let $X^p/\ZZ_p$ denote the quotient by this action.

\begin{theorem}
    \label{thm:Bockstein}
    For each $n \geq 1$, the Bockstein homomorphism $\beta : H^n(-) \to H^{n+1}(-)$ has a Brown representative
    \begin{equation}
        \label{eq:Bockstein}
        \begin{aligned}
            &b : \Z0(\Sp{n}) \to \Z0(\D^{n+1}, \partial\D^{n+1})
            \\
            &b(x_1 + \dotsb + x_k) = \sum_{[(i_1,\dotsc,i_p)] \in \{1,\dotsc,k\}^p/\ZZ_p} \frac{x_{i_1} + \dotsb + x_{i_p}}{p},
        \end{aligned}
    \end{equation}
    where each $x_i \in \Sp{n}$ and the fraction denotes the barycenter of $x_{i_1}, \dotsc, x_{i_p}$, considered as points in $\RR^{n+1}$.

    In particular, when $p = 2$, $b(x_1 + \dotsb + x_k)$ is the sum of the midpoints of every unordered pair $\{x_i,x_j\}$.
\end{theorem}

An example for $b$ when $n = 1$ are illustrated for $p = 2$ in \cref{fig:Bockstein}(a)--(b), which shows how it computes a sum of midpoints. A similar example for $n = 1$ and $p = 3$ in \cref{fig:Bockstein}(c)--(d) requires more explanation. When the input to $b$ is a mod 3 cycle $x_1 + x_2 + x_3$ (see \cref{fig:Bockstein}(c)), the sum in \cref{eq:Bockstein} is indexed by the elements of $\{1,2,3\}^3/\ZZ_3$, which are 3-tuples of indices with cyclic permutations identified. One representative from each equivalence class is listed below:
\begin{equation*}
\begin{array}{cccc}
    (1,1,1) & (1,1,2) & (1,2,2) & (1,2,3) \\
    (2,2,2) & (2,2,3) & (2,3,3) & (3,2,1) \\
    (3,3,3) & (3,3,1) & (3,1,1)
\end{array}
\end{equation*}
For each $(i,j,k)$ listed above, the output of $b$ contains the point $\frac13(x_i + x_j + x_k)$. The points corresponding to $(i,i,i)$ lie on the boundary and vanish. Both $(1,2,3)$ and $(3,2,1)$ correspond to the point $\frac13(x_1 + x_2 + x_3)$, so that point has multiplicity 2. The points in the resulting mod 3 0-cycle trisect the edges of the triangle $x_1x_2x_3$, and also mark its barycenter (see \cref{fig:Bockstein}(d)).

\begin{figure}[h]
    \centering
    \begin{tabular}{cc}
        (a) & (b)
        \\
        \includegraphics[height=5cm]{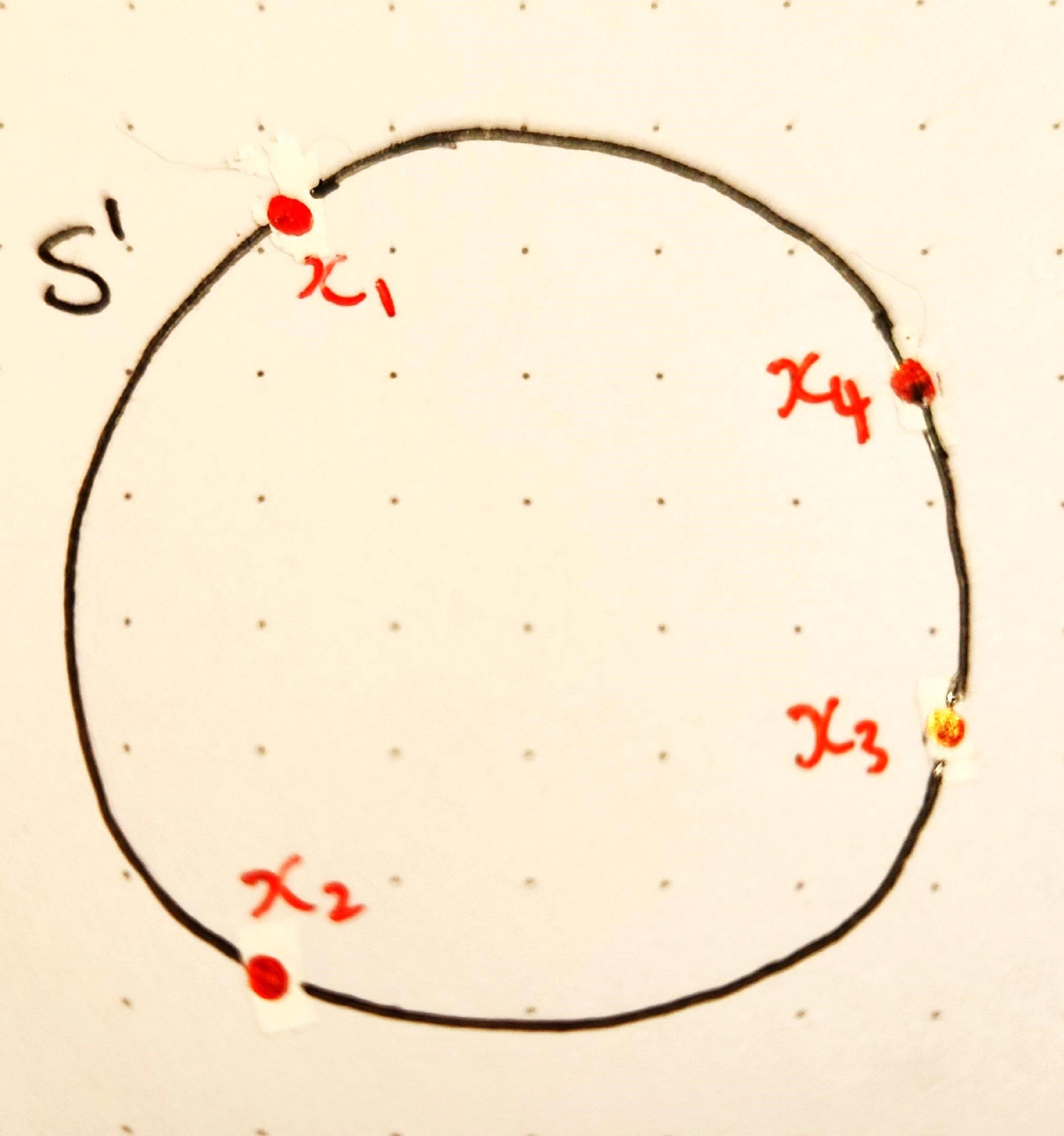} &
        \includegraphics[height=5cm]{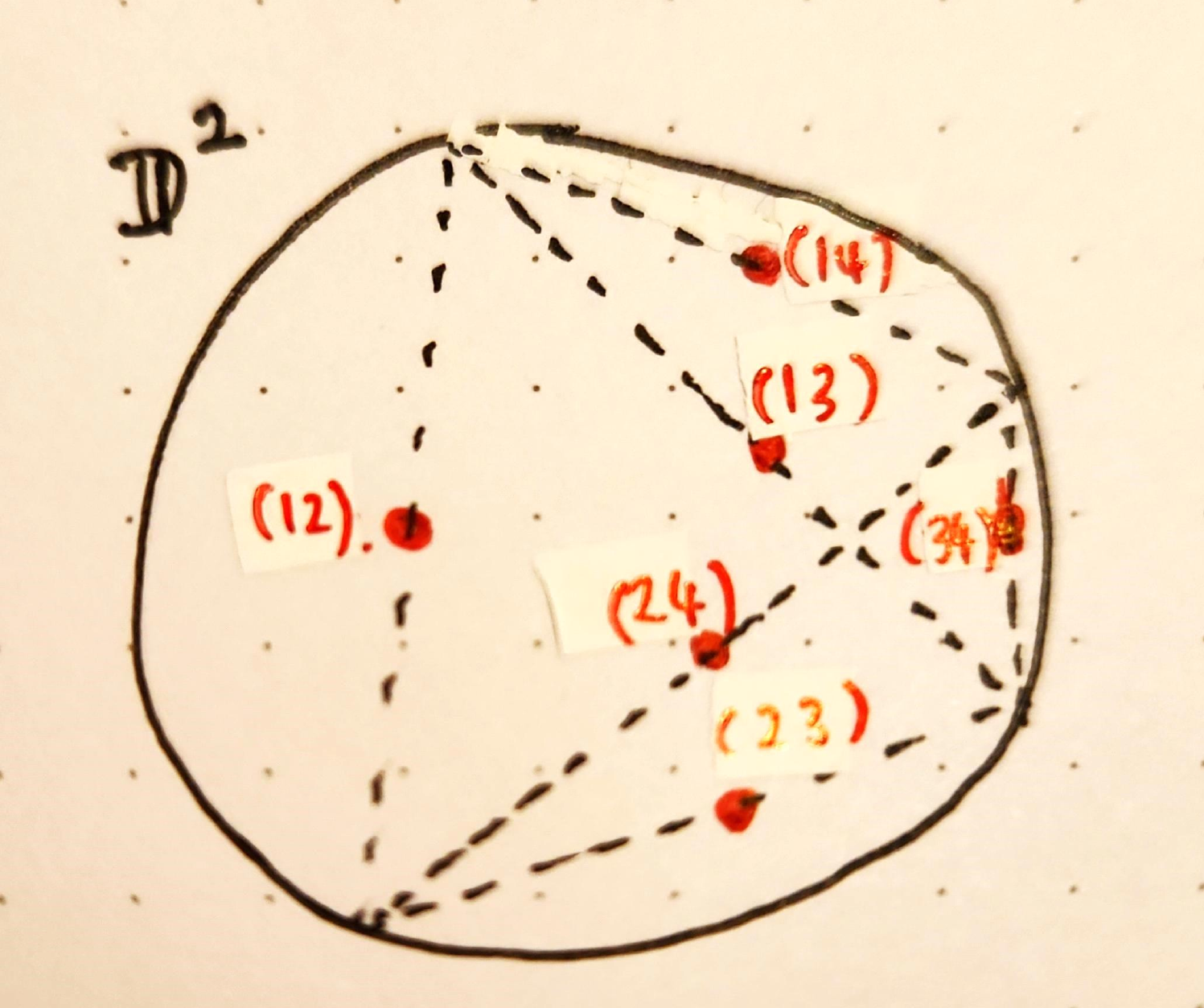}
        \\
        (c) & (d)
        \\
        \includegraphics[height=5cm]{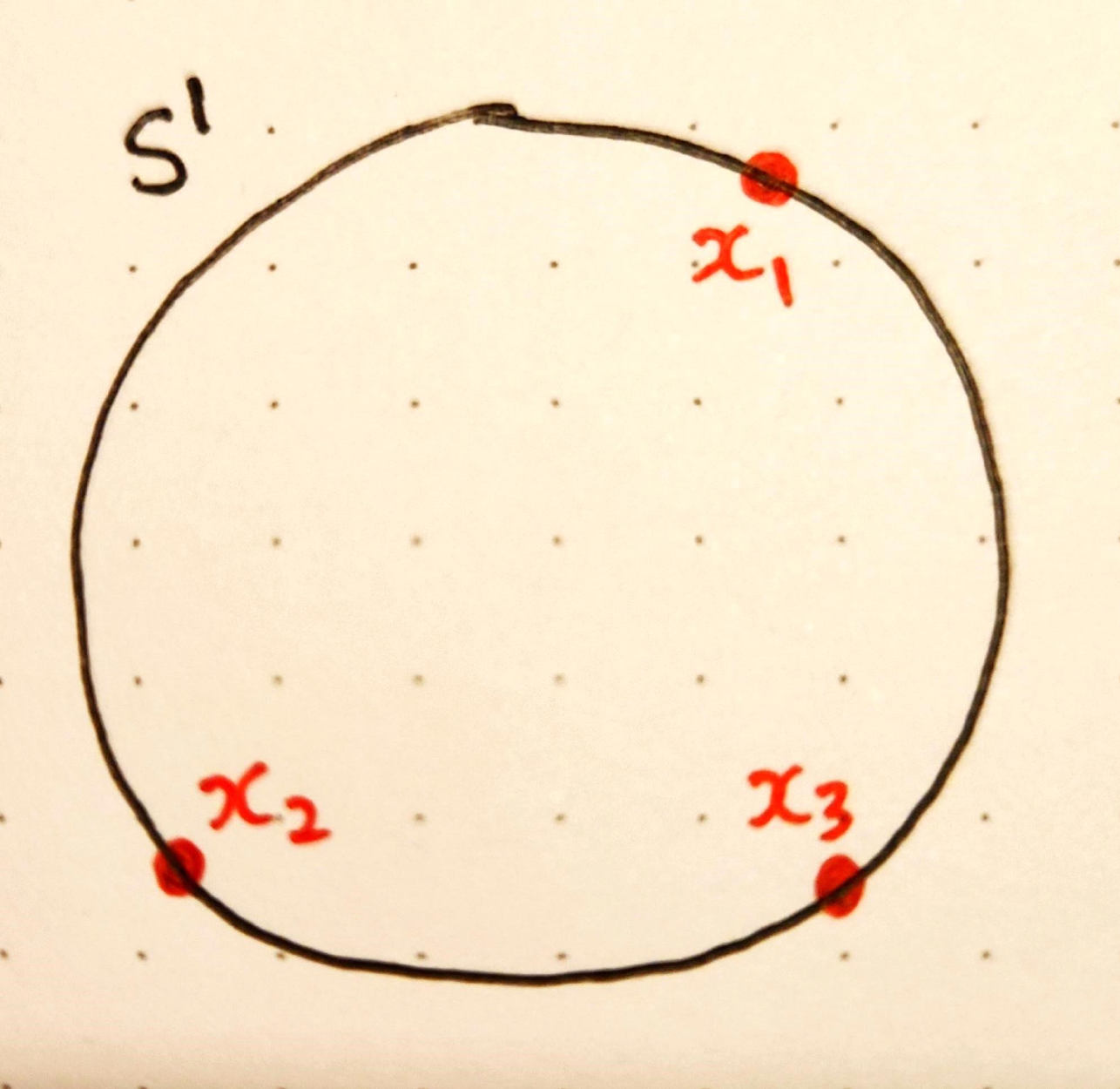} &
        \includegraphics[height=5cm]{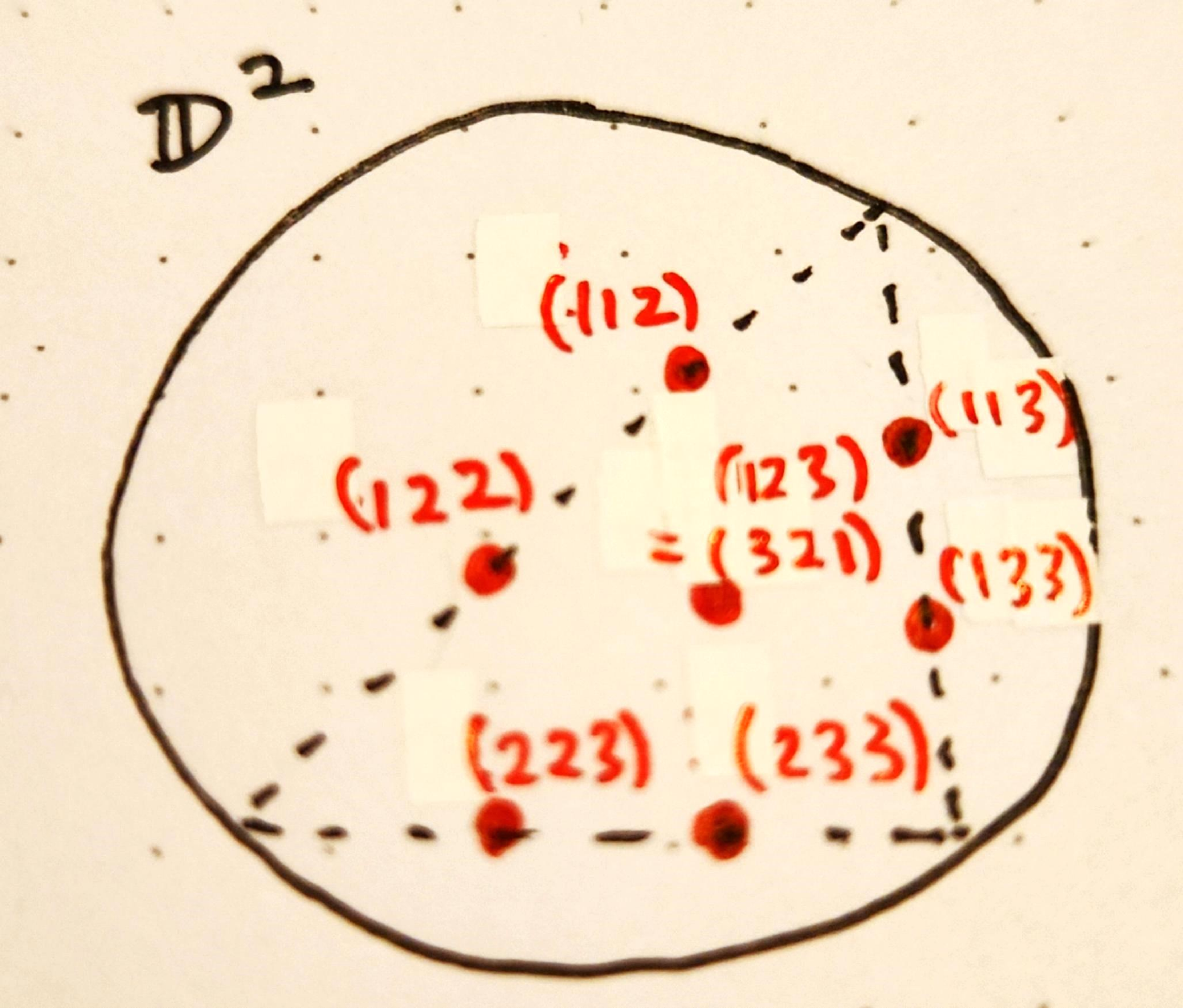}
    \end{tabular}
    
    \caption{Brown representatives $b : \Z0(\Sp1) \to \Z0(\D^2, \partial\D^2)$ of the Bockstein homomorphism $\beta : H^1(-) \to H^2(-)$ for $p = 2$ in (a)--(b) and $p = 3$ in (c)--(d). (a) The mod 2 0-cycle $x_1 + \dotsb + x_4$ shown as red points. (b) $b(x_1 + \dotsb + x_4)$ shown as red points, where $(ij)$ denotes the midpoint $\frac12(x_1 + x_j)$. (c) The mod 3 0-cycle $x_1 + x_2 + x_3$ shown as red points. (b) $b(x_1 + x_2 + x_3)$ shown as red points, where $(ijk)$ denotes the barycenter $\frac13(x_1 + x_j + x_k)$. The point $(123) = (321)$ has multiplicity 2.}
    \label{fig:Bockstein}
\end{figure}

The action of $\ZZ_p$ by cyclic permutations is fundamental to the construction of the Steenrod powers \cite{Steenrod_CohomOps}, so naturally it also plays a central role in the Brown representatives that we construct for Steenrod powers. Roughly speaking, every Brown representative of a Steenrod power and Bockstein homomorphism whose domain is $\Z{k}(\Sp{n})$ is encapsulated in a single \emph{cyclic product map}, denoted by $\cyc$, that takes a cycle $T \in \Z{k}(\Sp{n})$ and returns $T^p/\ZZ_p$. Brown representatives for individual Steenrod powers $P^i$ and the Bockstein homomorphism can be obtained by composing $\cyc$ with other maps.

To explain $\cyc$ in more detail, observe that there is an inclusion $(\Sp{n})^p \hookrightarrow \Sp{p(n+1)-1}$ where $\Sp{p(n+1)-1}$ is viewed as a sphere of radius $\sqrt{p}$. Since $\Sp{p(n+1)-1} \subset (\RR^{n+1})^p$, $\ZZ_p$ acts on it by cyclic permutations, with fixed points the diagonal $\Delta = \{(x,\dotsc,x) : x \in \Sp{n}\}$. Deleting the diagonal from $\Sp{p(n+1)-1}$ and then taking the quotient by the action yields a familiar space: there is a homeomorphism
\begin{align}
    \label{eq:SphereQuotientLensBundle}
    h : (\Sp{p(n+1)-1} \setminus \Delta)/\ZZ_p \to L_n \times \itr \D^{n+1} && \text{where } L_n = \begin{cases}
        \RP^{(p-1)(n+1)-1} & p = 2
        \\
        \Lens_p((1,\dotsc,\frac{p-1}2)^{n+1}) & p > 2,
    \end{cases}
\end{align}
and $\itr$ denotes the interior. $L_n$ is a general lens space specified with the notation of \cite[p.~144]{Hatcher_AlgTop}. The notation $(1,\dotsc,\frac{p-1}2)^{n+1}$ simply means that the tuple of indices $(1,\dotsc,\frac{p-1}2)$ is repeated $n+1$ times.

\begin{theorem}
    \label{thm:SteenrodPowers}
    For each prime $p$, $i \geq 0$, and $0 \leq k < n$. Let $m = n - k$. Then there exists a cyclic product map,
    \begin{equation}
        \label{eq:CycProdMap_Intuitive}
    \begin{aligned}
        &\cyc : \Z{k}(\Sp{n}) \to \Z{pk}(L_n \times \D^{n+1}, L_n \times \partial\D^{n+1})
        \\
        &\cyc(T) = \overline{h((T^p \setminus \Delta)/\ZZ_p)},
    \end{aligned}
    \end{equation}
    from which Steenrod powers and Bockstein homomorphisms may be derived. More precisely, $P^i : H^m(-) \to H^{m + 2i(p-1)}(-)$ and $\beta \circ P^i : H^m(-) \to H^{m+2i(p-1)+1}(-)$ have Brown representatives
    \begin{align*}
        T \mapsto \bigcup_{x \in \cyc(T)} g_{pk+m+2i(p-1)}(x) &&\text{and}&& T \mapsto \bigcup_{x \in \cyc(T)} g_{pk+m+2i(p-1)+1}(x) && \text{respectively,}
    \end{align*}
    where each $g_q : L_n \times \D^{n+1} \to \Z{a}(\D^{a+q}, \partial\D^{a+q})$ is a Brown representative for a generator of $H^q(L_n \times \D^{n+1}, L_n \times \partial\D^{n+1})$ for some $a \geq 0$ that may depend on $q$.

    (The maps $g_q$ have explicit formulas, as explained in \cref{rem:ProjLensSpacesCohom_BrownReps}.)
\end{theorem}

\begin{remark}
    By a Brown representative of a generator of $H^q(L_n \times \D^{n+1}, L_n \times \partial\D^{n+1})$, we mean a Brown representative of the corresponding cohomology generator of the quotient of $L_n \times \D^{n+1}$ by its boundary. This quotient is the Thom space of a trivial real vector bundle over $L_n$, and the Brown representative can be chosen in the form of a map from $L_n \times \D^{n+1}$ to a space of cycles, so that the map vanishes over the boundary.
\end{remark}

\begin{remark}
    The map $h$ is induced by multiplying with a $p(n+1) \times p(n+1)$ matrix which is a real analogue of a Discrete Fourier Transform matrix. This arises from the fact that the action of $\ZZ_p$ on $(\RR^{n+1})^p$ gives a representation of $\ZZ_p$ which splits into a direct sum of $n+1$ trivial representations and, when $p$ is odd, $\frac{p-1}2(n+1)$ 2-dimensional representations. Each 2-dimensional representation is a rotation by a multiple of $2\pi/p$. The trivial representations correspond to the $\D^{n+1}$ factor, while the other representations correspond to the lens space factor.
\end{remark}

\begin{remark}
    For simplicity, the formulas in \cref{thm:SteenrodPowers} manipulate currents as thoough they were sets, by identifying currents with their supports. Strictly speaking, the operations of unions and closure are not well-defined on currents. This theorem has been stated rigorously in \cref{thm:SteenrodPowers_Formal}.
\end{remark}

Let us illustrate the map $\cyc$ for $p = 2$, $k = 0$ and $n = 1$. In this case, we have $\cyc : \Z0(\Sp1) \to \Z0(\RP^1 \times \D^2, \RP^1 \times \partial\D^2)$. Suppose that the input to $\cyc$ is a 0-cycle in $\Sp1$ made of the vertices of a square (see \cref{fig:CyclicProductMap}(a)). Then if we consider $\RP^1$ to be the space of lines through the origin in $\RR^2$, the output of $\cyc$ will be $\sum_{1 \leq i < j \leq 4} (\ell_{ij}, \frac12(x_i + x_j))$, where each $x_i$ is treated as a point in $\RR^2$ and $\ell_{ij} = \vspan\{x_i - x_j\}$ (see \cref{fig:CyclicProductMap}(b)--(c)).

\begin{figure}[h]
    \centering
    \begin{tabular}{ccc}
        (a) & (b) & (c) \\
        \includegraphics[height=5cm]{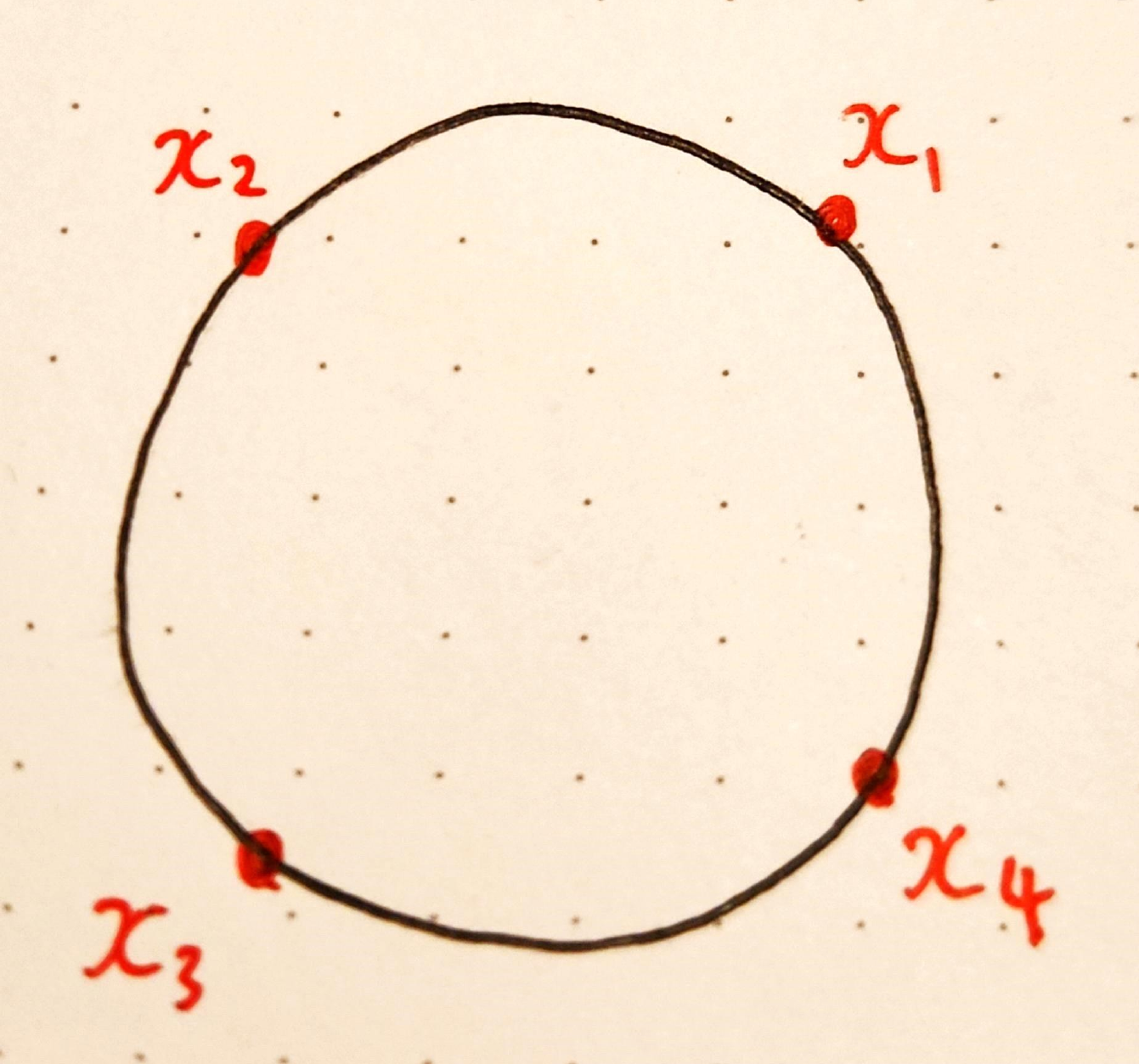} &
        \includegraphics[height=5cm]{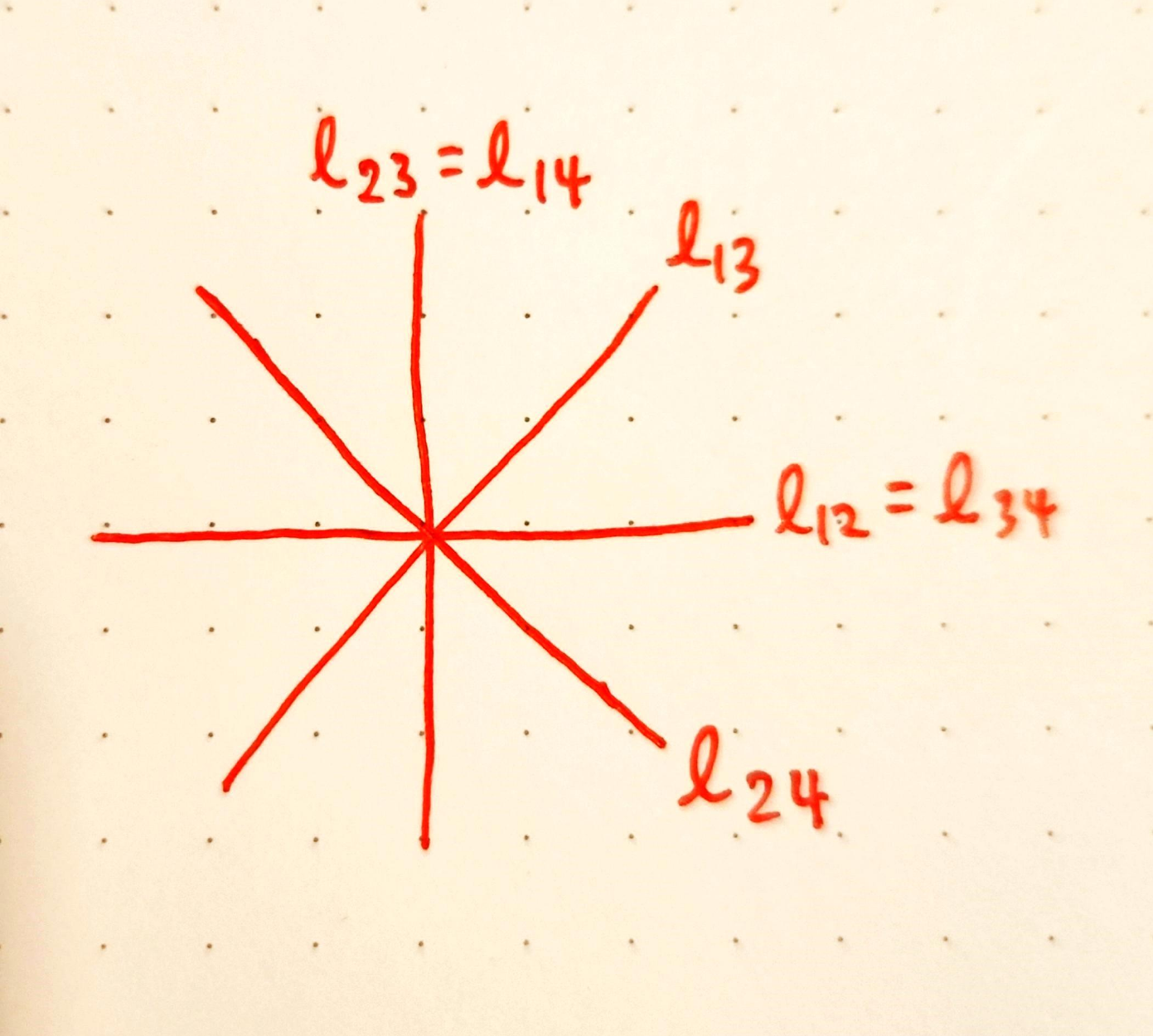} &
        \includegraphics[height=5cm]{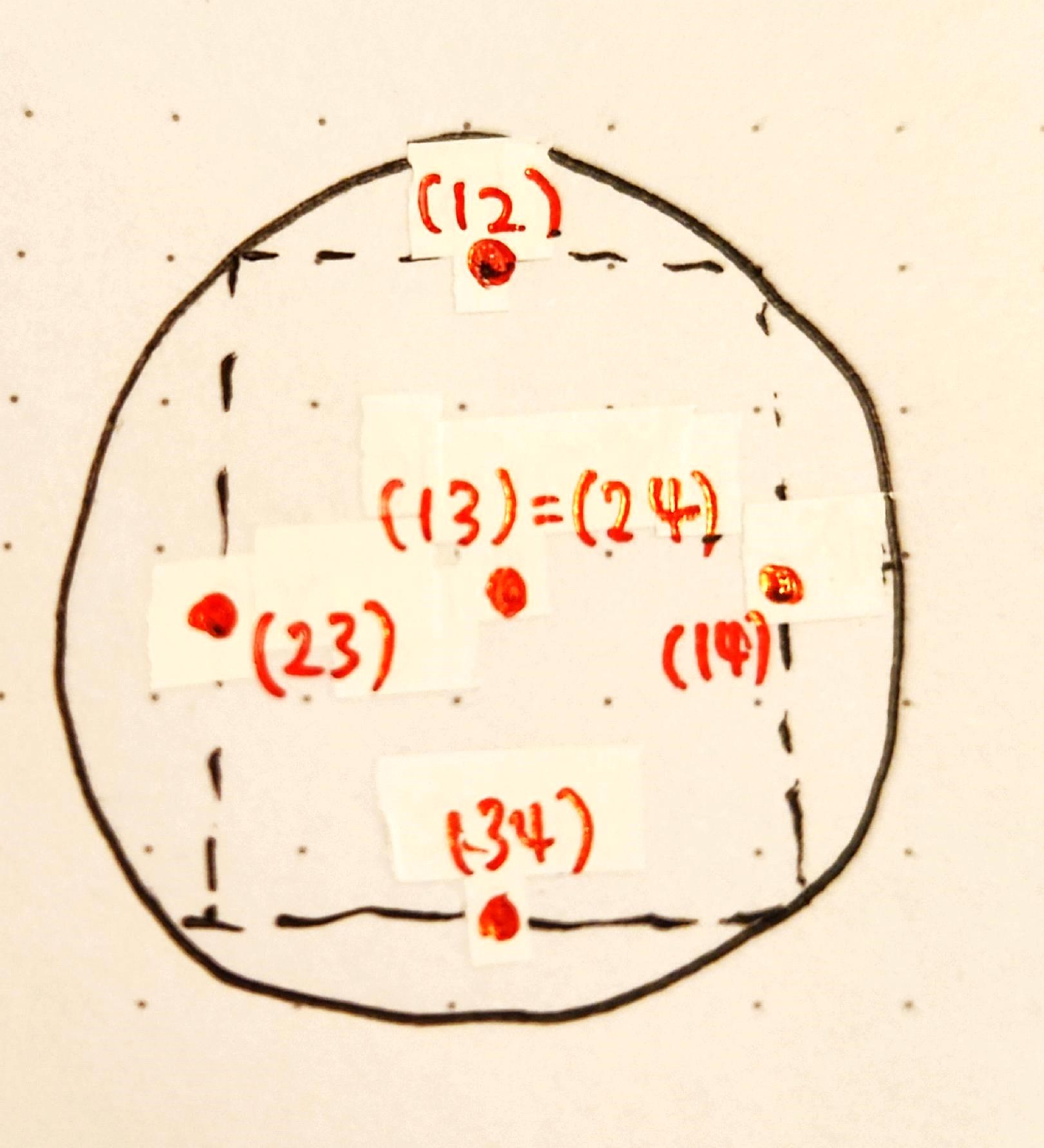}
    \end{tabular}
    
    \caption{A Brown representative $\cyc$ for the total Steenrod power, when $p = 2$, $k = 0$, and $n = 1$. The input 0-cycle $x_1 + \dotsb + x_4$, made of the vertices of a square, is shown as red points in (a). The output $\cyc(x_1 + \dotsb + x_4)$ is a sum of pairs $(\ell_{ij}, \frac12(x_i + x_j))$, where $\ell_{ij} = \vspan\{x_i - x_j\}$ is a line in $\RR^2$ shown in (b), and $\frac12(x_i + x_j)$ is a point in $\D^2$ that is shown in red and labeled as $(ij)$ in (c).}
    \label{fig:CyclicProductMap}
\end{figure}

When $p = 2$, the formulas for the Brown representatives simplify further:

\begin{theorem}
    \label{thm:SteenrodSquares}
    When $p = 2$, there exists a choice for the $g_q$'s in \cref{thm:SteenrodPowers} that gives the following Brown representative $\sq^i$ of $\Sq^i : H^m(-) \to H^{m+i}(-)$:
    \begin{equation}
        \label{eq:SteenrodSquare}
    \begin{aligned}
        &\sq^i : \Z{k}(\Sp{n}) \to \Z{2k + n(k+i-1)}(\D^{(n+1)(k+i)}, \partial\D^{(n+1)(k+i)})
        \\
        &\sq^i(T) = \bigcup_{\substack{
            (x,y) \in (T^2 \setminus \Delta)/\ZZ_2
            \\
            \ell = \vspan\{x - y\} \subset \RR^{n+1}
        }} \left( (\ell^\perp)^{k+i-1} 
        \times 
        \left\{ \frac{x + y}2 \right\} \right) \cap \D^{(n+1)(k+i)}.
    \end{aligned}
    \end{equation}

    In particular, when $T$ is a ``planar cycle,'' i.e. $T = V \cap \Sp{n}$ for some $(k+1)$-dimensional affine subspace $V \subset \RR^{n+1}$, then \cref{eq:SteenrodSquare} simplifies to
    \begin{equation}
        \sq^i(V \cap \Sp{n})
        = \bigcup_{\ell = \text{line through origin parallel to } V} \left( (\ell^\perp)^{k+i-1} \times (\ell^\perp \cap V) \right) \cap \D^{(n+1)(k+i)}.
    \end{equation}
\end{theorem}

Let us illustrate the map $\sq^i$ for planar cycles when $k = 1$, $n = 2$, and $i \geq 0$. Suppose that $V$ is a plane in $\RR^3$ that is parallel to the $xz$-plane and intersects $\Sp2$ in a circle (see \cref{fig:SqPlanarCycle})(a)). Then $\sq^i(V \cap \Sp2)$ is the union of a family of linear subspaces of $\RR^{(n+1)(k+i)}$ (intersected with $\D^{(n+1)(k+i)}$). The family is parametrized by lines $\ell$ through the origin in $\RR^3$ that are parallel to $V$. When $\ell$ is the $z$-axis (see \cref{fig:SqPlanarCycle})(a)), the corresponding linear subspace is the Cartesian product of $k + i - 1$ copies of $\ell^\perp$, namely the $xy$-plane, and a copy of $\ell^\perp \cap V$.

\begin{figure}[h]
    \centering
    \begin{tabular}{cc}
        (a) & (b) \\
        \includegraphics[height=4cm]{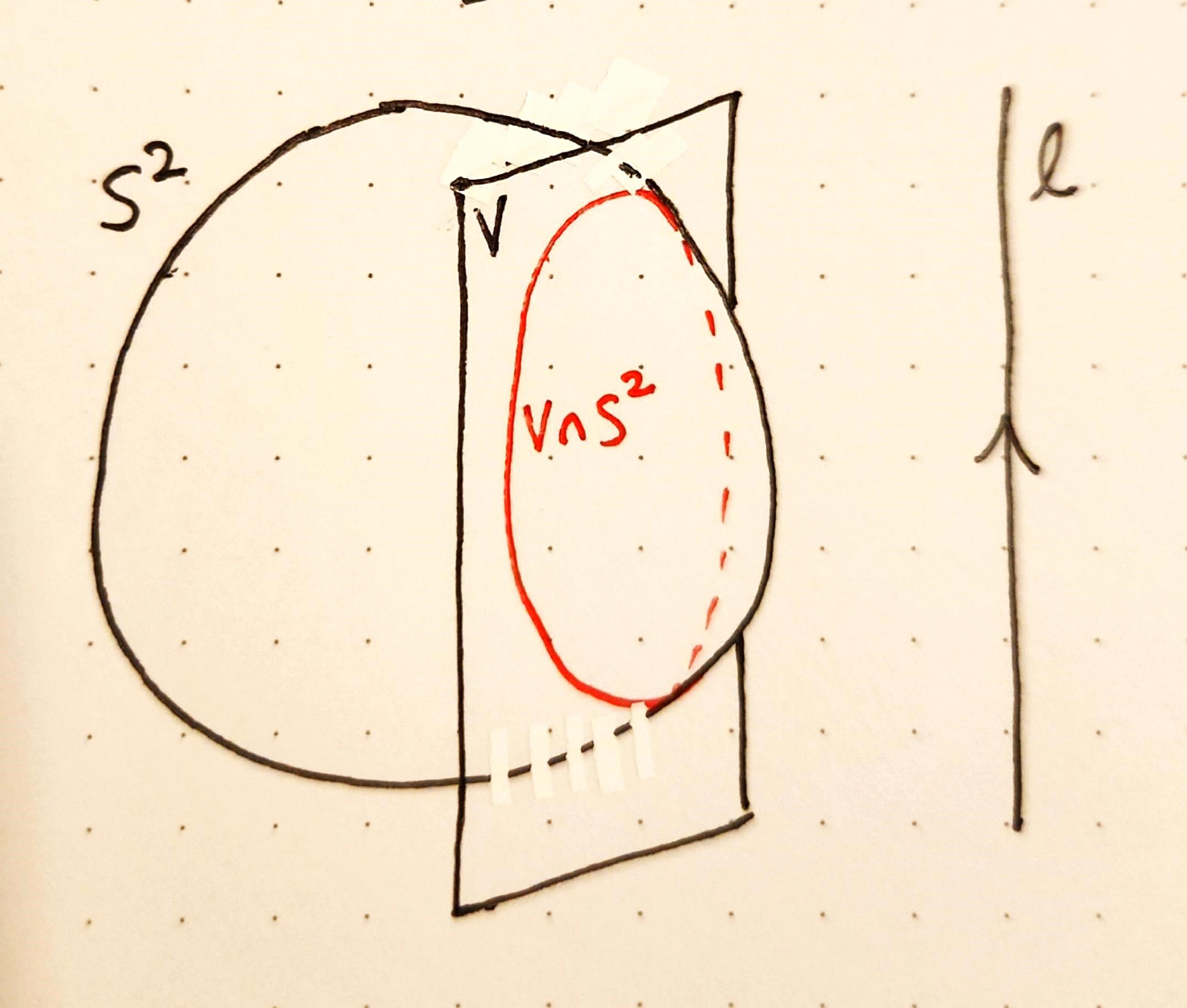} &
        \includegraphics[height=3.5cm]{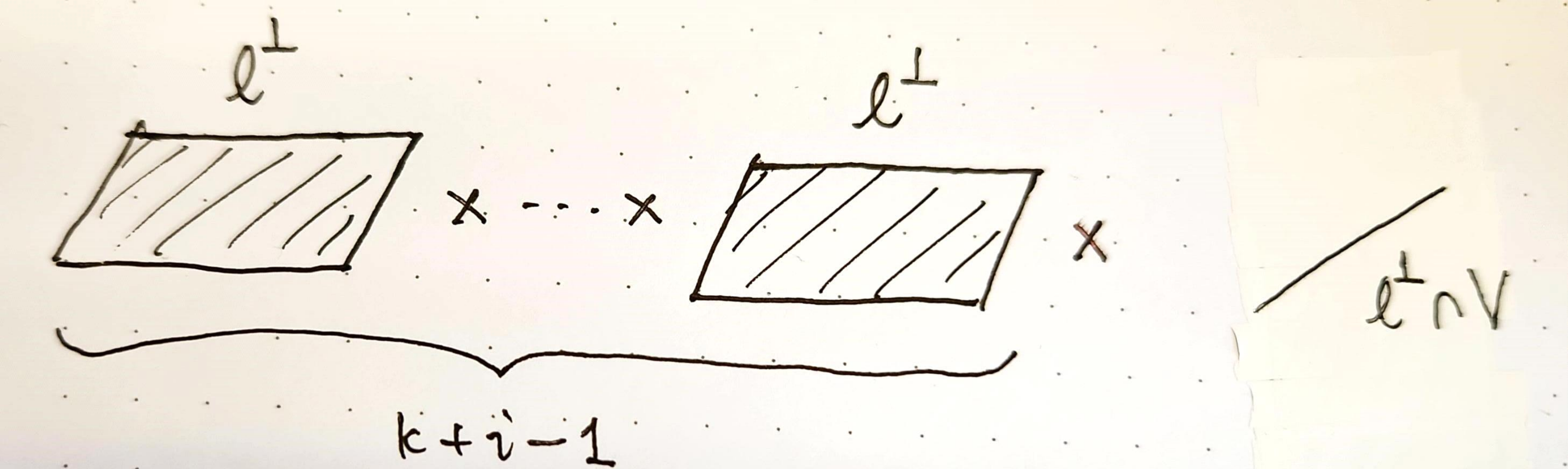}
    \end{tabular}
    
    \caption{A Brown representative $\sq^i$ for $\Sq^i$, when $p = 2$, $k = 1$, $n = 2$, and $i \geq 0$. If the input is a planar cycle $V \cap \Sp2$ as shown in (a), then the output is a union, over lines $\ell \subset \RR^3$ through the origin and parallel to $V$, of a Cartesian product of copies of $\ell^\perp$ and a copy of $V$ (and then intersected with $\D^{(n+1)(k+i)}$), as shown in (b).}
    \label{fig:SqPlanarCycle}
\end{figure}

\subsection{Main technical challenges}
\label{sec:TechnicalChallenges}

Actually, to even state \cref{thm:Bockstein,thm:SteenrodPowers,thm:SteenrodSquares} rigorously, we must overcome many technical difficulties related to Geometric Measure Theory. Overcoming these obstacles will take a large fraction of this work. The formula in \cref{eq:CycProdMap_Intuitive} is intuitively written using familiar operations from point-set topology like set difference and closure, but strictly speaking those operations are not defined for flat cycles. \emph{A priori}, the ``closure'' of the current $h(T^p\setminus \Delta)/\ZZ_p$ may not be well-defined as its mass may blow up near the boundary of the ambient space $L_n \times \D^{n+1}$. We will construct the map $\cyc$ rigorously in \cref{sec:CyclicProductMap} using the Compactness Theorem for integral chains and the Federer-Fleming Deformation Theorem, and prove that the mass does not blow up.

After constructing $\cyc$, we proved that $\cyc$ is continuous. Our proof is related to a question about an equivariant version of the Isoperimetric Inequality: \emph{Suppose that a group $G$ acts on a manifold $X$. If a cycle $T$ in $X$ is $G$-invariant, does it have a filling (that is, a chain bounded by $T$) that is $G$-invariant and whose mass is bounded in terms of $\M(T)$?} In our situation, suppose that we ignore the homeomorphism $h$ in \cref{eq:CycProdMap_Intuitive} for simplicity. Then the fact that the flat metric is related to areas of fillings implies that to prove that $\cyc$ is continuous at zero, we need to show that whenever $T \in \Z{k}(\Sp{n})$ has a filling of small mass $\mu$, the $\ZZ_p$-invariant cycle $T^p \in \Z{pk}(\Sp{p(n+1)-1})$ must have a $\ZZ_p$-invariant filling of mass bounded in terms of $\mu$ and $\M(T)$.

The Brown representatives for the individual Steenrod powers $P^i$ and Bockstein homomorphisms from \cref{thm:SteenrodPowers} are formulated as maps that send a cycle $T$ to the union of a family of cycles $g_q(x) \in \Z{a}(\D^{a+q}, \partial\D^{a+q})$ that is parametrized by points $x$ in a $pk$-dimensional cycle $\cyc(T)$. This is also why unions appear in \cref{thm:SteenrodSquares}. Implicitly, this definition treats the union of a $pk$-dimensional family of $a$-cycles as a $(pk+a)$-cycle. Technically, this does not make sense as currents cannot be treated the same as their supports. Furthermore, even if we treated currents the same as their supports, it is not clear \emph{a priori} that when $g_q$ is an arbitrary continuous family of cycles, the union of the supports of currents in this family is the support of some mod $p$ integral cycle. This is because each $\spt g_q(x)$ may be rectifiable, but $\bigcup_x \spt g_q(x)$ may not be rectifiable. We will define a well-behaved class of families of cycles called \emph{piecewise smooth families} so that this kind of union is rectifiable, and formalize this notion of unions of families of cycles into \emph{gluings of families of cycles that are parametrized by cycles.}

Gluings will serve as a key foundational tool to formalize and prove our main theorems. We will prove that every continuous family of cycles is homotopic to an arbitrarily close family that can be ``glued'' in the aforementioned sense. Our proof adapts the arguments of L.~Guth and Y.~Liokumovich in their approximation theorems which they used to provide an alternative proof of the Almgren Isomorphism Theorem for spaces of cycles with the flat topology \cite{GuthLiokumovich_ParamIneq}. Their approximation theorems have been extended by B.~Staffa in \cite{Staffa_WeylLaw}. The alternative proof by Guth and Liokumovich of the Almgren Isomorphism Theorem also used a discrete version of gluing that is well-defined only at the level of homology: to glue a family $g(x)$ of $k$-cycles parametrized by points $x$ in a $l$-cycle $T$, they choose finitely many points $x_i$ that are finely distributed in $\spt T$, construct a $(k+l)$-cycle that is an ``approximate gluing'' of the finitely many cycles $g(x_i)$, and prove that the homology class of the approximate gluing does not depend on the choice of points $x_i$. However, as we are constructing maps between spaces of cycles, we needed to develop our more general notion of gluing so that families of cycles are glued into cycles that are well-defined as-is, rather than only at the level of homology.

\subsection{Proving \cref{thm:SteenrodPowers} using gluing}

The Steenrod powers are intimately related to the action of $\ZZ_p$ by cyclic permutation on the $p$-fold smash products $X^{\smsh p}$, for cell complexes $X$. Indeed, Steenrod showed how the Steenrod powers arise from \emph{equivariant cohomology}, namely the cohomology groups of a homotopy-theoretic analogue of the quotient $X^{\smsh p}/\ZZ_p$ \cite{Steenrod_CohomOps}. More precisely, $X^{\smsh p}$ is homotopy equivalent to $\Sp\infty \times X^{\smsh p}$, on which there is a free action of $\ZZ_p$. Let $\Sp\infty \times_{\ZZ_p} X^{\smsh p}$ denote the quotient by this action. Then for any $\alpha \in H^m(X)$, there exists a class $\Gamma(\alpha) \in H^{pm}(\Sp\infty \times_{\ZZ_p} X^{\smsh p})$ from which $\beta(\alpha)$ and every $P^i(\alpha)$ can be derived; this is explained in detail in \cref{sec:SteenrodPowersEquivCohom}. 

To prove \cref{thm:SteenrodPowers}, we started by finding Brown representatives for $\Gamma(\iota_m)$, from which we derived Brown representatives for $\beta(\iota_m)$ and $P^i(\iota_m)$. One of our key ideas is that if $X = \Z{k}(\Sp{n})$, where $n - k = m$, then a Brown representative for $\Gamma(\iota_m)$ can be obtained using gluing from a Brown representative for $\Gamma(\alpha)$ for a generator $\alpha \in H^n(\Sp{n})$. More precisely, $\Gamma(\alpha)$ must have some Brown representative $f : \Sp{\infty} \times_{\ZZ_p} (\Sp{n})^{\smsh p} \to \Z0(\Sp{pn})$. Then we can define the following Brown representative for $\Gamma(\iota_m)$:
\begin{equation}
    \label{eq:ExternalOperationByGluing_Intuitive}
\begin{aligned}
    & F : \Sp{\infty} \times_{\ZZ_p} \Z{k}(\Sp{n})^{\smsh p} \to \Z{pk}(\Sp{pn})
    \\
    & F(t,T_1, \dotsc, T_p) = \bigcup_{x \in \{t\} \times (T_1 \smsh \dotsb \smsh T_p)} f(x).
\end{aligned}
\end{equation}
Even though $f$ is not known explicitly, the formula in \cref{eq:ExternalOperationByGluing_Intuitive} makes $F$ compatible with the $\ZZ_p$ action in a way that allows us to relate $F$ to the map $\cyc$.

\subsection{Proving the Almgren Isomorphism Theorem using gluing}

We have also used gluings to provide an alternative proof of \cref{eq:AlmgrenIsoFMetric}, which we call the \emph{Almgren Isomorphism Theorem for mod $p$ cycles with the inductive limit topology}. Our strategy was basically to verify the Eilenberg-Steenrod Axioms for a homology theory, in the spirit of standard proofs of the Dold-Thom Theorem. For instance, to verify the Exactness Axiom, we demonstrated the existence of quasifibrations $\Z{k}(\partial M) \to \Z{k}(M) \to \Z{k}(M, \partial M)$ by using our technical tool of gluing to prove a weak version of the homotopy lifting property. These quasifibrations yield the necessary long exact sequences:
\begin{equation*}
    \dotsb \to \pi_i(\Z{k}(\partial M)) \to \pi_i(\Z{k}(M)) \to \pi_i(\Z{k}(M, \partial M)) \to \dotsb
\end{equation*}
In this way, our proof introduces a homotopy-theoretic perspective to the Almgren Isomorphism Theorem, in contrast to the more direct proofs of the known cases of the theorem which focused more explicitly on how the geometry of the families of cycles relates to the geometry of homology classes.

\subsection{Earlier geometric representations of cohomology classes and operations}

Maps from a simplicial complex $X$ to $\Z{k}(\D^n, \partial\D^n)$ (with the flat topology) that represent a nontrivial cohomology class of $X$ have been studied in other contexts along with other similar objects under the name of \emph{sweepouts} \cite{Almgren_Varifolds,Gromov_Waists,Gromov_FillingRiemnMfds,Guth_Minimax}. The name evokes a family of cycles, parametrized by $X$, that ``sweeps through'' every point in $\Sp{n}$ or $\D^n$, such as in \cref{fig:BrownRepExamples}(a). In particular, given a particular chain of compositions of Steenrod squares $\Sq^{i_1}\dotsm\Sq^{i_q}$ that does not vanish, Guth inductively constructed cell complexes $X$ and maps from $X$ to spaces of cycles that are Brown representatives of a nonzero class $\Sq^{i_1}\dotsm\Sq^{i_q}(\alpha)$ for some $\alpha \in H^*(X)$ \cite{Guth_Minimax}. Sweepouts have been applied in the context of Almgren-Pitts min-max theory to prove the existence of minimal submanifolds, study their geometry, and prove several related conjectures.\footnote{A survey of some early applications of sweepouts and Almgren-Pitts min-max theory is available in \cite{Marques_MinimalSurfacesSurvey}.}

H.~B.~Lawson had constructed Brown representatives of cohomology classes as maps to spaces of \emph{complex algebraic $k$-cycles} in $\CP^n$, which we denote by $\Z[alg]{2k}(\CP^n)$ (the cycles have real dimension $2k$). He proved that $\Z[alg]{2k}(\CP^n) \whe K(\ZZ, 2) \times K(\ZZ,4) \times \dotsb \times K(\ZZ, 2n-2k)$. He also represented the total Chern class of the complex Grassmannian $\Gr_\C(n-k,n+1)$ by the map $f : \Gr_\C(n-k,n+1) \to \Z[alg]{2k}(\CP^n)$ which sends each complex $(n-k)$-plane to its complex orthogonal complement, considered as an algebraic cycle \cite{Lawson_AlgCyclesHomotTheory}. The total Chern class is equal to $f^*(\iota_2 \otimes \iota_4 \otimes \dotsb \otimes \iota_{2n-2k})$, where $\otimes$ denotes the cross product. T.~K.~Lam defined a similar map that represents the total Stiefel-Whitney class of real Grassmannians \cite{Lam_RealAlgCycles}. Our example for $\RP^2$ (\cref{fig:BrownRepExamples}(b)--(c)) can be derived from that map.

The fact that homology classes are geometrically represented by cycles have led several authors to formulate geometric constructions of the \emph{Steenrod homology operations} $\Sq_i : H_{n+i}(-) \to H_n(-)$ on mod 2 homology groups. The $\Sq_i$ are dual to the Steenrod squares, in the sense that evaluating $\Sq^i(\alpha)$ on a mod 2 homology class $\beta$ is the same as evaluating $\alpha$ on $\Sq_i(\beta)$. M.-L.~Michelson considered the infinite symmetric product of an even-dimensional sphere, $\SP(\Sp{2n})$, which is homotopy equivalent to $K(\ZZ,2n)$ by the Dold-Thom theorem. For each $0 \leq k \leq n$ they constructed an explicit integral $2(n+k)$-cycle in $\SP(\Sp{2n})$ whose homology class reduced mod 2, namely $\alpha \in H_{2(n+k)}(\SP(\Sp{2n}))$, satisfies $\Sq_{2k}(\alpha) \neq 0$ \cite{Michelson_AlgCyclesCohomOps}. Analogues of $\Sq_i$ have also been constructed for algebraic varieties using intersection theory and its extensions by several authors (see \cite[p.~377]{Fulton_IntersectionTheory} and \cite{Brosnan_SteenrodOps}). R.~M.~Hardt and C.~G.~McCrory represented $\Sq_i$ for a space $X$ using the double points of certain maps from cycles to Euclidean space \cite{McCrory_GeometricHomOps,HardtMcCrory_SteenrodOps}. P.~F.~dos Santos and P.~Lima-Filho encoded the Steenrod squares in $RO(C_2)$-graded equivariant cohomology using a restriction map to the fixed points of a $\ZZ_2$-action \cite{dosSantosLimaFilho_Sq}.

We are not aware of any prior research on geometric representations of Steenrod powers and Bockstein homomorphisms for odd primes $p$.

\subsection{Potential applications to quantitative homotopy theory}

One potential application of our Brown representatives arises from how the Steenrod powers and Bockstein homomorphisms serve as the starting ingredients of the mod $p$ Adams spectral sequence, which computes the $p$-primary part of the stable homotopy groups of spaces \cite{Adams_SpectralSequence}.\footnote{An exposition is available in \cite{FomenkoFuchs_HomotopicalTopology}.} This connection suggests that our constructions could allow lead to quantitative results about the homotopy groups of spheres. For instance, it is natural to ask whether elements of $\pi_m(\Sp{n})$ can be represented by maps between spheres that are ``efficient'' or have a ``simple'' geometry. Let us give each sphere the unit round metric, upon we could quantify the ``efficiency'' or ``complexity'' of a map $f : \Sp{m} \to \Sp{n}$ by its \emph{$k$-dilation}, which is defined as the infimal value of $C$ such that $\vol f(\Sigma) \leq C \vol\Sigma$ for all $k$-dimensional submanifolds $\Sigma \subset \Sp{m}$. Note that 1-dilation is simply the Lipschitz constant.

L.~Guth analyzed the geometry of the Steenrod squares in a different way from us to prove a lower bound on the $k$-dilation of $f$ when the cell complex obtained by attaching an $(m+1)$-cell to an $n$-cell via $f$ has a nonzero Steenrod square and $k \leq m/2$ \cite{Guth_Dilation}. Due to results about the Hopf invariant \cite{Adams_HopfInvariant}, the condition in Guth's result is satisfied only when $m = n + 1$, or when $n$ is sufficiently large, $m = n + 3$ or $m = n + 7$. This leaves open nearly every other homotopy group of a sphere. Our constructions could potentially help us prove versions of this result for other homotopy groups of spheres and for general Steenrod powers for odd primes.

On the flipside, torsion-free elements of $\pi_m(\Sp{n})$ are known to have representatives with controlled Lipschitz constant, while almost none of the torsion classes have a known representative with an explicitly computed Lipschitz constant. For each $m > n$, a result of Serre implies that $\pi_m(\Sp{n})$ is generated by finitely many torsion classes and at most one torsion-free class. The torsion-free generator $\alpha$ can be represented by a Hopf fibration or a Whitehead product that is 100-Lipschitz \cite[p.~34]{MosherTangora_CohomOps}. Evidently, it remains to study the torsion classes.

Let us define $L(m,n)$ as the infimal value of a constant $L$ such that every torsion class in $\pi_m(\Sp{n})$ has a $L$-Lipschitz representative. How fast does $L(m,n)$ grow as a function of $m$ and $n$? When $\pi_m(\Sp{n}) \neq 0$, $L(m,n) \geq 1$, as any surjective map between unit spheres must be at least 1-Lipschitz. Another lower bound is suggested by the proof by M.~Gromov that the number of classes in $\pi_m(\Sp{n})$ that have $L$-Lipschitz representatives is at most a polynomial in $L$ \cite[p.~305]{Gromov_Dilatation}. An argument of Gromov suggests an upper bound on $L(m,n)$ by a tower of exponentials involving $m$ and $n$ \cite{Gromov_TowerExp}. However, our geometric construction of the Steenrod powers and Bockstein homomorphisms could potentially give rise to a quantitative version of the Adams spectral sequence, which may improve this upper bound on the growth of $L(m,n)$.

\subsection{Organization of Content}

In \cref{sec:Definitions} we define the key objects from Geometric Measure Theory that we will use, such as mod $p$ integral currents and cycles, varifolds, and the flat metric.

In \cref{sec:PiecewiseSmooth} we will define \emph{piecewise smooth families of cycles}, which are families of cycles for which gluings exist. We will prove in \cref{thm:PiecewiseSmoothApprox} that continuous families of cycles can be approximated by piecewise smooth families. In \cref{sec:Gluings} we will prove that piecewise smooth families of cycles have gluings (\cref{thm:PiecewiseSmoothGluing_Existence}) and use that to prove the Almgren Isomorphism Theorem for the inductive limit topology (\cref{Thm:AlmgrenIsoFMetric}) and the Brown Representability Theorem for spaces of cycles (\cref{thm:BrownRepSpacesOfCycles}).

In \cref{sec:BrownRepCohomClassesOps} we construct explicit Brown representatives of cohomology generators of projective and lens spaces, and Brown representatives for the cross product (\cref{prop:CrossProduct_BrownRep}) and cup product (\cref{prop:CrossProduct_BrownRep}). We also use families of cycles to give geometric representations of a mod $p$ analogue of fiber integration (\cref{prop:GeomRep_PushforwardCohom}) and the K\"unneth formula (\cref{cor:KunnethFormula}). 

In \cref{sec:CyclicProductMap} we construct the cyclic product map $\cyc$ rigorously and prove that it is continuous. We also use that to prove \cref{thm:Bockstein}. In \cref{sec:SteenrodPowers} we give an exposition of Steenrod's construction of the Steenrod powers via equivariant cohomology, then combine that with gluings to prove \cref{thm:SteenrodPowers,thm:SteenrodSquares}.

The appendices contains proofs of results that are technical but straightforward, or proofs that are minor adaptations of ideas from the existing literature. The statements of these results may appear in earlier sections but will be labeled using the Roman alphabet, such as \cref{lem:DeltaAdmissibleDoubling}. \Cref{sec:Appendix_GMT} contains proofs of lemmas in Geometric Measure Theory, and \cref{sec:Appendix_TopLemmasLensSpaces} contains proofs of lemmas about the topology of lens spaces.

\section{Definitions}
\label{sec:Definitions}

\subsection{Mod $p$ relative integral currents and the flat metric}

Fix a prime $p$ and some integers $m \geq k \geq 0$. Let $M$ be a connected and compact $m$-dimensional Riemannian manifold, possibly with boundary. Let $\Haus^k$ denote the $k$-dimensional Hausdorff measure on $M$. We say that a mod $p$ singular $k$-chain $T = \sum_i a_i \sigma_i$, where $a_i \in \ZZ_p$, is \emph{Lipschitz} if each $\sigma_i : \Delta^k \to M$ is Lipschitz. We also say that $T$ is \emph{non-overlapping} if it is ``one-to-one almost everywhere'' in the sense that $\Haus^k(\{y \in M : \# \bigcup_i\sigma_i^{-1}(y) > 1\}) = 0$. We may subdivide $T$ using a triangulation of $\Delta^k$ into $k$-simplices $\tau_j$ which admit positively oriented bilipschitz homeomorphisms $f_j : \Delta^k \to \tau_j$, and replacing each $\sigma_i$ with $\sum_j(\sigma_i \circ f_j)$. Let $I_k(M)$ denote the set of non-overlapping mod $p$ Lipschitz $k$-chains in $M$, except that two such chains are identified if one of them is a subdivision of the other.

The \emph{mass} of some $T = \sum_i a_i \sigma_i \in I_k(M)$ is defined as $\M(T) = \sum_i \abs{a_i} \Haus^k(\im\sigma_i)$, where if $a_i$ is represented by an integer from 0 to $p-1$, then $\abs{a_i} = \min\{a_i, p - a_i\}$. Note that we can compute the sums and boundaries of elements of $I_k(M)$ as singular chains, but the resulting chains may not be non-overlapping.

Let $N$ be a compact submanifold of $M$. Consider the equivalence relation $\sim$ on $I_k(M)$ where $S \sim T$ if their difference as singular chains can be written as a linear combination of singular simplices with images in $N$. Note that since we have identified each $T \in I_k(M)$ with its subdivisions, there exists some $T' \sim T$ so that the simplices of $T'$ with images in $N$ have arbitrarily small total volume. We say that some $\sum_i a_i\sigma_i \in I_k(M)$ is \emph{non-overlapping with respect to $N$} if $\Haus^k(\bigcup_i\im\sigma_i \cap \partial(M \setminus N)) = 0$. Now we define the set of \emph{non-overlapping mod $p$ relative $k$-chains} $I_k(M,N)$ to be the subset of chains in $I_k(M)$ that are non-overlapping with respect to $N$, identified according to $\sim$. For each $T \in I_k(M,N)$ we define its mass to be $\M(T) = \inf_{\tilde{T} \in T} \M(\tilde{T})$. Given any $S, T \in I_k(M,N)$, define
\begin{equation*}
    \Fl(S,T) = \inf \{\M(P) + \M(Q) ~:~ P \in I_{k+1}(M), Q \in I_k(M), \partial P + Q \in S - T\}.
\end{equation*}

$\Fl$ is a pseudometric on $I_k(M,N)$. Let $I_k(M,N)/\Fl^{-1}(0)$ denote the induced metric space obtained by identifying elements at zero $\Fl$ distance from each other. The completion of this metric space is the \emph{space of mod $p$ relative flat $k$-chains} $F_k(M,N)$. It can be shown that $\M$ is well-defined on $I_k(M,N)/\Fl^{-1}(0)$ (see \cref{lem:ZeroFlatDistSameCurrent}), and that it can be extended to $F_k(M,N)$ using the formula
\begin{equation}
    \label{eq:MassLowerSemiCtsExtension}
    \M(T) = \liminf_{\varepsilon \to 0} \{ \M(S) : S \in I_k(M,N)/\Fl^{-1}(0), \Fl(S,T) < \varepsilon \}.
\end{equation}
This allows us to define the space of \emph{mod $p$ relative integral $k$-currents with the flat metric}, $\I{k}(M,N; \Fl) = \{T \in F_k(M,N) : \M(T) < \infty\}$. 

\begin{lemma}
    $\I{k}(M,N; \Fl)$ is a group and the boundary map on non-overlapping mod $p$ relative chains extends to a continuous map $\partial : \I{k+1}(M,N; \Fl) \to \I{k}(M,N; \Fl)$ for each $k \geq 0$.
\end{lemma}

\subsection{The inductive limit topology}

Henceforth we will write $\I{k}(M, N)$ and $\Z{k}(M, N)$ to mean $\I{k}(M, N; \Fl)$ and $\Z{k}(M, N; \Fl)$ but with the \emph{inductive limit topology} \cite[(1.9)]{Almgren_HomotopyGroupsIntegralCycles} which is the inductive limit of the mass-bounded subspaces of $\I{k}(M, N; \Fl)$ and $\Z{k}(M, N; \Fl)$.

Let us give sufficient conditions for maps between spaces of currents or cycles with the inductive limit topology to be continuous. Let $\I{k}(M,N)^\mu$ and $\Z{k}(M,N)^\mu$ denote the sets $\{T \in \I{k}(M,N) : \M(T) < \mu\}$ and $\{T \in \Z{k}(M,N) : \M(T) < \mu\}$ respectively. We will only consider finite $\mu$. Observe that the inductive limit topology on $\I{k}(M,N)^\mu$ is exactly the same as the flat topology, and similarly for $\Z{k}(M,N)^\mu$.

\begin{restatable*}{lemma}{LemContinuityIndLimTop}
    \label{lem:ContinuityIndLimTop}
    A map $g : \Z{k}(M,N) \to Y$ is continuous in the inductive limit topology if and only if $g$ is continuous in the flat topology when restricted to $\Z{k}(M,N)^\mu$ for every $\mu > 0$.

    A map $f : \Z{k}(M,N) \to \Z{k'}(M',N')$ is continuous with respect to the inductive limit topologies on its domain and codomain if for all $\mu > 0$, there exists a $\mu' > 0$ such that $\M(T) < \mu \implies \M(f(T)) < \mu'$, and $f$ restricts to a map $f_\mu : \Z{k}(M,N)^\mu \to \Z{k'}(M',N')^{\mu'}$ that is continuous which respect to the flat topologies on the domain and codomain.

    The above statements are still true when $\Z{}$ is replaced by $\I{}$.
\end{restatable*}

\subsection{Concentration of mass}

When applying local modifications to every cycle in a large family of cycles, we may wish to modify the part of each cycle in some ball of radius $r$ that is independent of the particular cycle. Ideally, $r$ can be chosen so that the pieces of those cycles inside balls of radius $r$ must have mass as small as desired. This is known as the \emph{no concentration of mass condition}, and it does not hold in general. We introduce a \emph{mass concentration profile} to quantify how small $r$ must be chosen to guarantee that the masses of the parts of cycles inside $r$-balls must be smaller than some given $\mu > 0$. Roughly speaking, we will use this profile in \cref{sec:PiecewiseSmooth} to ensure that our local modifications to families of cycles preserve the no concentration of mass condition.

\begin{definition}[Mass concentration profile]
    Given any map $F : X \to \Z{k}(M,N)$, let its \emph{mass concentration profile} be the function $\chi_F : \RR \to \RR$ where $\chi_F(r)$ is the supremum of all $\M(F(x) \res B)$, where $x$ ranges over $X$ and $B$ ranges over all balls of radius $r$ in $M$.

    We say that a map $f : X \to \Z{k}(M, N)$ \emph{has no concentration of mass} if $\lim_{r \to 0} \chi_f(r) = 0$.
\end{definition}

\section{Piecewise Smooth Families of Cycles}
\label{sec:PiecewiseSmooth}

Suppose that $M$ is a closed Riemannian manifold, and let $X$ be a $n$-dimensional manifold that is oriented when $p$ is odd. Consider a continuous map $f : X \to \Z{d}(M)$. When $X = \Sp{n}$, the Almgren isomorphism is defined using a ``gluing homomorphism'' that ``glues'' the $m$-parameter family $f$ of $d$-cycles into an $(n+d)$-cycle $A$. One difficulty in defining this gluing homomorphism is evident from considering each $f(x)$ as an integral current mod $p$, which has an \emph{orientation}: a field $v(x)$ of $d$-vectors defined almost everywhere on the support of $f(x)$, which intuitively span the tangent planes. If we view $f(x)$ as a ``submanifold'' of $A$, then $f(x)$ should have a ``normal bundle'' within $A$. Thus we would like to find some $n$-vector field $n(x)$ with the goal that so that $v(x) \wedge n(x)$ is an orientation of $A$. However, it is not clear \emph{a priori} how to find $n(x)$, especially if the cycles $f(y)$ close to $f(x)$ approach $f(x)$ in an irregular way as $y \to x$. The standard constructions of the gluing homomorphism avoid constructing $n(x)$ directly by approximating $f$ with a discrete family of cycles, and filling between nearby cycles in this discrete family \cite{GuthLiokumovich_ParamIneq,Guth_WidthVolume}. 

Throughout the rest of this section, let $M$ be a compact $m$-dimensional manifold, and $N$ be either $\partial M$ or a compact $m$-dimensional submanifold. We call such a pair $(M, N)$ an \emph{$m$-dimensional collar pair.}

\begin{definition}
    \label{def:PiecewiseSmooth}
    Consider a map $f : X \to \Z{k}(M,N)$, where $X$ is a compact smooth manifold with a smooth polyhedral structure. Then we say that a \emph{piecewise smooth structure for $f$} is a triplet of families $(\{\varphi_\sigma\}, \{Z_\sigma\}, \{\varphi_{\tau\sigma}\})$, where the first two families are indexed over the cells $\sigma$ of $X$ and the third is indexed over pairs of cells $\tau \subset \sigma$. Each $\varphi_\sigma$ is a smooth map $\sigma \times M \to M$ called a \emph{chart} and each $Z_\sigma \in \Z{k}(M,N)$ is called a \emph{model relative cycle} such that
    \begin{align}
        \label{eq:PiecewiseSmooth_Chart}
        \varphi_\sigma(\sigma \times N) &\subset N & &\text{and} & \varphi_{\sigma\sharp}(\{x\} \times Z_\sigma) &= f(x) \text{ for all } x \in \sigma.
    \end{align}
    We further require that for all $x$ in the interior of $\sigma$, $\varphi_\sigma(x,-)$ is a diffeomorphism $M \to M$ that restricts to a diffeomorphism $N \to N$.
    
    Each $\varphi_{\tau\sigma}$ is a smooth map $\tau \times M \to M$ called a  \emph{collapse map} that satisfies the following properties:
    \begin{itemize}
        \item $\varphi_{\tau\sigma}(x,-)_\sharp(Z_\sigma) = Z_\tau$ for any $x \in \tau$.
        
        \item $\varphi_\sigma(x,-) = \varphi_\tau(x,-) \circ \varphi_{\tau\sigma}(x,-)$ for all $x \in \tau$. In other words,
        \begin{equation}
            \label{eq:CollapseMapIdentity}
            \varphi_\sigma|_{\tau \times M} = \varphi_\tau \circ \big( (x,y) \mapsto (x, \varphi_{\tau\sigma}(x,y)) \big).
        \end{equation}

        \item $\varphi_{\sigma\sigma}(x,-) = \id$ for all $x \in \sigma$ and for all $\kappa \subset \tau \subset \sigma$ and $x \in \kappa$, $\varphi_{\kappa\sigma}(x,-) = \varphi_{\kappa\tau}(x,-) \circ \varphi_{\tau\sigma}(x,-)$.
    \end{itemize}

    When a map has a piecewise smooth structure, we say that the map is \emph{piecewise smooth.}
\end{definition}

Piecewise smooth maps offer a natural way to define families of relative cycles that are continuous in the inductive limit topology, as shown in the following lemma.

\begin{lemma}
    \label{lem:PiecewiseSmoothLipschitz}
    When $X$ is a compact Riemannian manifold, every piecewise smooth map $f : X \to \Z{d}(M,N)$ is Lipschitz with respect to the geodesic distance in the domain and the inductive limit topology in the codomain.
\end{lemma}
\begin{proof}
    Denote the injectivity radius of $X$ by $r$. Fix any $x \in X$. For almost every $y \in X$ within distance $r$ of $x$, the unique minimizing geodesic $\gamma$ from $x$ to $y$ is transverse to every cell of $X$, except possibly at $x$. Thus $\gamma$ can be broken into segments, each of which lies in some top-dimensional cell. To prove the lemma, it suffices to prove that $f$ is $(C+L)$-Lipschitz when restricted to each segment, where $L = \max\{1,\max_\sigma \Lip(\varphi_\sigma)\}$ and $C = \max_\sigma \Lip(\varphi_\sigma)\M(Z_\sigma)$. In other words, we may reduce to the case where $\gamma$ lies within a single top-dimensional cell $\sigma$.

    First we prove that $\Fl(f(x), f(y)) \leq C\length(\gamma)$, by ``gluing'' the family of cycles $f \circ \gamma$ into a filling $Q$ of $f(y) - f(x)$. More formally, if we regard $\gamma$ as a Lipschitz 1-chain (with coefficients in $G$) in the obvious way, then we may define $Q = (\varphi_\sigma)_\sharp(\gamma \times Z_\sigma)$, because
    \begin{equation*}
        \partial Q = (\varphi_\sigma)_\sharp(\partial\gamma \times Z_\sigma) = (\varphi_\sigma)_\sharp((y - x) \times Z_\sharp) = \varphi_\sharp(\{y\} \times Z_\sharp) - \varphi_\sharp(\{x\} \times Z_\sharp) = f(y) - f(x).
    \end{equation*}
    Therefore $\Fl(f(x), f(y)) \leq \M(Q) \leq C\length(\gamma)$.

    Since $\M(f(x)) \leq C$ for all $x \in X$, this implies that $f$ is continuous in the inductive limit topology.
\end{proof}

Here are several examples of piecewise smooth families of relative cycles. Any two cycles in these families are ambient isotopic to each other, and this observation allows us to show that the families are piecewise smooth, continuous in the $\F$ topology when $k \geq 1$, and continuous in the inductive limit topology by \cref{lem:PiecewiseSmoothLipschitz}.

\begin{example}
    \label{eg:FiberBundle_PiecewiseSmooth}
    Consider an oriented fiber bundle $F \to E \xrightarrow{\xi} B$, where $E$ is a closed and oriented Riemannian manifold, $B$ is a polyhedral complex, and the fibers are $k$-dimensional submanifolds with boundary inside $\partial E$. Then the map $\xi^{-1}(-) : B \to \Z{k}(E)$ is piecewise smooth.

    Find a triangulation of $B$ so that each simplex $\sigma$ lies within a trivializing neighbourhood $U$. Let the trivializing homeomorphism be $h : \xi^{-1}(U) \to U \times F$. Define the model relative cycle $Z_\sigma = \xi^{-1}(v)$, where $v$ is the barycenter of $\sigma$. To define the charts, construct a family of orientation-preserving diffeomorphisms, $f : \sigma \times \cl{U} \to \cl{U}$, that fix $\partial U$, so that $f(x,v) = x$. Then define the chart $\varphi_\sigma$ so that each $\varphi_\sigma(x,-) : E \to E$ is the smooth map that is the identity outside of $\xi^{-1}(U)$, and within $\xi^{-1}(U)$ it corresponds, under the trivializing homeomorphism $h$, so the smooth map $(y,z) \mapsto (f(x,y),z)$.

    For any $\tau \subset \sigma$, define the corresponding collapse map by the formula $\varphi_{\tau\sigma}(x,-) = \varphi_\sigma(x,-) \circ \varphi_\tau(x,-)^{-1}$.
\end{example}

\begin{example}
    \label{eg:RPn_PiecewiseSmooth}
    Let $p = 2$ and $n \geq 1$. Then the map $f : \RP^n \to \Z{n}(\D^{n+1}, \partial\D^{n+1})$ that sends $\ell \subset \RR^{n+1}$ to $\ell^\perp \res \D^{n+1}$ is piecewise smooth.

    Consider the fiber bundle $\xi : \SO(n+1) \to \RP^n$ that sends each matrix to the span of its first column. Choose any smooth simplicial structure $X$ on $\RP^n$. Let $\Pi = 0 \times \RR^n \subset \RR^{n+1}$. All of the model currents are $\Pi \cap \D^{n+1}$.
    
    Consider the fiber bundle $\SO(n+1) \to \RP^n$ that sends each matrix to the span of its first column. Since each simplex $\sigma$ is contractible, the inclusion $\sigma \hookrightarrow \RP^n$ lifts using parallel transport to a smooth map $i_\sigma : \sigma \to \SO(n+1)$. Now define the corresponding chart so that $\varphi_\sigma(x,-)$ is the linear transformation $i_\sigma(x)$. For any simplex $\tau \subset \sigma$, define the collapse maps so that $\varphi_{\tau\sigma}(x,-) = \varphi_\tau(x,-)^{-1} \circ \varphi_\sigma(x,-)$.
\end{example}

Let $\perp_{\C}$ denote the orthogonal complement of a complex vector subspace in some $\C^n$ with respect to the standard Hermitian inner product. Endow each complex vector space with its standard orientation as a complex manifold.

\begin{example}
    \label{eg:CPn_PiecewiseSmooth}
    Let $p > 2$ and $n \geq 1$. Then the map $f : \CP^n \to \Z{2n}(\D^{2n+2}, \partial\D^{2n+2})$ that sends a complex line $\ell \subset \C^{n+1}$ to $\ell^{\perp_\C} \res \D^{2n+2}$ is piecewise smooth. (We have $\D^{2n+2} \subset \C^{n+1}$ by the standard identification of $\C^{n+1}$ with Euclidean space.)

    The details of this example follow those of \cref{eg:RPn_PiecewiseSmooth}, except with $\SU(n+1)$ instead of $\SO(n+1)$.
\end{example}

\begin{example}
    \label{eg:Lens_PiecewiseSmooth}
    For any $n \geq 1$, the map $\Lens_p^{2n-1} \to \Z{2n-1}(\D^{2n}, \partial\D^{2n})$ that sends each equivalence class of points
    \begin{equation*}
        \{x_1,\dotsc,x_p\} \mapsto \sum_{j = 1}^p \{z + \lambda x_j : z \perp_\C x_1, \lambda \in [0,\infty)\} \res \D^{2n}.
    \end{equation*}
    Each summand on the right-hand-side is contained in the direct sum of $\vspan_\C\{x_1\}^{\perp_\C}$ (treated as a real vector space) with $\vspan_\RR\{x_j\}$, thus we may give the summand the direct sum orientation. The right-hand-side is a mod $p$ relative cycle: its boundary is $p \cdot \vspan_\C\{x_1\}^{\perp_\C} = 0$.

    The details of this example are a straightforward adaptation of those of \cref{eg:CPn_PiecewiseSmooth}.
\end{example}

\subsection{The Piecewise Smooth Approximation Theorem}

The main result in this section is that every family of relative cycles that is continuous in the inductive limit topology is homotopic to a piecewise smooth approximation. We will prove this by adapting the proof of similar approximation theorems by Guth and Liokumovich \cite{GuthLiokumovich_ParamIneq} and Staffa \cite{Staffa_WeylLaw}. Their methods approximate families $\{f(x)\}_{x \in X}$ of relative cycles that are continuous in the flat metric by ``nicer'' families where if $x$ and $y$ are ``close'' in $X$, then $f(x) - f(y)$ is supported in a small ball. To state our theorem formally, we will need to formalize the above notions, following Guth and Liokumovich.

Like \cite{GuthLiokumovich_ParamIneq}, we will use covers of $\bar{M} = \cl{M \setminus N}$ by \emph{$\delta$-admissible sets} which are either open balls in the interior or ``thickenings'' of open balls in $\bar{M}$. To define those thickenings, \cite[Lemma~2.1]{GuthLiokumovich_ParamIneq} gives a parametrization of a collar neighbourhood of $\bar{M}$ by a map $E : \partial \bar{M} \times [0,r_0] \bar{M}$ that is close to 1-Lipschitz, where $r_0$ is less than the injectivity radius of $\bar{M}$. When $\partial\bar{M}$ is smooth, $E$ is defined using the exponential map normal to $\partial \bar{M}$.

\begin{definition}[$\delta$-admissible sets, {\cite[Section~2.4]{GuthLiokumovich_ParamIneq}}]
    For any $0 < r < r_0$, a \emph{generalized ball of radius $r$} in $\bar{M}$ is either an open ball of radius $r$ in the interior of $\bar{M}$ or $E(B_r \times [0,r])$ where $B_r$ is an open ball of radius $r$ in $\partial\bar{M}$. A collection $\mathcal{U}$ of generalized balls in $\bar{M} = \cl{M \setminus N}$ is called \emph{$\delta$-admissible} for some $\delta > 0$ if they are all disjoint and their radii add up to less than $\delta$.
\end{definition}

\begin{definition}[Localized maps]
    Consider a cubical complex $X$, and a map $F : V \to \Z{k}(M, N)$ where $V$ is a subcomplex of $X$. We say that $F$ is \emph{$\varepsilon$-fine} if for any $x, y \in V$ that lie in the same cell of $X$, either $\F(F(x),F(y)) \leq \varepsilon$ (when $k \geq 1$) or $\Fl(F(x),F(y)) < \varepsilon$ (when $k = 0$).

    We say that $F$ is \emph{$\delta$-localized} if every cell $C$ of $X$ is associated with a $\delta$-admissible collection $\mathcal{U}_C$ of sets in $\bar{M} = \cl{M \setminus N}$ such that for all $x,y \in C \cap V$, $F(x) - F(y)$ is supported inside $\bigcup\mathcal{U}_C$. Moreover, we say that the family $\{\mathcal{U}_C\}_C$  is \emph{doubling} if it satisfies the following properties:
    \begin{itemize}
        \item For every cell $C$, the elements of $\mathcal{U}_C$ are generalized balls that, even when their radii are doubled, are disjoint from one another.

        \item For all cells $C \subset C'$ and every $B \in \mathcal{U}_C$, $2B$ is contained inside some $U' \in \mathcal{U}_{C'}$.
    \end{itemize}
\end{definition}

Now we can state our approximation theorem:

\begin{theorem}[Piecewise Smooth Approximation Theorem]
    \label{thm:PiecewiseSmoothApprox}
    Let $f : X \to \Z{k}(M,N)$ be a continuous from a finite cubical complex. Then for any $\delta > 0$, $f$ is homotopic to a piecewise smooth map $f'$ that is also $\delta$-localized.
    
    Moreover, there exist some $\delta_0, \varepsilon_0 > 0$ such that if $X$ has a subcomplex $Y$ such that $f|_Y$ is piecewise smooth with respect to a refinement of $Y$, $f|_Y$ is $\delta$-localized for some $\delta < \delta_0$, and $f$ is $\varepsilon$-fine for some $\varepsilon < \varepsilon_0$, then $f'$ can be chosen so that $f'|_Y = f|_Y$ and the piecewise smooth structure of $f'$, restricted to $Y$, is a refinement of that of $f|_Y$.
\end{theorem}

Let us first prove \cref{thm:PiecewiseSmoothApprox} for the case $k = 0$.

\begin{proof}[Proof of \cref{thm:PiecewiseSmoothApprox} when $k = 0$]
    We will explain the proof for the case where $Y = \emptyset$, as the case of $Y \neq \emptyset$ follows similarly. By \cite[p.~262]{Almgren_HomotopyGroupsIntegralCycles}, the cycles $f(x)$ all have mass bounded by some constant $\mu$. Thus by \cref{lem:ContinuityIndLimTop}, $f$ is continuous in the flat topology. From the proof of \cite[Theorem 2.3(2)]{GuthLiokumovich_ParamIneq} we know that $f$ is homotopic to a map $F$ that is $\delta$-localized, where $\delta$ can be chosen as small as desired. The homotopy is constructed using only cycles that are all bounded by some constant mass (see \cite[Proposition~2.6(3)]{GuthLiokumovich_ParamIneq}), so in fact $F$ and the homotopy are all continuous in the inductive limit topology. $F : \tilde{X} \to \Z{k}(M,N)$ is defined using a refinement $\tilde{X}$ of $X$: first the map $F|_{\tilde{X}^0}$ is constructed in a way that is $\delta$-localized for some small $\delta < \inj(\bar{M})$, then it is extended into the whole of $\tilde{X}$.

    It remains to construct a piecewise smooth structure for $F$. We will construct this inductively by following the inductive extension of $F|_{\tilde{X}^0}$ over the higher and higher skeleta of $\tilde{X}$, so let us explain how this extension is done. Suppose that $F$ has been extended into a $\delta$-localized map on $\tilde{X}^d$. Consider any $(d+1)$-cell $C$. Since $F|_{X^d}$ is $\delta$-localized, $C$ corresponds to a $\delta$-admissible family $\mathcal{U}_C$ such that for all $x, y \in \partial C$, $F(x) - F(y)$ is supported in $\bigcup\mathcal{U}_C$. Since $\mathcal{U}_C$ consists of disjoint generalized balls, we may define a 1-parameter family of maps $\psi_t : M \to M$ so that roughly speaking, each $\psi_t$ is a radial contraction on each generalized ball $B$ and identity outside. More precisely, by induction we can show that $\bigcup_{x \in \partial C} \spt (F(x) \res B)$ is contained in a generalized ball $B'$ contained in $B$ and with strictly smaller radius. Then define $\psi_t$ so that it is always the identity on $M \setminus \bigcup\mathcal{U}_C$, and over each ball $B$ disjoint from $\partial\bar{M}$, $\psi_0$ is the identity, and $\psi_1$ collapses each generalized ball to its center and stretches the annulus $\cl{B \setminus B'}$ onto the entirely of $B$.
    
    For generalized balls in the form of $U = E(B \times [0,r))$, the center is regarded as $(x,0)$ where $x$ is the center of $B$. Then $\bigcup_{x \in \partial C} \spt (F(x) \res U)$ is supported in $U' E(B' \times [0,r'))$ where $B'$ is a ball in $B$ of strictly smaller radius. Thus we define $\psi_t$ over $U$ is similarly to the previous case to collapse $U'$ to a point and stretch $U \setminus U'$ to the entire $U$. It is possible to define this so that $\psi_t$ restricts to a diffeomorphism $
    \partial\bar{M} \to \partial\bar{M}$ for $t < 1$, and so we can extend $\psi_t$ to a smooth map $M \to M$ so that $\psi_t(N) \subset N$, and so that for $t < 1$, $\psi_t$ is a diffeomorphism that restricts to a diffeomorphism on $N$.
    
    Then thinking of $C$ as a cone over $\partial C$, it suffices to construct a nullhomotopy $F_t : \partial C \to \Z{k}(M,N)$ of $F|_{\partial C}$. Simply define $F_t(x) = \psi_{t\sharp}(F(x))$. One need to check that $\psi_{1\sharp}(F(x))$ does not depend on $x$. If we had been dealing with cycles of positive dimension, then collapsing it to points would zero, so this would be true. This issue is more subtle for 0-cycles: $\psi_{t\sharp}(F(x))$ consists of a sum of points, each of which is the center of a generalized ball $B$, with coefficient equal to the sum of the coefficients of the points in $F(x) \res B$. (Ignore any points of $F(x)$ on $\partial\bar{M}$.) Nevertheless, the construction of the extension from \cite{GuthLiokumovich_ParamIneq} shows that for any $x, y \in \partial C$ and $B \in \mathcal{U}_C$, $F(x) \res B$ can be obtained from $F(y) \res B$ by adding the boundaries of some 1-chains supported in $B$, so the sum of the coefficients of the points in $F(x) \res B$ is the same as that for $F(y) \res B$.
    
    Now we proceed to construct the piecewise smooth structure. The polyhedral complex structure on $\tilde{X}$ will be its barycentric subdivision. $F|_{\tilde{X}^0}$ is trivially piecewise smooth: the model currents are $Z_v = F(v)$ and the charts are $\varphi_v(v,-) = \id$. Let us assume by induction that $F|_{X^d}$ is piecewise smooth with respect to the barycentric subdivison $\Sigma$ of $X^d$, for $d \geq 0$. Each simplex $\sigma$ in $\Sigma$ has a chart $\varphi_\sigma$ and model chain $Z_\sigma$.   The barycentric subdivision of $C$ is the cone over that of $\partial C$. At the apex $v$ of the cone we assign $Z_v = F(v)$ and $\varphi_v(v,-) = \id$. For each simplex $\sigma$ of $\partial C$, let $C\sigma$ denote its cone parametrized as a quotient of $I \times \sigma$. We define $Z_{C\sigma} = Z_\sigma$ and $\varphi_{C\sigma}(t,x,-) = \psi_t \circ \varphi_\sigma(x,-)$. It can be verified that these definitions satisfy \cref{eq:PiecewiseSmooth_Chart}.

    It remains to verify the compatibility relations. For each simplex $\sigma$ in $\Sigma$, the faces of $C\sigma$ are $\sigma$ itself and the cones over each face of $\sigma$. (If $\sigma$ is a vertex, then $v$ is also a face of $C\sigma$.) Define the collapse maps $\varphi_{\sigma(C\sigma)}(x,y) = y$ so that $\varphi_{\sigma(C\sigma)}(x,-)_\sharp(Z_{C\sigma}) = \id(Z_\sigma) = Z_\sigma$ and for any $x \in \sigma$, $\varphi_\sigma(x,-) \circ \varphi_{\sigma(C\sigma)}(x,-) = \psi_0 \circ \varphi_\sigma(x,-) = \varphi_{C\sigma}((0,x),-)$. For each face $\tau$ of $\sigma$, define $\varphi_{(C\tau)(C\sigma)}((t,x),y) = \varphi_{\tau\sigma}(x,y)$, so that $\varphi_{(C\tau)(C\sigma)}((t,x),-)_\sharp(Z_{C\sigma}) = \varphi_{\tau\sigma}(x,-)_\sharp(Z_\sigma) = Z_\tau = Z_{C\tau}$, and for any $(t,x) \in C\tau$ and $y \in M$,
    \begin{multline*}
        \varphi_{C\tau}((t,x), \varphi_{C\tau,C\sigma}((t,x),y)) = \psi_t(\varphi_\tau(x,\varphi_{\tau\sigma}(x,y)))
        \\
        = \psi_t \circ \varphi_\tau(x,-) \circ \varphi_{\tau\sigma}(x,-)(y) = \psi_t \circ \varphi_\sigma(x,-)(y) = \varphi_{C\sigma}((t,x),y).
    \end{multline*}
    Define $\varphi_{v,C\sigma}(v,-) = \psi_1$, Observe that $\sigma$ is adjacent to exactly one vertex $u$ of $\tilde{X}$, and $Z_{C\sigma} = Z_\sigma = \dotsb = Z_u = F(u)$. Hence $(\varphi_{v(C\sigma)})_\sharp(Z_{C\sigma}) = \psi_{1\sharp}(Z_\sigma) = \psi_{1\sharp}(F(u)) = F_1(u) = Z_v$ and \begin{equation*}
        \varphi_v(v,-) \circ \varphi_{v(C\sigma)}(v,-)(y) = \psi_1(y) = \psi_1(\varphi_\sigma(u,y)) =
        \varphi_{C\sigma}((1,u),y) = \varphi_{C\sigma}(v,y).
    \end{equation*}
\end{proof}

The $k \geq 1$ case follows a similar proof but needs a further modification for the following reason. The maps $\psi_t$ in the $k = 0$ case, when restricted to each generalized ball in the $\delta$-admissible neighbourhoods, stretch an annular subregion of the generalized ball onto the entire generalized ball. We chose this annular region to be away from the support of the cycles $F(x)$ involved, which is possible when $\dim F(x) = 0$. However, when $\dim F(x) \geq 1$, we must deal with the situation where $F(x)$ has some mass in the annular region. Thus we control the amount of the stretching in the asnnular region using $\delta$-admissible collections that are doubling.

\begin{restatable*}{lemma}{LemDeltaAdmissibleDoubling}
    \label{lem:DeltaAdmissibleDoubling}
    Let $X$ be a finite cubical complex, each of whose cells $C$ of positive dimension is associated with a $\delta$-admissible family $\mathcal{U}_C$. Then for some constant $c = c(\dim X, \max_C \#\mathcal{U}_C)$, there is another family of $c\delta$-admissible collections $\mathcal{V}_C$ that is indexed by the cells of $X$ of positive dimension and is also doubling, such that for each cell $C$, $\#\mathcal{V}_C \leq c$, and every $U \in \mathcal{U}_C$ is contained inside some $V \in \mathcal{V}_C$.
\end{restatable*}

\begin{proposition}
    \label{prop:ExtendPiecewiseSmooth}
    Let $X$ be an $n$-dimensional cubical complex that contains some subcomplex $Y$, and $k \geq 1$. Suppose that $f : X^0 \cup Y \to \Z{k}(M, N)$ is a $\delta$-localized family for some $\delta > 0$, which is also piecewise smooth with respect to the barycentric subdivision of $Y$. Suppose further that the corresponding family $\{\mathcal{U}_C\}$ of $\delta$-admissible collections is also doubling. Then $f$ extends to a map $F : X \to \Z{k}(M, N)$ that satisfies the following properties for some constant $c = c(\dim X, \max_C \#\mathcal{U}_C)$:
    \begin{enumerate}
        \item $F$ is $\delta$-localized with the same family $\{\mathcal{U}_C\}_C$.

        \item $F$ is piecewise smooth with respect to the barycentric subdivision of $X$.

        \item $\chi_F(2\delta) \leq c\chi_f(2\delta)$.

        \item For all $x$ and $y$ that lie in the same cell of $X$,
        \begin{equation}
            \label{eq:ExtendPiecewiseSmooth}
            \M(F(x), F(y)) < c\chi_f(2\delta).
        \end{equation}
    \end{enumerate}
\end{proposition}
\begin{proof}
    We will prove this by induction on $\dim X$. When $\dim X = 0$ there is nothing to prove. Suppose the lemma is true for $\dim X = d$, and now suppose that $\dim X = d+1$. By applying the lemma to $X^d$ and $X^d \cap Y$, we have an extension of $f$ to a piecewise smooth map $F : X^d \cup Y \to \Z{k}(M, N)$ that is $\delta$-localized with the family $\{\mathcal{U}_C\}$ such that $\M(F(x), F(y)) \leq c \chi_f(2\delta)$ for any $x,y \in X^d \cup Y$ that lie in the same cell of $X$.
    
    Let $C$ be a $(d+1)$-cell in $X$ that is not contained in $Y$. List the generalized balls inside every $\mathcal{U}_{C'}$, where $C'$ is $d$-face of $C$, as $B'_1, \dotsc, B'_{q'}$. Note that $q' \leq 2(d+1)\max_C \#\mathcal{U}_C \leq c$. Then for all $x, y \in \partial C$, $F(x) - F(y)$ is supported within $\bigcup_i B'_i$. Write $\mathcal{U}_C = \{B_1, \dotsc, B_q\}$. Now define $F$ over $C$ using a nullhomotopy $F_t$ of $F|_{\partial C}$: for each $B'_i$, one can define a radial deformation $\psi_t : \bigcup_i 2B'_i \to \bigcup_i 2B'_i$, so that $\psi_0$ is the identity, and $\psi_1$ collapses each $B'_i$ to its center and stretches $2B'_i \setminus B'_i$ to cover the whole $2B'_i$. This is well-defined because the family $\{\mathcal{U}_C\}$ is doubling, so the generalized balls $2B'_i$ are disjoint from one another. Moreover, we may choose this family so that every $\psi_t$ is 2-Lipschitz. Define $F_t(x) = (\psi_t)_\sharp(F(x))$. This is a nullhomotopy because as $x$ varies within $\partial C$, $F(x)$ stays the same outside of $\bigcup_i( 2B'_i \setminus B'_i)$.

    After extending $F$ over $C$, it is still $\delta$-localized with the same family $\{\mathcal{U}_C\}$ because, as the family is doubling, we have $\bigcup_i 2B' \subset \bigcup_j B_j$.
    \begin{lemma}
        \label{lem:ExtendPiecewiseSmooth_MassConcProfileRecursion}
        $\chi_{F|_C}(2\delta) \leq (1 + q')\chi_{F|_{X^d \cup Y}}(2\delta) \leq c\chi_{F|_{X^d \cup Y}}(2\delta)$.
    \end{lemma}
    \begin{proof}
        Let $B$ be any ball of radius $2\delta$ in $M$, and consider any $x \in C$. Let $y$ be any point on $\partial C$ that lies a common radius of $C$ with $x$. Then $F(x)$ differs from $F(y)$ only within $\bigcup_i 2B'_i$. Thus $\M(F(x) \res B) \leq \M(F(y) \res B) + \sum_i \M(F(y) \res 2B'_i) \leq (1 + q')\chi_{F|_{X^d \cup Y}}(2\delta)$. The lemma follows.
    \end{proof}
    
    Therefore, we may extend $F$ to the whole of $X$, after which we have
    \begin{equation*}
        \chi_F(2\delta) \leq \max_C \chi_{F|_C}(2\delta) \leq 2(d+1)(\max_C \#\mathcal{U}_C) \chi_{F|_{X^d \cup Y}}(2\delta) \leq c\chi_{F|_{X^d \cup Y}}(2\delta) \leq c\chi_f(2\delta),
    \end{equation*}
    where the last inequality follows from the induction hypothesis.
    
    And for any $x, y \in C$, since $F(x)$ differs from $F(y)$ only within $\bigcup_i 2B'_i$,
    \begin{align*}
        \M(F(x),F(y))
        &\leq \sum_i (\M(F(x) \res 2B'_i) + \M(F(y) \res 2B'_i))
        \\
        &\leq 2q'\chi_{F|_C}(2\delta)
        \\
        (\text{\cref{lem:ExtendPiecewiseSmooth_MassConcProfileRecursion}}) &\leq c\chi_{F|_{X^d \cup Y}}(2\delta)
        \\
        (\text{induction hypothesis}) &\leq c\chi_f(2\delta). 
    \end{align*}

    Now we must show that $F$ is piecewise smooth with respect to the barycentric subdvision of $X$. This is entirely analogous to the inductive construction of the piecewise smooth structure from the proof of the $k = 0$ case of \cref{thm:PiecewiseSmoothApprox}.
\end{proof}

Given an open set $U \subset M$, let $\partial_r U$ denote the $r$-neighbourhood of $\partial U$.

\begin{restatable*}{proposition}{PropDiscreteDeltaLocExtension}
    \label{prop:DiscreteDeltaLocExtension}
    For any $\delta > 0$ and $k \geq 1$, there exist constants $\varepsilon_0 = \varepsilon_0(n,\bar{M},\delta)$, $c = c(n,\bar{M},\delta)$, and $0 < r(M,\delta) \leq \delta$ so that the following holds. Suppose that a map $f : X^0 \to \Z{k}(M,N)$ from the vertices of an $n$-dimensional cubical complex $X$ is $\varepsilon$-fine for some $0 < \varepsilon < \varepsilon_0$. Then for some a refinement $\tilde{X}$ of $X$, $f$ can be extended to a $\delta$-localized map $F : \tilde{X}^0 \to \Z{k}(M,N)$ that satisfies following properties:
    \begin{enumerate}
        \item\label{enum:DiscreteDeltaLocExtension_FlatBound} For any cell $C$ of $X$, any vertex $v$ of $C$, and any $x \in C \cap \tilde{X}^0$, $\Fl(f(v),F(x)) \leq c\varepsilon$.

        \item\label{enum:DiscreteDeltaLocExtension_MassBound} $\max_{x \in \tilde{X}^0}\M(F(x)) \leq \max_{x \in X^0} \M(f(x)) + c\varepsilon$.
        \item\label{enum:DiscreteDeltaLocExtension_Inductive} For each cell $C$ of positive dimension in $X$ and each $x \in C \cap \tilde{X}^0$, $F(x)$ is, roughly speaking, ``patched together'' from restrictions of $F(y)$ for $y \in \partial C \cap \tilde{X}^0$. More precisely, there exists a cover $U_1,\dotsc,U_N$ of $M$ by open sets of diameter $< \delta$, depending on $C$, so that each $x \in C \cap \tilde{X}^0$ corresponds to some $y_1, \dotsc, y_N \in \partial C \cap \tilde{X}^0$, such that
    \begin{equation}
        \label{eq:DiscreteDeltaLocExtension}
        \M\left( F(x) - \sum_{i = 1}^N F(y_i) \res U_i \right) \leq c\varepsilon.
    \end{equation}
        Moreover, for any sequence of parameters $0 < \lambda_1,\dotsc, \lambda_N < 1$, the $U_i$'s can be chosen so that for any $x \in C \cap \tilde{X}^0$,
        \begin{align}
            \M(F(x) \res \partial U_i) &\leq \frac{2N^2}{\delta} (\M(F(x)) + c\varepsilon) \label{eq:DiscreteDeltaLocExtension_Coarea}
            \\
            \M(F(x) \res \partial_{\lambda_i r/4} U_i) &\leq \lambda_i N(\M(F(x))  + c\varepsilon). \label{eq:DiscreteDeltaLocExtension_BdryControl}
        \end{align}

        \item For each cell $C$ of $X$, the $\delta$-admissible collection associated with $C$ has at most $c_4(n,\dim M)$ elements.
    \end{enumerate}
\end{restatable*}

\begin{corollary}
    \label{cor:DiscreteDeltaLocExtensionIsFine}
    For any $\delta > 0$, $\mu > 0$, and $k \geq 1$, there exist constants $\varepsilon_0(n,\bar{M},\delta)$ and $c(n,\bar{M},\delta,\mu)$ so that the following holds. Suppose that a map $f : X^0 \to \Z{k}(M, N)$ from the vertices of an $n$-dimensional cubical complex $X$ is $\varepsilon$-fine. Then the $\delta$-localized extension $F : \tilde{X}^0 \to \Z{k}(M,N)$ of $f$ that is obtained from \cref{prop:DiscreteDeltaLocExtension} satisfies the following property:
    \begin{itemize}
        \item $\chi_F(r) \leq c_3(r/\delta)\chi_f(r)$ for all $r > 0$, where $c_3(r/\delta) = c_3(k,r/\delta)$
    \end{itemize}
\end{corollary}
\begin{proof}
    $F$ is as given from \cref{prop:DiscreteDeltaLocExtension}. We will prove that for any $r > 0$, $\chi_{F_j}(r) \leq c_3(r/\delta)\chi_f(r)$. Let $B$ be any generalized ball of radius $r$. Consider any vertex $z$ of $C(q_j)$, and suppose that $F_j(z) = G(x,t)$. Since the $U_i$'s come from a cover of $M$ by generalized balls, $B$ can intersect at most $(r/\delta)^{\dim M}$ of the open sets $U_i$. Thus
    \begin{align*}
        \M(F_j(z) \res B)
        \leq{}& \M(F_j(x) \res B) + \sum_{U_i \cap B \neq \emptyset} \M(F_j(y_i) \res \partial U_i)\\
        +{}& \M(\tau(x) \res B) + \sum_{U_i \cap B \neq \emptyset} \M(\tau(y_i) \res B)
        \\
        \leq&{} c_3(r/\delta)(\chi_{F_{j-1}}(r) + c_2\varepsilon)
        \\
        \leq&{} c_3(r/\delta)(\chi_f(r) +\varepsilon).
    \end{align*}
\end{proof}

\begin{proof}[Proof of \cref{thm:PiecewiseSmoothApprox} when $k \geq 1$]
    We will prove the case where $Y = \emptyset$ as the alternative follows similarly. Let us reduce to the case where $f$ has no concentration of mass. Since $X$ is compact and $f$ is continuous in the inductive limit topology, for all $x \in X$ we have $\M(f(x)) < \mu$ for some constant $\mu$ \cite[p.~262]{Almgren_HomotopyGroupsIntegralCycles}. Thus we can use \cite[Theorem~2.3]{GuthLiokumovich_ParamIneq} to homotope $f$ to a map $f^*$ such that $\Fl(f(x), f^*(x)) < \varepsilon$ for every $x \in X$ and some $\varepsilon > 0$ that can be made arbitrarily small. From the construction of $f^*$, we know that $f^*$ has no concentration of mass. We also know from the construction that there exists some constant $\mu' > 0$ such that $f$, $f^*$, and the homotopy between them consist of only cycles of mass less than $\mu'$. Thus $f^*$ and the homotopy are also continuous in the inductive limit topology. Thus we may rename $f^*$ to $f$ and proceed with the proof, assuming that $f$ has no concentration of mass.

    Next, we will construct a sequence of piecewise smooth approximations of $f$ that ``converge'' to $f$, show that each approximation is homotopic to a better one, and finally chain all of these homotopies together into a single homotopy from the worst approximation to $f$.
    
    Let $\mu = \max_{x \in X} \M(f(x))$. Choose some decreasing sequence $1 = \delta_1 > \delta_2 > \dotsb \to 0$, and let $\eta_i = \min\{\delta_i, \varepsilon_0(n,M,\delta_i,\mu)\}$, where $\varepsilon_0$ is defined in \cref{cor:DiscreteDeltaLocExtensionIsFine}. Choose an increasing sequence of integers $0 < q_1 < q_2 < \dotsb$ so that $f : X(q_i) \to \Z{k}(M,N)$ is $\eta_i$-fine.

    $f : X(q_i) \to \Z{k}(M,N)$ is $\eta_i$-fine so \cref{prop:DiscreteDeltaLocExtension} gives an extension of $f|_{X(q_i)^0}$ to a $\delta_i$-localized map $f_i : \tilde{X}_i^0 \to \Z{k}(M,N)$, where $\tilde{X}_i$ is a refinement of $X(q_i)$. \Cref{lem:DeltaAdmissibleDoubling} allows us to assume that the family $\{\mathcal{U}_C\}$ of $\delta$-admissible collections associated to $f_i$ is doubling, at the expense of making $f_i$ $c\delta_i$-localized. Furthermore, due to \cref{prop:DiscreteDeltaLocExtension}(4), $\max_C \#\mathcal{U}_C \leq c$.
    
    Thus \cref{prop:ExtendPiecewiseSmooth} extends $f_i$ to a $c\delta_i$-localized map $F_i : \tilde{X}_i \to \Z{k}(M,N)$ that is piecewise smooth with respect to the barycentric subdivision of $\tilde{X}_i$. Now let us show that each $F_i$ is homotopic to $F_{i+1}$. via some homotopy $G_i : \tilde{X}_{i,i+1} \times I \to \Z{k}(M,N)$, where $\tilde{X}_{i,i+1}$ denotes the common refinement of $\tilde{X}_i$ and $\tilde{X}_{i+1}$. Define $G_i(x,0) = F_i(x)$ and $G_i(x,1) = F_{i+1}(x)$. We wish to extend $G_i$ over $\tilde{X}_{i,i+1} \times I$ using \cref{prop:ExtendPiecewiseSmooth}, so we have to build a family of admissible collections for $G_i$ out of those for $F_i$ and $F_{i+1}$.

    For each cell $C$ of $\tilde{X}_{i,i+1}$ and $j = i, i+1$, let $\mathcal{U}^j_C$ denote the $c\delta_j$-admissible collection associated to the cell of lowest dimension in $\tilde{X}_j$ that contains $C$. Now, for each cell $E = C \times C'$ in $\tilde{X}_{i,i+1} \times I$, define
    \begin{equation*}
        \mathcal{V}_E = \begin{cases}
            \mathcal{U}^i_C & C' = \{0\}
            \\
            \mathcal{U}^{i+1}_C & C' = \{1\}
            \\
            \mathcal{U}^i_C \cup \mathcal{U}^{i+1}_C & C' = [0,1].
        \end{cases}
    \end{equation*}
    Observe that even though each $\mathcal{V}_E$ may not be $\delta$-admissible (its members may intersect), the proof of \cref{lem:DeltaAdmissibleDoubling} applies to the family $\{\mathcal{V}_E\}_E$. From before we have $\max_C \#\mathcal{U}^i_C \leq c$ and $\max_C \#\mathcal{U}^{i+1}_C \leq c$, so we can obtain a family $\{\tilde{\mathcal{V}}_E\}_E$ of $c\delta_i$-admissible collections that is doubling. Thus it can be verified that $G_i$ is $c\delta_i$-localized, and hence $G_i$ can be extended into a homotopy from $F_i$ and $F_{i+1}$ by \cref{prop:ExtendPiecewiseSmooth}.

    It remains to ``chain'' the homotopies $G_i$ together to form a homotopy from $F_1$ to $f$. Define a map $G : X \times I \to \Z{k}(M,N)$ as follows: define $G$ over $X \times \{0\}$ using $f$, and for each positive integer $i$, define $G$ over $X \times [\frac1i, \frac1{i+1}]$ using $G_i$ (using $F_i$ at $X \times \{\frac1i\}$ and so on). Thus $G$ is continuous except possibly at points in $X \times \{0\}$. To prove that it is also continuous there, it suffices to prove the following claim:

    \begin{claim}
        For any $\lambda > 0$, there exists some $i$ such that for any cell $C$ of $X(q_i)$, $x_0, x \in C$, and $t \in [0,\frac1i]$, $\Fl(G(x,t), G(x_0,0)) = \Fl(G(x,t), f(x_0)) < \lambda$.
    \end{claim}

    Suppose that $i$ is a large integer to be chosen later. Let $t \in [\frac1j, \frac1{j+1}]$ for some $j \geq i$. Let $C'$ be the cell of lowest dimension in $X(q_j)$ that contains $x$, so $C' \subset C$. Let $C''$ be the cell of lowest dimension in $\tilde{X}_j$ that contains $x$, so $C'' \subset C'$. Let $v'$ (resp. $v''$) be a vertex of $C'$ (resp. $C''$). Then by the triangle inequality we have
    \begin{equation*}
        \Fl(G(x,t), f(x_0)) \leq \Fl(G(x,t), G(x,\tfrac1j)) + \Fl(\underbrace{G(x,\tfrac1j)}_{F_j(x)}, F_j(v'')) + \Fl(F_j(v''), f(v')) + \Fl(f(v'), f(x_0)).
    \end{equation*}

    Since $G$ was extended over $X \times [\frac1j, \frac1{j+1}]$ using \cref{prop:ExtendPiecewiseSmooth}, we have
    \begin{align*}
        \Fl(G(x,t), G(x,\tfrac1j))
        &\leq 2\M(G(x,t), G(x,\tfrac1j))
        \\
        (\text{\cref{prop:ExtendPiecewiseSmooth}})
        &\leq c\chi_{G_j|_{X \times \{0,1\}}}(2c\delta_j)
        \\
        &\leq c(\chi_{F_j}(2c\delta_j) + \chi_{F_{j+1}}(2c\delta_j))
        \\
        (\text{\cref{prop:ExtendPiecewiseSmooth}})
        &\leq c(\chi_{f_j}(2c\delta_j) + \chi_{f_{j+1}}(2c\delta_j))
        \\
        (\text{\cref{cor:DiscreteDeltaLocExtensionIsFine}})
        &\leq c_3(2c)c\chi_f(\delta_j)
        \\
        &\leq c\chi_f(\delta_j).
    \end{align*}

    Similarly, we have
    \begin{align*}
        \Fl(F_j(x), F_j(v''))
        &\leq 2\M(F_j(x), F_j(v''))
        \\
        (\text{\cref{prop:ExtendPiecewiseSmooth}})
        &\leq c\chi_{f_j}(2c\delta_j)
        \\
        (\text{\cref{cor:DiscreteDeltaLocExtensionIsFine}})
        &\leq c_3(2c)c\chi_f(\delta_j)
        \\
        &\leq c\chi_f(\delta_j).
    \end{align*}

    Meanwhile, \cref{prop:DiscreteDeltaLocExtension}(1)  implies that $\Fl(F_j(v''), f(v)) \leq c\eta_j \leq c\delta_j$, and the fact that $f : X(q_i) \to \Z{k}(M,N)$ is $\eta_i$-fine implies that $\Fl(f(v'), f(x_0)) \leq \eta_i \leq \delta_i$.

    To conclude, the claim follows from our assumption that $f$ has no concentration of mass, which guarantees that $\Fl(G(x,t), f(x_0))$ can be made arbitrarily small with a sufficiently large choice of $i$.
\end{proof}

\section{Gluings}
\label{sec:Gluings}

Throughout this section, let $(M,N)$ be a collar pair.

\begin{definition}[Gluing]
    Given a compact Riemannian manifold $X$ and a continuous map $f : X \to \Z{d}(M,N)$ that vanishes over $\partial X$, a \emph{gluing} of $f$ is a sequence of continuous homomorphisms $\Phi^k : \I{k}(X) \to \I{k+d}(M,N)$ for $k \geq 0$ such that $\partial \circ \Phi^{k+1} = \Phi^k \circ \partial$ and for any $x \in X$, $\Phi^0(x) = f(x)$.
    
    In other words, we may write $\Phi : \I{*}(X) \to \I{*+d}(M,N)$ and call it a \emph{continuous chain homomorphism of degree $d$}.
\end{definition}

\begin{definition}
    We say that a singular chain $T$ in $X$ is a \emph{$C^k$ chain} for some $k \leq \infty$ if for each singular simplex $\sigma : \Delta^m \to X$ in the chain, $\sigma$ is $C^k$ when restricted to the interior of every subsimplex of $\Delta^n$.

    We say that $T$ intersects a submanifold of $X$ transversally if the same is true of each of its singular simplices.
\end{definition}

\begin{restatable*}{lemma}{LemCInfinityChainsDense}
    \label{lem:CInfinityChainsDense}
    Let $X$ be a compact Riemannian manifold, and consider any finite collection of its submanifolds. The $C^\infty$ $k$-chains in $X$ that intersect each of those submanifolds transversally form a dense subspace of $\I{k}(X)$. In fact, any $T \in \I{k}(X)$ is the limit in the flat topology of a sequence of $C^\infty$ $k$-chains $T_i$ that intersect each of those submanifolds transversally, such that $\M(T_i) \to \M(T)$ and $\M(\partial T_i) \to \M(\partial T)$.
\end{restatable*}

\begin{theorem}
    \label{thm:PiecewiseSmoothGluing_Existence}
    If $f : X \to \Z{d}(M,N)$ vanishes over $\partial X$ and has a piecewise smooth structure $(\{\varphi_\sigma\}, \{Z_\sigma\}, \{\varphi_{\tau\sigma}\})$, then $f$ has a gluing $\Phi$, called the \emph{standard gluing} for this piecewise smooth structure, such that for all $C^\infty$ $k$-chains $T$ that intersects every cell of $X$ transversally,
    \begin{equation}
        \label{eq:PiecewiseSmoothGluingDefn}
        \Phi(T) = \sum_{\dim \sigma = \dim X} \varphi_{\sigma\sharp}((T \res \sigma) \times Z_\sigma).
    \end{equation}
\end{theorem}
\begin{proof}
    For convenience, let us say that $T \in \I{k}(X)$ is \emph{in general position} if it is a $C^\infty$ $k$-chain that intersects every open cell of $X$ transversally. For such $T = \sum_i a_i\sigma_i$, we first show that $\partial\Phi(T) = \Phi(\partial T)$. For any top-dimensional cell $\tau$, the fact that $T$ is in general position implies that $\partial(T \res \tau) = \partial T \res \tau + T \cap \partial\tau$. To be precise, by $T \cap \partial\tau$ we mean the $(k-1)$-chain $\sum_j\sum_i a_i \sigma_i|_{\sigma_i^{-1}(\itr \tau_j)}$, where the $\tau_j$ are the faces of $\tau$ with the boundary orientation, $\sigma^{-1}(\itr \tau_j)$ is triangulated and given the preimage orientation, and $\sigma_i|_{\sigma_i^{-1}(\itr \tau_j)}$ is treated as a Lipschitz $(k-1)$-chain. Thus it can be verified that, if we write $n = \dim X$,
    \begin{multline*}
        \partial\Phi(T)
        = \sum_{\dim \sigma = n} \partial\varphi_{\sigma\sharp}((T \res \sigma) \times Z_\sigma)
        = \sum_{\dim \sigma = n} \varphi_{\sigma\sharp}(\partial(T \res \sigma) \times Z_\sigma)
        \\
        = \sum_{\dim \sigma = n} \varphi_{\sigma\sharp}((\partial T \res \sigma + T \cap \partial \sigma) \times Z_\sigma)
        = \Phi(\partial T) + \underbrace{\sum_{\dim \sigma = n} \varphi_{\sigma\sharp}((T \cap \partial \sigma) \times Z_\sigma)}_{(\star)}.
    \end{multline*}
    
    Let us show that $(\star)$ vanishes. Each $(n-1)$-cell $\tau$ of $X$ that is not contained in $\partial X$ is a face of two $n$-cells, $\sigma$ and $\sigma'$. Pick any orientation on $\tau$; by \cref{eq:CollapseMapIdentity} we have $\varphi_\sigma|_{\tau \times M} = \varphi_\tau \circ \psi$ where $\psi(x,y) = (x, \varphi_{\tau\sigma}(x,y))$, so
    \begin{equation}
        \label{eq:PiecewiseSmoothGluing_Collapse}
        \varphi_{\sigma\sharp}((T \cap \tau) \times Z_\sigma) = \varphi_{\tau\sharp} \circ \psi_\sharp((T \cap \tau) \times Z_\sigma) = \varphi_{\tau\sharp}((T \cap \tau) \times Z_\tau).
    \end{equation}
    The second equality holds because of the following lemma:
    \begin{lemma}
        $\psi_\sharp((T \cap \tau) \times Z_\sigma) = (T \cap \tau) \times Z_\tau$.
    \end{lemma}
    \begin{proof}
        This is true essentially because Cartesian products of currents are characterized by projections. To be precise, let $\tilde{T} \in T \cap \tau$ and $\tilde{Z} \in Z_\tau$, which implies that $\tilde{T} \times \tilde{Z} \in (T \cap \tau) \times Z_\tau$. Then consider any smooth differential forms $\omega \in \Omega^i(\tau)$ and $\theta \in \Omega^{d+k-1-i}(M)$, and let $\pr_1$ and $\pr_2$ be the projections of $\tau \times M$ onto its first and second factor respectively. Then
        \begin{multline*}
            \psi_\sharp(\tilde{T} \times \tilde{Z})(\pr_1^*\omega \wedge \pr_2^*\theta) = (\tilde{T} \times \tilde{Z})\psi^*(\pr_1^*\omega \wedge \pr_2^*\theta)
            = (\tilde{T} \times \tilde{Z})((\underbrace{\pr_1 \circ \psi}_{\id})^*\omega \wedge (\underbrace{\pr_2 \circ \psi}_{\varphi_{\tau\sigma}})^*\theta)
            \\
            = (\tilde{T} \times \tilde{Z})(\omega \wedge \varphi_{\tau\sigma}^*\theta)
            = \begin{cases}
                0 & i \neq k-1,
                \\
                \tilde{T}(\omega)\varphi_{\tau\sigma\sharp}(\tilde{Z})(\theta) & i = k-1.
            \end{cases}
        \end{multline*}

        The uniqueness condition in the definition of the Cartesian product of currents \cite[4.1.8]{Federer_GMT} then implies that $\psi_\sharp(\tilde{T} \times \tilde{Z}) = \tilde{T} \times \varphi_{\tau\sigma\sharp}(\tilde{Z})$. Hence, $\psi_\sharp((T \cap \tau) \times Z_\sigma) = [\psi_\sharp(\tilde{T} \times \tilde{Z})] = [\tilde{T} \times \varphi_{\tau\sigma\sharp}(\tilde{Z})] = [\tilde{T}] \times \varphi_{\tau\sigma\sharp}[\tilde{Z}] = (T \cap \tau) \times Z_\tau$.
    \end{proof}

    Now observe that $\tau$ appears in $\partial\sigma$ and $\partial\sigma'$ but with opposite orientations. Thus $(\star)$ may be expanded into terms of the form in \cref{eq:PiecewiseSmoothGluing_Collapse}, and those terms involving $\tau$ will cancel each other out. The remaining terms correspond to $(n-1)$-cells $\tau \subset \partial X$, but since $f$ vanishes over $\partial X$, we have $Z_\tau = 0$, so those terms must also vanish. Therefore, we have now proven that $\partial\Phi(T) = \Phi(\partial T)$.

    It remains to analyze the continuity of $\Phi^k$. By the characterization of continuity in the inductive limit topology in \cref{lem:ContinuityIndLimTop}, we see that a sufficient condition for a map $\I{k}(X) \to \I{k+d}(M,N)$ to be continuous is that it is Lipschitz in both the flat metric and $\M$, in the sense that the map increases mass by at most some constant factor.

    It follows from \cref{eq:PiecewiseSmoothGluingDefn} that $\M(\Phi(T)) \leq C\M(T)$ when $T$ is in general position, where $C = \max_\sigma \Lip(\varphi_\sigma)\M(Z_\sigma)$. Let us also prove that $\Phi^k$, defined over mod $p$ integral currents in general position, is $2C$-Lipschitz with respect to the flat metric. Let $S, T \in \I{k}(X)$ be in general position. Then there exist rectifiable chains $Q$ and $R$ such that $S - T = \partial Q + R$ and $\M(Q) + \M(R) < 2\Fl(S,T)$. By the deformation theorem (projecting $Q$), we may assume that $Q$ and $R$ are in general position. Thus $\Phi$ is defined on $Q$ and $R$, and $\Phi(S) - \Phi(T) = \Phi(\partial Q) + \Phi(R) = \partial\Phi(Q) + \Phi(R)$. Thus $\Fl(\Phi(S), \Phi(T)) \leq 2C\Fl(S,T)$.

    Thus $\Phi^k$ is continuous over the chains in general position, which form a dense subspace of $\I{k}(X)$ by \cref{lem:CInfinityChainsDense}. Moreover, $\Phi^k$ is uniformly continuous with respect to the flat metric. Thus it extends uniquely to a continuous map $\Phi^k : \I{k}(X) \to F_k(M,N)$. Let us show that the chains in the image have finite mass: Indeed, for any $T \in \I{k}(X)$, \cref{lem:CInfinityChainsDense} gives us a Cauchy sequence $T_i$ of $C^\infty$ $k$-chains in generic position so that $\Fl(T,T_i) \to 0$, $\M(T_i) \to \M(T)$, and $\M(\partial T_i) \to \M(\partial T)$. Thus we may assume that $\M(T), \M(\partial T), \M(T_i), \M(\partial T_i) < \mu$ for some $\mu > 0$. Thus $\Phi^k(T_i)$ is a Cauchy sequence in $\I{k}(M,N; \Fl)$, which by the Compactness Theorem for integral currents mod $p$ \cite[p.~432]{Federer_GMT} implies that $\Phi^k(T)$ also has mass at most $C\mu$.
\end{proof}

\begin{example}
    \label{eg:FiberBundle_Gluing}
    Consider a fiber bundle $\xi : E \to B$, where $E$ and $B$ are closed Riemannian manifolds. \Cref{eg:FiberBundle_PiecewiseSmooth} shows that the preimage map $\xi^{-1}(-) : B \to \Z{k}(E)$ is piecewise smooth. The definition of $(\xi^{-1}(-))_\sharp$ shows that it sends the fundamental cycle of $B$ to that of $E$. Thus $(\xi^{-1}(-))_*$ sends the fundamental class of $B$ to that of $E$.
\end{example}

Given a map $f : X \to \Z{k}(M,N)$ that has a gluing $\Phi$, $\Phi$ may not be the unique gluing of $f$. However, we will show that the induced homomorphism of homology groups which we denote by $\Phi_* : H_\bullet(X, \partial X) \to H_{\bullet + k}(M, N)$ does not depend on the choice of gluing.

\begin{lemma}
    \label{lem:GluingPiecewiseSmoothApproxChainHomotopy}
    Suppose that $f : X \to \Z{k}(M,N)$ has a gluing $\Phi$. Then there exists some $\varepsilon > 0$, depending on $f$ and $\Phi$, such that for any piecewise smooth approximation $g$ of $f$ such that $\Fl(f(x), g(x)) < \varepsilon$ for all $x \in X$, the homomorphism on homology groups induced by $\Phi$ is precisely $\Gamma_*$, where $\Gamma$ is the standard gluing of $g$.
\end{lemma}
\begin{proof}
    Both $\Phi$ and $\Gamma$ restrict to homomorphisms $C^\text{cell}_\bullet(X,\partial X) \to \I{\bullet + k}(M,N)$. By subdivision, we may assume that for any cell $\sigma$ of positive dimension in $X$, we have $\M(\Phi(\sigma)), \M(\Gamma(\sigma)) < \varepsilon$. We construct a chain homotopy $\Omega_i : C^\text{cell}_i(X) \to \I{i+k+1}(M,N)$ between these two homomorphisms inductively by using the isoperimetric inequality, following \cite[p.~1153]{Guth_WidthVolume}. For each vertex $v$ of $X$, since $g$ is a piecewise smooth approximation of $f$, we have $\Fl(f(v), g(v)) < \varepsilon$. By the isoperimetric inequality we may assume that $f(v) - g(v)$ can be filled by a $(k+1)$-chain of mass at most $\varepsilon$, which we denote by $\Omega_0(v)$. This defines $\Omega_0$. Now suppose $\Omega_d$ has been defined for some $d \geq 0$. Then for any $(d+1)$-cell $\sigma$ in $X$, define $\Omega_{d+1}(\sigma)$ to be a chain filling $\Phi(\sigma) - \Gamma(\sigma) - \Omega_d(\partial \sigma)$ using the isoperimetric inequality. The mass of this filling can be made arbitrarily small by shrinking $\varepsilon$. Continue inductively.
\end{proof}

\begin{lemma}
    \label{lem:HomotopicPiecewiseSmoothChainHomotopy}
    If piecewise smooth maps $f, g : X \to \Z{k}(M,N)$ are homotopic, then their standard gluings $\Phi$ and $\Gamma$ are chain homotopic. That is, there is a sequence of homomorphisms $\Omega_i : \I{i}(X) \to \I{i+k+1}(M,N)$ such that $\Phi^i - \Gamma^i = \partial \circ \Omega_{i+1} + \Omega_i \circ \partial$.

    As a consequence, $\Phi_* = \Gamma_*$.
\end{lemma}
\begin{proof}
    Using a common refinement, we may assume that $f$ and $g$ are piecewise smooth respect to the same cubical structure. Denote the homotopy between $f$ and $g$ by $H : X \times I \to \Z{k}(M,N)$. By \cref{thm:PiecewiseSmoothApprox}, we may assume that $H$ is piecewise smooth with respect to some refinement $\widetilde{X \times I}$ of $X \times I$. By \cref{thm:PiecewiseSmoothGluing_Existence}, $H$ has a gluing $\Psi : \I{*}(X \times I) \to \I{*+k}(M,N)$. Observe that by the definition of $\Psi$ in \cref{eq:PiecewiseSmoothGluingDefn}, $\Phi(T) = \Psi(T \times \{0\})$. Similarly, $\Gamma(T) = \Psi(T \times \{1\})$.
    
    Set $\Omega(T) = \Psi(T \times I)$. Then
    \begin{multline*}
        (\partial \circ \Omega_{i+1} + \Omega_i \circ \partial)(T)
        = \partial \Psi^{i+1}(T \times I) + \Psi^i(\partial T \times I)
        = \Psi^i(\partial(T \times I)) \pm \Psi^i(\partial T \times I)
        \\
        = \Psi^i(T \times \{1\} - T \times \{0\})
        = \Psi^i(T \times \{1\}) - \Psi^i(T \times \{0\})
        = \Gamma(T) - \Phi(T).
    \end{multline*}
\end{proof}

\begin{corollary}
    \label{cor:HomotopicMapsWithGluingChainHomotopic}
    If $f$ and $g$ are homotopic maps with gluings $\Phi$ and $\Psi$, then $\Phi_* = \Psi_*$.
\end{corollary}

Thus the induced map on homology groups does not depend on the choice of gluing $\Phi$ of $f$. We denote this induced map on homology groups by $f_*$.

\begin{restatable*}{proposition}{PropPiecewiseSmoothGluing}
    \label{prop:GluingProperties}
    When $f : X \to \Z{k}(M,N)$ has a gluing $\Phi$, then the following laws hold:
    \begin{description}
        \item[Composition law.] Suppose that $g : L \to \Z{c}(X)$ also has a gluing $\Psi$. Then $\Phi \circ \Psi$ is a gluing for $\Phi \circ g$. More precisely, the gluing consists of chain homomorphisms $\Phi^{l+c} \circ \Psi^l : \I{l}(L) \to \I{l+c+d}(M)$.

        We also have $(\Phi \circ g)_* = f_* \circ g_*$.

        \item[Addition law.] Suppose that $g : X \to \Z{k}(M,N)$ also has a gluing $\Psi$. Then $(f + g)$ has a gluing $\Phi + \Psi$ and $(f + g)_* = f_* + g_*$.
    \end{description}
    
    \begin{enumerate}
        \item\label{enum:GluingProperties_WedgeSum}
        Suppose that $g : X' \to \Z{k}(M,N)$ also has a gluing $\Psi$. Then the induced map $f \vee g : X \vee X' \to \Z{k}(M,N)$ has a gluing $T \mapsto \Phi(T \res X) + \Psi(T \res X')$. Moreover, $(f \vee g)_* = f_* \vee g_*$.
        \item\label{enum:PiecewiseSmoothGluing_GluingOfProduct} For any $Z \in \Z{d'}(M')$, define the maps $Z \times f : N \to \Z{d' + d}(M' \times M)$ and $Z \smsh f : N \to \Z{d' + d}(M' \smsh M)$ by $(Z \times f)(x) = Z \times f(x)$ and $(Z \smsh f)(x) = Z \smsh f(x)$. Similarly define $f \times Z$ and $f \smsh Z$. Then
        \begin{align*}
            (Z \times f)_\sharp(A) &= Z \times f_\sharp(A) & (f \times Z)_\sharp(A) &= f_\sharp(A) \times Z
            \\
            (Z \smsh f)_\sharp(A) &= Z \smsh f_\sharp(A) & (f \smsh Z)_\sharp(A) &= f_\sharp(A) \smsh Z.
        \end{align*}

        \item\label{enum:GluingProperties_Suspension} Consider the map $F = SX \xrightarrow{Sf} S\Z{k}(M,N) \to \Z{k}(SM,SN)$. Then $F_*$ and $f_*$ commute with the suspension isomorphisms in the following diagram:
        \begin{equation*}
        \begin{tikzcd}[ampersand replacement = \&]
            H_\bullet(X) \rar{f_*} \arrow[d,"\iso"']
            \& H_{\bullet + k}(M,N) \dar{\iso}
            \\H_{\bullet+1}(SX) \rar{F_*} \& H_{\bullet + k+1}(SM,SN)
        \end{tikzcd}
        \end{equation*}
    \end{enumerate}
\end{restatable*}

\subsection{The Almgren Isomorphism Theorem and Brown representability for the inductive limit topology}

Let $X$ be a pointed compact Riemannian $d$-manifold. Let $[X, \Z{k}(M,N)]$ denote the abelian group of homotopy classes of continuous pointed maps $X \to \Z{k}(M,N)$ that send the point to 0 and which also vanish over $\partial X$, and group operation is induced by the addition of cycles. Now we describe the \emph{gluing homomorphism} $[X, \Z{k}(M,N)] \to H_{k+d}(M,N)$ as follows. For each element $\alpha \in [X, \Z{k}(M,N)]$, choose some representative $f : X \to \Z{k}(M,N)$ that has a gluing, which gives us an element $f_*[X] \in \hat{H}_{k+d}(M,N)$. This element exists as a result of \cref{thm:PiecewiseSmoothApprox,thm:PiecewiseSmoothGluing_Existence}. Moreover this correspondence is an additive homomorphism because of the addition law from \cref{prop:GluingProperties}.

When $X = \Sp{i}$, the gluing homomorphism $\pi_i(\Z{k}(M,N)) \to \hat{H}_{i+k}(M,N)$ is compatible with the usual group structure of homotopy groups because of the wedge sum law from \cref{prop:GluingProperties}.

\begin{theorem}(The Almgren Isomorphism Theorem for the inductive limit topology)
    \label{Thm:AlmgrenIsoFMetric}
    For any $i, k \geq 0$, the gluing homomorphism $\Z{k}(M,\partial M) \to \hat{H}_{k+i}(M,\partial M)$ is an isomorphism.
\end{theorem}

\begin{lemma}
    \label{lem:QuasifibLES}
    The map $\xi : \Z{k}(M) \to \Z{k}(M,N)$ satisfies a ``relatively piecewise smooth homotopy lifting property'': for any continuous map $\tilde{f} : X \to \Z{k}(M)$ and a piecewise smooth homotopy $F : X \times I \to \Z{k}(M,N)$ such that $F(-,0) = \xi \circ \tilde{f}$, there exists a continuous map $\tilde{F} : X \times I \to \Z{k}(M,N)$ such that $F = \xi \circ \tilde{F}$.

    Consequently, $\xi$ is a quasifibration with fiber $\Z{k}(N)$, and induces a long exact sequence
    \begin{equation*}
        \dotsb \to \pi_{i+1}(\Z{k}(M,N)) \to \pi_i(\Z{k}(N)) \to \pi_i(\Z{k}(M)) \to \pi_i(\Z{k}(M,N))
        \to \pi_{i-1}(\Z{k}(N)) \to \dotsb
    \end{equation*}
\end{lemma}
\begin{proof}
    Let $F'(x,t) = F(x,t) \res \cl{M \setminus N}$, which is also relatively piecewise smooth. Thus $\partial F'$, defined as the map $(x,t) \mapsto \partial F'(x,t) = \partial F(x,t) \res \cl{M \setminus N} + F(x,t) \res \partial\cl{M \setminus N}$, is also piecewise smooth, and has a standard gluing $\Phi$ by \cref{thm:PiecewiseSmoothGluing_Existence}. Then define
    \begin{equation*}
        \tilde{F}(x,t)
        = F'(x,t) - \Phi(\{x\} \times [0,t]) + F(x,0) - F'(x,0).
    \end{equation*}
    $\tilde{F}$ is continuous in the inductive limit topology. Note that by the definition of a gluing,
    \begin{align*}
        \partial(\Phi(\{x\} \times [0,t]))
        &= \Phi(\partial(\{x\} \times [0,t]))
        \\
        &= \Phi((x,t) - (x,0))
        \\
        &= \partial F'(x,t) - \partial F'(x,0)
        \\
        \implies
        \partial\tilde{F}(x,t)
        &= \partial F(x,0) = \partial(\xi(\tilde{f}(x))) = 0,
    \end{align*}
    thus $\tilde{F}$ is a map $X \to \Z{k}(M)$. And we also have $\xi \circ \tilde{F} = F$, because relative cycles in $\Z{k}(M,N)$ are defined up to chains supported in $N$, and the only summand of $\tilde{F}(x,t)$ that is not supported in $N$ is $F'(x,t)$, so $\xi(\tilde{F}(x,t)) = F'(x,t) + F(x,t) \res N = F(x,t)$.

    Now, we wish to show that the map $\xi_* : \pi_i(\Z{k}(M), \Z{k}(N)) \to \pi_i(\Z{k}(M,N))$ is an isomorphism. It suffices to base both homotopy groups at 0. First we show that it is surjective: for any homotopy class in $\pi_i(\Z{k}(M,N))$, choose a piecewise smooth representative $F : (\D^i, \partial\D^i) \to (\Z{k}(M,N),0)$, which is always possible due to \cref{thm:PiecewiseSmoothApprox}. (This representative can be chosen so that it passes through 0.) Since $\D^i = C\Sp{i-1}$, we may view $F$ as a nullhomotopy of a constant map $\Sp{i-1} \to \Z{k}(M,N)$ that sends $\Sp{i-1}$ to 0. This constant map clearly lifts to a constant map $\Sp{i-1} \to \Z{k}(M)$. Thus by the piecewise smooth homotopy lifting property proven earlier, $F$ lifts to a homotopy $\tilde{F} : \Sp{i-1} \times I \to \Z{k}(M)$ such that $\xi(\tilde{F}(x,1)) = 0$ for all $x$. That is, $\tilde{F}(x,1) \in \Z{k}(N)$ for all $x$, which give surjectivity.

    To prove injectivity, suppose that $\xi_*(\alpha) = 0$. Using \cref{thm:PiecewiseSmoothApprox}, choose a piecewise smooth representative $\tilde{f} : (\D^i, \partial\D^i) \to (\Z{k}(M), \Z{k}(N))$ of $\alpha$. Thus $\xi \circ \tilde{f}$ is nullhomotopic via a map $F : (\D^i, \partial\D^i) \times I \to \Z{k}(M,N)$ that sends $\partial\D^i \times I$ to 0. Since we can homotope $\xi \circ \tilde{f}$ to a piecewse smooth map, we may apply \cref{thm:PiecewiseSmoothApprox} to homotope relative to its boundary until it is piecewise smooth. Then $\tilde{f}$ lifts $F|_{(\D^i, \partial\D^i) \times \{0\}}$, so by the piecewise smooth homotopy lifting property, $F$ can be lifted to a nullhomotopy of $\tilde{f}$. Thus $\alpha = 0$.
\end{proof}

\begin{lemma}[Excision]
    \label{lem:Excision}
    The map $\Z{k}(\cl{M \setminus N}, \partial \cl{M \setminus N}) \hookrightarrow \Z{k}(M, N)$ is a weak homotopy equivalence.
\end{lemma}
\begin{proof}
    Consider the homomorphism $\phi : \pi_i(\Z{k}(\cl{M \setminus N}, \partial \cl{M \setminus N})) \to \pi_i(\Z{k}(M, N))$. We first show that it is surjective: any element in the codomain has a piecewise smooth representative $f : \Sp{i} \to \Z{k}(M, N)$ by \cref{thm:PiecewiseSmoothApprox}, and if we take its piecewise smooth structure and restrict every model relative cycle to $\cl{M \setminus N}$, then we get a piecewise smooth map $\Sp{i} \to \Z{k}(\cl{M \setminus N}, \partial \cl{M \setminus N})$ that is sent to $[f]$ by $\phi$.

    The injectivity proof is similar. Suppose that a piecewise smooth map $g : \Sp{i} \to \Z{k}(\cl{M \setminus N}, \partial \cl{M \setminus N})$ is sent to the zero class by $\phi$. Then $g$ extends to a map $G : \D^{i+1} \to \Z{k}(M, N)$. By \cref{thm:PiecewiseSmoothApprox}, we can choose $G$ to be piecewise smooth, and agreeing with the piecewise smooth structure of $g$ on $\partial\D^{i+1}$. Then as before, we restrict every model relative cycle to $\cl{M \setminus N}$, and obtain a map $\D^{i+1} \to \Z{k}(\cl{M \setminus N}, \partial \cl{M \setminus N})$. Thus $g$ is nullhomotopic.
\end{proof}

\begin{theorem}
    \label{thm:AlmgrenIsoFMetric_SpheresDisks}
    Consider any $i, k, n \geq 0$. Then the gluing homomorphism $\pi_i(\Z{k}(\Sp{n})) \to \hat{H}_{k+i}(\Sp{n})$ is an isomorphism. Moreover, if $n \geq 1$ then the gluing homomorphism $\pi_i(\Z{k}(\D^n, \partial\D^n)) \to \hat{H}_{k+i}(\D^n, \partial\D^n)$ is also an isomorphism.
\end{theorem}
\begin{proof}
    Let us begin by handling various edge cases. When $n = 0$, all of the groups in the theorem statement vanish unless $k = 0$, in which case $\Z0(\Sp0) = \{ax - ay~|~a \in \ZZ_p\}$ with the discrete topology where $\Sp0 = \{x,y\}$, and $\hat{H}_0(\Sp0) = \ZZ_p$, so the theorem holds.

    When $i = 0$ there are three cases:
    \begin{itemize}
        \item If $k > n$, then $\Z{k}(\Sp{n}) = \Z{k}(\D^n, \partial \D^n) = \{0\}$ so the theorem holds.

        \item If $k = n$, then $\Z{k}(\Sp{n})$ and $\Z{k}(\D^n, \partial \D^n)$ are both $\ZZ_p$ with the discrete topology, so the theorem also holds.

        \item If $k < n$, then it suffices to prove that $\Z{k}(\Sp{n})$ and $\Z{k}(\D^n, \partial \D^n)$ are both connected. For any $T \in \Z{k}(\Sp{n})$, the piecewise smooth approximation theorem gives a path connecting $T$ to a polyhedral cycle. Find an open ball in $\Sp{n}$ disjoint from that polyhedral cycle and dilate that ball until its complement is squashed into a point. Similarly for $\Z{k}(\D^n, \partial \D^n)$.
    \end{itemize}

    Now we prove the theorem by inducting on $n$; the base case of $n = 0$ has already been handled. Suppose that the theorem holds for some $n \geq 0$ (and all $i, k \geq 0$); we will prove it for $n + 1$. The case where $i = 0$ has been handled already. Consider any $k \geq 0$ and $i \geq 1$. Identify $\Sp{n}$ with $\partial\D^{n+1}$, and $\D^{n+1}$ with the lower hemisphere of $\Sp{n+1}$. Consider the following diagram, which may not commute \emph{a priori.} the isomorphisms in the upper row come from the long exact sequences in the labeled lemmas. The map labeled $\partial$ comes from a long exact sequence of reduced homology groups. The curved arrow arises from the suspension isomorphism. The columns are the gluing homomorphisms, among which the rightmost one is an isomorphism by the induction hypothesis. Let $\alpha$, $\beta$, and $\gamma$ be any choice of elements of the groups below them that correspond to each other via the isomorphisms.
    \begin{equation*}
    \begin{tikzcd}[column sep=large]
        \overset{\alpha}{\pi_i(\Z{k}(\Sp{n+1}))} \arrow[r,"\iso","\text{Lem.~\ref{lem:QuasifibLES}}"'] \dar & \pi_i(\Z{k}(\Sp{n+1},\D^{n+1}))  & \arrow[l,"\iso"',"\text{Lem.~\ref{lem:Excision}}"] \overset{\beta}{\pi_i(\Z{k}(\D^{n+1},\Sp{n}))} \arrow[r,"\iso","\text{Lem.~\ref{lem:QuasifibLES}}"'] \dar & \overset{\gamma}{\pi_{i-1}(\Z{k}(\Sp{n}))} \dar{\iso}
        \\
        \hat{H}_{k+i}(\Sp{n+1}) && \hat{H}_{k+i}(\D^{n+1}, \Sp{n}) \arrow[r,"\partial"',"\iso"] & \arrow[lll,"\text{suspension}","\iso"', bend left = 15] \hat{H}_{k+i-1}(\Sp{n})
    \end{tikzcd}
    \end{equation*}
    Thus the groups $\pi_i(\Z{k}(\Sp{n+1}))$, $\hat{H}_{k+i}(\Sp{n+1})$, $\pi_i(\Z{k}(\D^{n+1}, \Sp{n}))$, and $\hat{H}_{k+i}(\D^{n+1}, \Sp{n})$ are abstractly isomorphic to one another, but we need to prove that the corresponding gluing homomorphisms are isomorphisms. To do that, we prove that the outer ``rectangle'' and the inner square of the diagram commute.

    Suppose that $\gamma$ is represented by a piecewise smooth map $f : \Sp{i-1} \to \Z{k}(\Sp{n})$. Consider the map $F : C\Sp{i-1} \xrightarrow{Cf} C\Z{k}(\Sp{n}) \to \Z{k}(C\Sp{n}) = \Z{k}(\D^{n+1})$ where $C\Sp{n}$ is identified with $\D^{n+1}$. The proof of \cref{lem:QuasifibLES} shows that $\beta$ is represented by the map $F' : (\D^i, \Sp{i-1}) \to (\Z{k}(\D^{n+1}, \Sp{n}), 0)$ induced by $F$.

    The excision map sends $\beta = [F'] \in \pi_i(\Z{k}(\D^{n+1},\Sp{n}))$ to $[F'] \in \pi_i(\Z{k}(\Sp{n+1},\D^{n+1}))$. Thus $\alpha$ is in fact represented by the map $G : \Sp{i} = S\Sp{i-1} \xrightarrow{Sf} S\Z{k}(\Sp{n}) \to \Z{k}(S\Sp{n}) = \Z{k}(\Sp{n+1})$.  It can be verified that $F$, $F'$, and $G$ are all piecewise smooth. Thus the gluing homomorphism sends $\alpha$ to $G_*[\Sp{i}]$ and $\gamma$ to $f_*[\Sp{i-1}]$. \Cref{prop:GluingProperties}(\ref{enum:GluingProperties_Suspension}) implies that $G_*[\Sp{i}]$ and $f_*[\Sp{i-1}]$ are related by the suspension isomorphism, so the  outer rectangle of the diagram commutes.

    Now observe that the gluing homomorphism sends $\beta$ to $F'_*[\D^i]$. By definition, $\partial(F'_*[\D^i])$ is computed by taking an absolute chain that represents $F'_*[\D^i]$ and returning the homology class of its boundary. Let $\Phi$ be a gluing for $F$. Then if we identify $C\Sp{i-1}$ with $\D^i$, then $F'$ has a gluing $\Phi'$ which is simply $\Phi$ composed with the map $\Z{\bullet}(\D^{n+1}) \to \Z{\bullet}(\D^{n+1}, \Sp{n})$, because this map commutes with the boundary map. Therefore $\partial(F'_*[\D^i]) = \partial(\Phi'_*[\D^i])$ is the class, in $\hat{H}_{k+i-1}(\Sp{n})$, of $\partial\Phi[\D^i] = \Phi\partial[\D^i] = \Phi[\Sp{i-1}]$. Let $j : \Sp{n} \hookrightarrow \D^{n+1}$ be the inclusion of the boundary. The class of $\Phi[\Sp{i-1}]$ in $\hat{H}_{k+i-1}(\Sp{n})$ is $\Phi_*[\Sp{i-1}] = F_*[\Sp{i-1}] = F_*j_*[\Sp{i-1}]$ which by the composition law in \cref{prop:GluingProperties} is $(F \circ j)_*[\Sp{i-1}] = f_*[\Sp{i-1}]$. Thus the inner square of the diagram commutes.
\end{proof}

\begin{proof}[Proof of \cref{Thm:AlmgrenIsoFMetric}]
    Find a triangulation of $M$. Our proof will use the relative cellular homology of $M$, which is the homology of the relative cellular chain complex $\hat{H}_*(M^d, M^{d-1} \cup (\partial M)^d)$. There exists a sequence $\N_d$ of compact tubular neighbourhoods in $M$ of $M^d$, with increasing radii, such that each $\N_d \setminus \N_{d-1}$ is a disjoint union of ``thickenings'' of the $d$-cells of $M$. Similarly, let $\N^\partial_d$ be the tubular neighbourhood of the $d$-skeleton of $\partial M$ with the same radius as $\N_d$. Observe that $Y_d = \N_0 \cup \dotsb \cup \N_d$ deformation retracts onto $M^d$, $Y^\partial_d = \N^\partial_0 \cup \dotsb \cup \N^\partial_d$ deformation retracts onto $(\partial M)^d$, and that
    \begin{equation}
        \hat{H}_j(Y_d, Y_{d-1} \cup Y^\partial_d) \iso \hat{H}_j(M^d, M^{d-1} \cup (\partial M)^d).
    \end{equation}

    Note that $(Y_d, Y_{d-1} \cup Y^\partial_d)$ is a collar pair. Let us first prove that the gluing homomorphism,
    \begin{equation}
        \label{eq:AlmgrenIsoFMetric_RelCellChainCplx_GluingHom}
        \pi_i(\Z{k}(Y_d, Y_{d-1} \cup Y^\partial_d)) \to \hat{H}_{k+i}(Y_d, Y_{d-1} \cup Y^\partial_d),
    \end{equation}
    is an isomorphism. By \cref{lem:Excision}, the domain is isomorphic to $\pi_i(\Z{k}(\coprod_l e^d_l, \coprod_l \partial e^d_l))$, where the $e^d_l$'s are the $d$-cells of $M$ that are not contained in $\partial M$. The codomain is isomorphic to $\bigoplus_l \hat{H}_{k+i}(e^d_l, \partial e^d_l)$. Thus we can apply \cref{thm:AlmgrenIsoFMetric_SpheresDisks} to show that \cref{eq:AlmgrenIsoFMetric_RelCellChainCplx_GluingHom} is an isomorphism. After this, it is a matter of diagram chasing to prove the theorem.
\end{proof}

\begin{theorem}[Brown Representability for Spaces of Cycles]
    \label{thm:BrownRepSpacesOfCycles}
    For each cell complex $X$, there is an isomorphism $[X, \Z{k}(\Sp{n})] \xrightarrow{\iso} \hat{H}^{n-k}(X, \partial X)$, where the homotopy class of a map $f$ that has a gluing corresponds to the cohomology class that evaluates on a homology class $\alpha \in \hat{H}_{n-k}(X, \partial X)$ to $f_*(\alpha) \in \hat{H}_n(\Sp{n}) \iso \ZZ_p$.
    
    There is also an analogous isomorphism $[X, \Z{k}(\D^n, \partial\D^n)] \xrightarrow{\iso} \hat{H}^{n-k}(X)$, whose definition is the same except with $H_n(\D^n, \partial\D^n)$ in the place of $H_n(\Sp{n})$.
\end{theorem}
\begin{proof}
    We will prove the theorem for $\Z{k}(\Sp{n})$ as the other case follows similarly. $\Z{k}(\Sp{n})$ is weakly homotopy equivalent to the Eilenberg-Maclane space $K(\ZZ_p, n-k)$, due to \cref{thm:AlmgrenIsoFMetric_SpheresDisks}. Hence the Brown Representability Theorem and standard arguments from obstruction theory imply that there is an isomorphism $[X, \Z{k}(\Sp{n})] \xrightarrow{\iso} H^{n-k}(X)$ that sends $[f]$ to $f^*\iota_{n-k}$. We fix $\iota_{n-k}$ by first fixing an isomorphism $\pi_{n-k}(\Z{k}(\Sp{n})) \iso \hat{H}_n(\Sp{n}) \iso \ZZ_p$, where the first isomorphism comes from \cref{thm:AlmgrenIsoFMetric_SpheresDisks} and the second from identifying the fundamental class associated to the standard orientation with 1. This fixes a generator of $\pi_{n-k}(\Z{k}(\Sp{n}))$ corresponding to 1, which corresponds to a generator of $H_{n-k}(\Z{k}(\Sp{n}))$ by the Hurewicz theorem. The dual of this homology generator is $\iota_{n-k}$.

    Now consider some $f : X \to \Z{k}(\Sp{n})$ and homology class $\alpha \in \hat{H}_{n-k}(X)$. By the preceding choice of $\iota_{n-k}$, to evaluate $f^*\iota_{n-k}$ on $\alpha$, we have to push $\alpha$ forward using $f$ to a homology class in $H_{n-k}(\Z{k}(\Sp{n}))$, find its preimage in $\pi_{n-k}(\Z{k}(\Sp{n}))$ over the Hurewicz isomorphism, and then apply the gluing homomorphism to get a homology class in $\hat{H}_n(\Sp{n}) \iso \ZZ_p$.
    
    By obstruction theory, we may assume that we have homotoped $f$ until it vanishes on $X^{n-k-1}$. By the relative piecewise smooth appoximation theorem, $f$ is also piecewise smooth with respect to some refinement $\tilde{X}$ of $X$, so it has a gluing $\Phi$. Each homology class $\alpha \in \hat{H}_{n-k}(X)$ can be represented by a cellular cycle $A$. Since $f$ vanishes over the boundary of each summand cell $e^{n-k}$ of $A$, $f$ defines an element of $\beta \in \pi_{n-k}(\Z{k}(\Sp{n}))$ which is precisely the sum of the images of $e^{n-k}$ with their respective multiplicities in $A$. $\beta$ precisely corresponds to the pushforward of $\alpha$ along $f$ under the Hurewicz isomorphism. Then applying the gluing homomorphism we get $[\Phi(A)] = f_*(\alpha)$.
\end{proof}

\subsection{Gluings as Brown representatives}

Since gluings give maps between spaces of cycles, it is natural to ask what cohomology classes they represent.

\begin{proposition}
    \label{prop:GluingAsBrownRep}
    Let $X$ be a compact Riemannian manifold, and let $\alpha_{i1}, \dotsc, \alpha_{id_i}$ be a basis for $H_i(X, \partial X)$. Then for any $k \geq 0$, $\Z{k}(X, \partial X) \whe \prod_{i > k} \prod_{j = 1}^{d_i} K_{i-k,j}$ where each $K_{mj}$ is a copy of $K(\ZZ_p, m)$.
    
    In addition, let $\iota_{mj} \in H^m(\Z{k}(X, \partial X))$ correspond to the fundamental cohomology class of $K_{mj}$. Suppose that a class $\gamma \in H^i(X, \partial X)$ has a piecewise smooth Brown representative $f : X \to \Z{l}(\D^{i+l}, \partial\D^{i+l})$. Then if $\Phi$ is the standard gluing of $f$, then $\Phi^k : \Z{k}(X, \partial X) \to \Z{l+k}(\D^{i+l}, \partial\D^{i+l})$ is a Brown representative for $\gamma(\alpha_{i1})\iota_{i-k,1} + \dotsb + \gamma(\alpha_{id_i})\iota_{i-k,d_i}$.
\end{proposition}

The real content of this proposition is that $\Phi^k$ is a Brown representative for a cohomology class that only contains fundamental cohomology classes as summands, and not any other classes, such as cross products or Steenrod powers of some classes in lower degrees. In this sense, gluings by themselves cannot represent ``interesting'' cohomology classes.

\begin{proof}    
    For our purposes it is enough to prove this for $X = L_n \times \D^{n+1}$ as this is the only case we need for the proof of \cref{thm:SteenrodPowers}. The general case will follow similarly, just with more cumbersome notation.

    In this case, \cref{Thm:AlmgrenIsoFMetric} and the Thom Isomorphism Theorem imply that $\Z{k}(X, \partial X) \whe \prod_{j = n+1-k}^{p(n+1)-1-k} K(\ZZ_p, j)$. Thus $\Phi^k$ is a Brown representative of a class $\phi$ which is the $\ZZ_p$-linear combination of cross products of cup products of mod $p$ cohomology operations applied to $\iota_j$'s for $j \leq i-k$.

    Our main task is to show that $\phi$ is simply a multiple of $\iota_{i-k}$ (after which the precise multiplicative factor may be computed by a simple evaluation), and that there are no more complicated terms. To prove this, we ``restrict $f$ to $\prod_{j = n+1-k}^q K(\ZZ_p, j)$'' where $q = i-k-1$ and show that the result is a Brown representative for the zero class. Consider a triangulation of the polyhedral structure of $X$ that comes from the piecewise smooth structure of $f$, and let $\N_{i-1}$ denote a tubular neighbourhood of $\cl{X^{i-1} \setminus \partial X}$ in $X$. Then consider the inclusion
    \begin{equation*}
        \lambda : \Z{k}(\N_{i-1}, \N_{i-1} \cap \partial X) \hookrightarrow \Z{k}(X, \partial X).
    \end{equation*}
    \Cref{Thm:AlmgrenIsoFMetric} implies that $\pi_j(\Z{k}(\N_{i-1}, \N_{i-1} \cap \partial X)) \iso H_{k+j}(X^{i-1}, (\partial X)^{i-1})$. Thus $\lambda$ induces a surjection on homotopy groups of degree at most $q$. By considering $\Z{k}(\N_{i-1}, \N_{i-1} \cap \partial X)$ as weakly homotopy equivalent to a product of Eilenberg-Maclane spaces, one for each generator of a homotopy group, we see that the $\lambda$ also induces surjections on homology groups of degree at most $q$. Thus $\lambda$ induces injections on cohomology groups of degree at most $q$, and it suffices to check that $\Phi^k \circ \lambda$ is nullhomotopic.

    By the deformation retract of $\N_{i-1}$ onto $X^{i-1}$, $\lambda$ is homotopic to a map $\lambda'$ such that every $\lambda'(T)$ is supported in $X^{i-1}$. However, if the $(i-1)$-cells of $X$ are denoted by $e^{i-1}_r$, then the support of every $\Phi^k(\lambda'(T))$ is inside a $(i-1+l)$-dimensional set (because $\Phi$ is a standard gluing), which cannot cover the entirety of $\D^{i+l}$. This allows us to nullhomotope $\Phi^k \circ \lambda'$, and also $\Phi^k \circ \lambda$.
\end{proof}

\section{Brown Representatives of some Cohomology Classes and Operations}
\label{sec:BrownRepCohomClassesOps}

\subsection{Representing the cohomology generators of projective and lens spaces}

Let $n \geq 1$. When $p = 2$, let $\omega_k$ denote the generator of $H^k(\RP^n)$ for all $k \leq n$. When $p > 2$ and $L$ is an $n$-dimensional lens space for $n \leq \infty$, let $\omega_k$ denote the generator of $H^k(L)$ for all $k \leq n$, as chosen in \cref{sec:LensSpaceCohomConvention}. We will represent points of lens spaces as sets of $p$ points, which are the fibers of the covering map $\Sp{n} \to L$. Let $\vspan_\C$ denote the complex span.

\begin{lemma}
    \label{lem:ProjLensCohomGens}
    Let $n \geq 1$.
    \begin{enumerate}
        \item\label{enum:ProjLensCohomGens_RPn} When $p = 2$, the map $\RP^n \to \Z{n}(\D^{n+1}, \partial\D^{n+1})$ that is defined by $\ell \mapsto \ell^\perp \res \D^{n+1}$ (see \cref{eg:RPn_PiecewiseSmooth}) is a Brown representative for $\omega_1$.

        \item\label{enum:ProjLensCohomGens_CPn} When $p > 2$, the map $\CP^n \to \Z{2n}(\D^{2n+2}, \partial \D^{2n+2})$ that is defined by $\ell \mapsto \ell^{\perp_\C} \res \D^{2n+2}$ (see \cref{eg:CPn_PiecewiseSmooth}) is a Brown representative for the generator of $H^2(\CP^n)$ that evaluates to 1 on the 2-cell.

        \item\label{enum:ProjLensCohomGens_Lens_H2} When $p > 2$, the map $\Lens_p^{2n-1} \to \Z{2n-2}(\D^{2n}, \partial\D^{2n})$ that is defined by $\{x_1, \dotsc, x_p\} \mapsto \vspan_\C\{x_1\}^{\perp_\C} \res \D^{2n}$ is a Brown representative for $\omega_2$.

        \item\label{enum:ProjLensCohomGens_Lens_H1} When $p > 2$, the map $\Lens_p^{2n-1} \to \Z{2n-1}(\D^{2n},\partial \D^{2n})$ that is defined in \cref{eg:Lens_PiecewiseSmooth} is a Brown representative for $\omega_1$.
    \end{enumerate}
\end{lemma}
\begin{proof}
    To prove (\ref{enum:ProjLensCohomGens_RPn}), it suffices to evaluate the cohomology class represented by the map $\ell \mapsto \ell^\perp \cap \D^{n+1}$ on the homology class represented by $\RP^1 \subset \RP^n$, where $\RP^1$ consists of the lines through the origin in $\RR^2 \times \{0\}^{n-1}$. Observe that a generic point in $\D^{n+1}$ lies $\ell^\perp$ for exactly one $\ell \in \RP^1$. For this reason, the disks $\ell^\perp \cap \D^{n+1}$ glues into the relative fundamental class of $\D^{n+1}$. (\ref{enum:ProjLensCohomGens_CPn}) follows similarly.

    (\ref{enum:ProjLensCohomGens_Lens_H2}) follows from the fact that the map $\Lens_p^{2n-1} \to \CP^n$ sending $\{x_1, \dotsc, x_p\}$ to $\vspan_\C\{x_1\}$ induces an isomorphism on $H^2$. This is a well-known fact that follows from the Serre spectral sequence of the fiber bundle $S^1 \to \Lens_p^{2n-1} \to \CP^n$.

    To prove (\ref{enum:ProjLensCohomGens_Lens_H1}), let $f$ denote the map in \cref{eg:Lens_PiecewiseSmooth}. Note that the 1-skeleton of $\Lens_p^{2n-1}$ is $\Sp1$, which we may parametrize as an arc $\{(e^{i\theta},0,\dotsc,0)~|~\theta \in [0,2\pi/p]\} \subset \Lens_p^{2n-1}$ (see \cite[p.~144]{Hatcher_AlgTop}). Then $f_*[\Sp1]$ is precisely the fundamental cycle of $\D^{2n}$, because a generic point in $\C^{d-1}$ (with norm less than 1) lies in exactly one of the mod $p$ cycles $f(x)$ for $x \in \Sp1$. Intuitively, the cycles $f(x)$ for $x \in \Sp1$ rotate about $(\C \times \{0\}^{d-1})^{\perp_\C} = \{0\} \times \C^{d-1}$ and sweep out $\D^{2n}$ exactly once.
\end{proof}

\begin{lemma}
    Consider the general lens space $L = \Lens_p(\ell_1,\dotsc,\ell_n)$, where each $\ell_j$ is a nonzero element of $\ZZ_p$. Let $\ell_j^{-1}$ denote the multiplicative inverse of $\ell_j$. Then the map $L \to \Lens_p^{2n-1}$ defined by
    \begin{equation*}
        (r_1 e^{i\theta_1}, \dotsc, r_n e^{i\theta_n}) \mapsto (r_1 e^{i\ell_1^{-1}\theta_1}, \dotsc, r_n e^{i\ell_n^{-1}\theta_n}),
    \end{equation*}
    composed with the maps from \cref{lem:ProjLensCohomGens}(\ref{enum:ProjLensCohomGens_Lens_H2}) and (\ref{enum:ProjLensCohomGens_Lens_H1}) give Brown representatives $\omega_2$ and $\omega_1$ respectively.
\end{lemma}
\begin{proof}
    By \cite[p.~310]{Hatcher_AlgTop}, the map induces an isomorphism on $\pi_1(-)$, and since the Bockstein homomorphisms $H^1(-) \to H^2(-)$ are isomorphisms for $L$ and $\Lens_p^{2n-1}$, it is also an isomorphism on $H^2$.
\end{proof}

\begin{lemma}
    \label{lem:ThomSpace_BrownRep}
    Let $X$ be a closed manifold and $X \times \D^q$ be a trivial disk bundle over $X$. If some $\alpha \in H^{k-n}(X)$ has a Brown representative $a : X \to \Z{k}(\D^n, \partial\D^n)$, then the the class in the Thom space $(X \times \D^n)/(X \times \partial\D^n)$ that corresponds to $\alpha$ via the Thom isomorphism has the following Brown representative:
    \begin{align*}
        &f : X \times \D^n \to \Z{k}((\D^n \times \D^q), \partial(\D^n \times \D^q))
        \\
        &f(x,v) = a(x) \times \{v\}.
    \end{align*}
\end{lemma}
\begin{proof}
    This follows from \cref{thm:BrownRepSpacesOfCycles}, and the fact that evaluating this $f$ on a homology class involves gluing it over the product of a cycle in $X$ and a copy of $\D^n$.
\end{proof}

\subsection{The cohomology pushforward and the K\"unneth formula}

\begin{proposition}
    \label{prop:GeomRep_PushforwardCohom}
    Consider a fiber bundle $M \to E \xrightarrow{\pi} B$ where $M$ is a closed Riemannian $m$-manifold, and $E$, and $B$ are compact Riemannian manifolds of dimensions $\dim E = e$ and $\dim B = b$ respectively such that $\partial E = \pi^{-1}(\partial B)$. Suppose that $\partial B$ is the disjoint union of $\partial_- B$ and $\partial_+ B$, either of which may be empty. Let $\partial_\pm E = \pi^{-1}(\partial_\pm B)$. Then the pushforward map on $n^\text{th}$ cohomology, where $0 \leq n \leq e$
    \begin{equation*}
        \pi_* = H^n(E,\partial_+ E) \xrightarrow{\text{Poincar\'e duality}} H_{e-n}(E, \partial_- E) \xrightarrow{\pi_*} H_{e-n}(B, \partial_- B) \xrightarrow{\text{Poincar\'e duality}} H^{n - m}(B, \partial_+ B),
    \end{equation*}
    sends a relative cohomology class represented by a piecewise smooth map $f : (E, \partial_+ E) \to (\Z{k}(\Sp{n+k}),0)$ to the class represented by
    \begin{equation}
        \pi_*f := (B, \partial_+ B) \xrightarrow{\pi^{-1}(-)} (\Z{m}(E),\Z{m}(\partial_+ E)) \xrightarrow{f_\sharp} (\Z{k+m}(\Sp{n+k}),0).
    \end{equation}
\end{proposition}
\begin{proof}
    Let us first establish the adjointness property: for any piecewise smooth map $g : B \to \Z0(\Sp{b - (n-m)})$,
    \begin{align*}
        (\pi_*f \wedge g)_* [B] 
        &= \big( x \mapsto f_\sharp(\pi^{-1}(x)) \wedge g(x) \big)_* [B]
        \\ 
        \eqnote{wedge sum law from \cref{prop:GluingProperties}} &= \big(x \mapsto (f \wedge (g \circ \pi))_\sharp(\pi^{-1}(x)) \big)_* [B]
        \\
        &= ((f \wedge (g \circ \pi))_\sharp \circ \pi^{-1})_* [B]
        \\
        \eqnote{composition law from \cref{prop:GluingProperties}} &= (f \wedge (g \circ \pi))_* \circ \pi^{-1}_* [B]
        \\
        \eqnote{\cref{eg:FiberBundle_Gluing}} &= (f \wedge (g \circ \pi))_*[E].
    \end{align*}
    
    The fact that $\pi_*$ represents the cohomology pushforward can be proven from this adjointness relation, by noting that the wedge product here represents cup product, relating the cup product to the cap product and thus Poincar\'e duality, and then performing algebraic manipulations.
\end{proof}

\begin{corollary}
    \label{cor:KunnethFormula}
    Let $X$ be a closed manifold and $Y$ be a manifold with boundary, and let $f : X \times Y \to \Z{k}(\Sp{n})$ be a piecewise smooth map that represents $\gamma = \bigoplus_{i + j = n - k} \alpha_i \otimes \beta_j$, where $\alpha_i \in H^i(X)$ and $\beta_j \in H^j(Y)$. Then for any closed $d$-dimensional submanifold $A \subset X$, $\alpha_d([A])\beta_{n-k-d} \in H^{n-k-d}(Y)$ is represented by the map
    \begin{align*}
        Y &\to \Z{k+d}(\Sp{n})
        \\
        y &\mapsto f_\sharp(A \times \{y\}).
    \end{align*}
\end{corollary}
\begin{proof}
    Observe that the above map is $\pi_*f|_{A \times Y}$, where $\pi : A \times Y \to Y$ is the projection onto the second factor. Let the classes $\alpha_i$ pull back to $\alpha' \in H_i(A)$ along the inclusion $A \hookrightarrow X$. Then $f|_{A \times Y}$ represents $\bigoplus_{i = 0}^d \alpha_i' \otimes \beta_{n-k-i}$, so \cref{prop:GeomRep_PushforwardCohom} implies that $\pi_*f|_{A \times Y}$ represents the cohomology pushforward of $\bigoplus_{i = 0}^d \alpha_i' \otimes \beta_{n-k-i}$. That is, take the Poincar\'e dual to $\bigoplus_{i = 0}^d \overline{\alpha_i'} \otimes \bar\beta_{n-k-i}$ (bars indicate Poincar\'e duals), push forward via $\pi_*$ to get $\overline{\alpha_d'}\bar\beta_{n-k-d} = \alpha_d'[A]\bar\beta_{n-k-d} = \alpha_d[A]\bar\beta_{n-k-d}$, then take the Poincar\'e dual again to get $\alpha_d[A]\beta_{n-k-d}$.
\end{proof}

\subsection{Brown representatives for the cross product and cup product}

Let $(M,N)$ and $(M',N')$ be collar pairs. In this section we will show that the Cartesian product of cycles gives a Brown representative for the cross product. However, to formalize this we need to decide on the topology on products of the form $\Z{k}(M,N) \times \Z{k'}(M',N')$. A product of this form has two natural topologies: the product topology of the two inductive limit topologies, and the inductive limit topology with respect to the family of sets $A \times B$, where $A$ is mass-bounded in $\Z{k}(M,N)$ and $B$ is mass-bounded in $\Z{k'}(M',N')$. We call the former the \emph{product topology} and the latter the \emph{inductive limit topology}. The product topology is finer than the inductive limit topology \cite[2.6]{Cohen_WeakTop}.

Let us show that these topologies are weakly homotopy equivalent. The following remark will be helpful:

\begin{remark}
    \label{rem:ProdIndLimTopEquivOnMassBounded}
    On the space $\Z{k}(M,N)^\mu \times \Z{k'}(M',N')^{\mu'}$ for any $\mu, \mu' > 0$, the topolgoies induced as a subspace of the product topology and as a subspace of the inductive limit topology agree with the topology as a subspace of the product of the flat topologies.
\end{remark}

\begin{lemma}
    \label{lem:ProdIndLimTopsEquiv}
    The product and inductive limit topologies on $\Z{k}(M,N) \times \Z{k'}(M',N')$ are weakly homotopy equivalent.
\end{lemma}
\begin{proof}
    We wish to prove that the relevant homomorphisms between homotopy groups, going from the product topology to the inductive limit topology, are isomorphisms. First we show surjectivity: for any map $f : \Sp{i} \to \Z{k}(M,N) \times \Z{k'}(M',N')$ that is continuous in the inductive limit topology, its image is compact so by adapting \cite[(1.9)]{Almgren_HomotopyGroupsIntegralCycles} we see that its image must lie in $\Z{k}(M,N)^\mu \times \Z{k'}(M',N')^{\mu'}$ for some $\mu, \mu' > 0$. By \cref{rem:ProdIndLimTopEquivOnMassBounded}, $f$ is also continuous in the product topology. A similar argument shows injectivity. 
\end{proof}

Henceforth, we will give the inductive limit topology to finite products of the form $\prod_{i = 1}^j\Z{k_i}(M_i,N_i)$. By \cref{lem:ProdIndLimTopsEquiv}, we can compute their cohomology using the K\"unneth formula.

\begin{lemma}
    \label{lem:CartProdContinuous}
    The map $c : \Z{k}(M,N) \times \Z{k'}(M',N') \to \Z{k+k'}(M \times M', M \times N' \cup N \times M')$ where $c(S,T) = S \times T$ is continuous.
\end{lemma}
\begin{proof}
    We will adapt the sufficient condition for continuity in \cref{lem:ContinuityIndLimTop}. Suppose that for some $S_i \in \Z{k}(M,N)$ and $T_i \in \Z{k}(M',N')$ for $i = 1,2$, and $\mu, \mu' > 0$ we have $\M(S_i) < \mu$ and $\M(T_i) < \mu'$. Note that $\M(S_i \times T_i) < \mu\mu'$. The triangle inequality implies that
    \begin{equation*}
        \Fl(S_0 \times T_0, S_1 \times T_1) \leq \Fl(S_0(T_0 - T_1)) + \Fl((S_0 - S_1)T_1) \leq \mu\Fl(T_0,T_1) + \mu'\Fl(S_0, S_1).
    \end{equation*}
    This implies that $c : \Z{k}(M,N)^\mu \times \Z{k'}(M',N')^{\mu'} \to \Z{k+k'}(M \times M', M \times N' \cup N \times M')^{\mu\mu'}$ is continuous in the flat topology in the codomain, and the product topology of the flat topologies in the domain, for all $\mu, \mu' > 0$. The lemma statement follows.
\end{proof}

Let $\smsh$ denote the smash product, and let $(-)^{\smsh k}$ denote the $k$-fold smash power.

\begin{proposition}
    \label{prop:CrossProduct_BrownRep}
    The map $f : \Z{k}(\D^n, \partial\D^n)^m \to \Z{mk}((\D^n)^m, \partial(\D^n)^m)$ is a Brown representative for the $m$-fold cross product. In other words, $f^*(\iota_{m(n-k)}) = \iota_{n-k}^{\otimes m}$.
    
    Likewise, the map $g : \Z{k}(\Sp{n})^m \to \Z{mk}(\Sp{mn})$ defined by $g(T_1, \dotsc, T_m) = T_1 \smsh \dotsb \smsh T_m$ is also a Brown representative for the $m$-fold cross product.
\end{proposition}
\begin{proof}
    We will prove this for $f$, as the proof is similar for $g$. The continuity of $f$ follows from \cref{lem:CartProdContinuous}. \emph{A priori}, $f^*(\iota_{m(n-k)})$ is a class in degree $m(n-k)$, so by the K\"unneth formula it must be the sum of classes of the form $\gamma_1 \otimes \dotsb \otimes \gamma_m$. Let us first show that each $\gamma_i$ must have positive degree: if not, suppose without loss of generality that $\gamma_1$ has degree 0. Then we may evaluate $f^*(\iota_{m(n-k)})$ on a cycle $A_1 \times \dotsb \times A_m$ where $A_i \subset \Z{k}(\D^n, \partial\D^n)$, but $A_1$ is a linear combination of cycles, which we may assume is the zero cycle. Clearly $f(A_1 \times \dotsb \times A_m) = 0$ by the formula for $f$. 

    Thus each $\gamma_i$ has degree at least $n-k$, which forces $f^*(\iota_{m(n-k)})$ to be $a\iota_{n-k}^{\otimes m}$ for some $a \neq 0$. To compute $a$, consider the canonical map $h : \Sp{n-k} \to \Z{k}(\D^n, \partial\D^n)$ defined by $h(x) = (\{x\} \times \D^k) \cap \D^n$. It can be checked using \cref{thm:BrownRepSpacesOfCycles} that $h^*(\iota_{n-k})$ is a generator $\alpha \in H^{n-k}(\Sp{n-k})$. In fact, if $h^m$ denotes the induced map $(\Sp{n-k})^m \to \Z{k}(\Sp{n})^m$, then it can be verified using \cref{thm:BrownRepSpacesOfCycles} that $f \circ h^m$ is a Brown representative for $\alpha^{\otimes m}$. Thus $a = 1$.
\end{proof}

\begin{lemma}
    \label{lem:DiagMapContinuous}
    The diagonal map $\Delta : \Z{k}(M,N) \hookrightarrow \Z{k}(M,N) \times \Z{k}(M,N)$, defined by $T \mapsto (T,T)$, is continuous.
\end{lemma}
\begin{proof}
    For any $T \in \Z{k}(M,N)$ and $\mu > \M(T)$, $(T,T) \in \Z{k}(M,N)^\mu \times \Z{k}(M,N)^\mu$. Clearly the diagonal map $\Z{k}(M,N)^\mu \hookrightarrow \Z{k}(M,N)^\mu \times \Z{k}(M,N)^\mu$ is continuous in the flat metric. By \cref{rem:ProdIndLimTopEquivOnMassBounded} the diagonal map $\Z{k}(M,N)^\mu \hookrightarrow \Z{k}(M,N) \times \Z{k}(M,N)$ is continuous with respect to the flat topology in the domain and the inductive limit topology in the codomain. Thus by \cref{lem:ContinuityIndLimTop}, the diagonal map $\Z{k}(M,N) \hookrightarrow \Z{k}(M,N) \times \Z{k}(M,N)$ is continuous.
\end{proof}

\begin{corollary}
    \label{cor:CupPower_BrownRep}
    The following maps are Brown representatives for the $p$-fold cup power.
    \begin{align*}
        \Z{k}(\D^n, \partial\D^n) &\to \Z{pk}((\D^n)^p, \partial(\D^n)^p)
        &
        \Z{k}(\Sp{n})^m &\to \Z{mk}(\Sp{mn})
        \\
        T &\mapsto T^p & T &\mapsto T^{\smsh p}
    \end{align*}
\end{corollary}
\begin{proof}
    This follows from \cref{lem:DiagMapContinuous} and the definition of the cup product in terms of the cross product and the diagonal map \cite[p.~279]{Hatcher_AlgTop}.
\end{proof}

\begin{corollary}
    \label{cor:CupProd_BrownRep}
    Suppose that for $i = 1, \dotsb, q$, $f_i : X \to \Z{k_i}(\D^{n_i}, \partial\D^{n_i})$ is a Brown representative for $\alpha_i \in H^{n_i-k_i}(X)$. Then the class $\alpha_1 \smile \dotsb \smile \alpha_q$ has a Brown representative which is the following composite map:
    \begin{equation*}
        X \xrightarrow{\Delta} X^q \xrightarrow{f_1 \times \dotsb \times f_q} \prod_i \Z{k_i}(\D^{n_i}, \partial\D^{n_i}) \xrightarrow{c} \Z{k_1 + \dotsb + k_q}(\textstyle\prod_i\D^{n_i}, \partial(\prod_i\D^{n_i})),
    \end{equation*}
    where $\Delta(x) = (x,\dotsc,x)$ and $c(T_1,\dotsc,T_q) = T_1 \times \dotsb \times T_q$.
\end{corollary}

\begin{remark}
    \label{rem:ProjLensSpacesCohom_BrownReps}
    \Cref{cor:CupProd_BrownRep} can be combined with \cref{lem:ProjLensCohomGens,lem:ThomSpace_BrownRep} to obtain Brown representatives of every cohomology class in any real or complex projective space and any lens space, by using the cup product structure on their cohomology rings.
\end{remark}

\section{The Cyclic Product Map and the Bockstein Homomorphisms}
\label{sec:CyclicProductMap}

Let us define the map $\cyc : \Z{k}(\Sp{n}) \to \Z{pk}(L_n \times \D^{n+1}, L_n \times \partial\D^{n+1})$ from \cref{thm:SteenrodPowers} formally. First let us define it over \emph{polyhedral cycles}, i.e. the cycles formed from linear combinations of the faces of a fine cubical structure on $\Sp{n}$. Consider the diagonal $\Delta = \{(x,\dots,x)~|~x \in \Sp{n}\} \in \Sp{p(n+1)-1}$ and let $\Delta_\varepsilon$ denote the open $\epsilon$-neighbourhood of $\Delta$.

We will need the following ``equivariant isoperimetric inequality'':

\begin{theorem}
    \label{thm:InvariantFilling}
    For any $0 \leq k \leq p(n+1)-3$, let $R \in \Z{k}(\Sp{p(n+1)-1})$ be a $\ZZ_p$-invariant cycle that has an invariant filling and a filling $Q$ of mass $\varepsilon = \M(Q) \leq 1$. Then for any $m > 4pn$, there exists a $\ZZ_p$-invariant chain $S$ in $\Sp{p(n+1)-1}$ such that
    \begin{align*}
        \M(S) \leq C\varepsilon^{\frac{m-2pn}m} && \text{and} && \M(R - \partial S) \leq C\varepsilon^{\frac{m-2pn}m} + \M(R \res \Delta_{2pn\varepsilon^{1/m}}).
    \end{align*}
\end{theorem}
\begin{proof}
    Let $F$ be the uncontrolled invariant filling. Since $F - Q$ is a $(k+1)$-cycle and $k + 1 < p(n+1)-1$, it can be filled by a $(k+2)$-chain $G$. Using the coarea inequality, choose some $r \in [2pn\varepsilon^{1/m}, (2pn+1)\varepsilon^{1/m}]$ such that $\M(Q \res \partial\Delta_r) \leq C\varepsilon^{\frac{m-1}m}$.
    
    Take the standard $\ZZ_p$-invariant cell structure on $\Sp{(p-1)(n+1)-1}$, find the product cell structure on $\Sp{(p-1)(n+1)-1} \times \D^{n+1}$. Let $X_d$ be the preimage over $f$ of the $d$-skeleton of $\Sp{(p-1)(n+1)-1} \times \D^{n+1}$. Given any sequence of distances $r_{n+1}, \dotsc, r_{p(n+1)-1}$, we can define a $\ZZ_p$-invariant filtration
    \begin{equation}
        \label{eq:MakeFillingInvariant_SkeletalFiltration}
        \emptyset \subset Y_{n+1} \subset Y_{n+2} \subset \dotsb \subset Y_{p(n+1)-1} = \Sp{p(n+1)-1}_r, \qquad\text{where } Y_i = \cl{\bigcup_{j = n+1}^i N_{r_j}(X_j) \setminus \Delta_r}.
    \end{equation}
    
    We will modify $Q$ in each successively higher elements of the filtration. More precisely, we will inductively construct, for each $n+1 \leq d \leq p(n+1)-1$, a distance $r_d$ and a $\ZZ_p$-invariant $(k+1)$-chain $Q_d$ such that:
    \begin{itemize}
        \item $r_d \in [(2pn - i)\varepsilon^{1/m}, (2pn - d+1)\varepsilon^{1/m}]$.

        \item $Q_d$ is a relative filling of $R \res Y_d$. That is, $Q_d$ is supported in $Y_d$ and $\partial Q_d - R \res Y_d$ is suported in $\partial Y_d$.

        \item $Q_d$ is relatively homologous to $F \res Y_d$. That is, there is some chain $G_d$ such that $\partial G_d + Q_d - F \res Y_d$ is supported in $\partial Y_d$.

        \item We have the bounds
        \begin{gather}
            \M(Q_d) \leq C\varepsilon^{\frac{m-d}m} \label{eq:InvariantFilling_FillingMass}
            \\
            \M(Q \res \partial N_{r_d}(X_d)) \leq C\varepsilon^{\frac{m-1}m} \label{eq:InvariantFilling_FillingSliceMass}
            \\
            \M(Q_d \res \partial N_{r_d}(X_d)) \leq C\varepsilon^{\frac{m-d}m}, \label{eq:InvariantFilling_InvariantFillingSliceMass}
            \intertext{and, for $d > n+1$,}
            \M(Q_{d-1} \res (\partial N_{r_{d-1}}(X_{d-1}) \cap \partial N_{r_d}(X_d))) \leq C\varepsilon^{\frac{m-d}m}.  \label{eq:InvariantFilling_InvariantFillingDoubleSliceMass}
        \end{gather}
    \end{itemize}
    
    Observe that if all of the $r_d$'s are chosen from the stipulated intervals, then each $Y_i \setminus Y_{i-1}$ is the disjoint union of $p$ components, each homeomorphic to a disk. For the base case of $d = n+1$: use the coarea inequality to choose $r_{n+1} \in [(2pn-n-1)\varepsilon^{1/m}, (2pn-n)\varepsilon^{1/m}]$ to satisfy \cref{eq:InvariantFilling_FillingSliceMass}. Then $Y_{n+1}$ is homeomorphic to $p$ disks $D_1, \dotsc, D_p$, and the $\ZZ_p$ action maps each disk homeomorphically to another. Thus we can simply replace each $Q \res D_i$ with the image of $Q \res D_1$ to get the required $Q_1$. Thus \cref{eq:InvariantFilling_FillingMass,eq:InvariantFilling_InvariantFillingSliceMass} are also satisfied. We can also define $G_{n+1}$ by replacing each $G \res D_i$ with an image of $G \res D_1$.
    
    For the induction step, assume that the induction hypothesis holds for $d - 1$. Use the simultaneous coarea inequality to slice $Q$, $Q \res \partial N_{r_{d-1}}(X_{d-1})$, and $Q_{d-1} \res \partial N_{r_{d-1}}(X_{d-1})$, so we choose $r_d \in [(2pn-d)\varepsilon^{1/m}, (2pn-d+1)\varepsilon^{1/m}]$ to satisfy the bounds \cref{eq:InvariantFilling_FillingSliceMass,eq:InvariantFilling_InvariantFillingDoubleSliceMass} and the following:
    \begin{equation}
        \label{eq:InvariantFilling_FillingDoubleSliceMass}
        \M(Q \res (\partial N_{r_{d-1}}(X_{d-1}) \cap \partial N_{r_d}(X_d))) \leq C\varepsilon^{\frac{m-2}m}.
    \end{equation}
    
    Observe that $\cl{Y_d \setminus Y_{d-1}}$ is $p$ disjoint disks $D_1, \dotsc, D_p$. $D_1$ and $Y_{d-1}$ have disjoint interiors but share part of their boundary at $B = \partial Y_{d-1} \cap \partial D_1 = \partial Y_{d-1} \cap N_{r_d}(e^d)$ for some $d$-cell $e^d$ of $X_d$. Thus $B$ is homeomorphic, and nearly 1-bilipschitz, to $S^{n+1}(\Sp{d-n-2} \times \D^{p(n+1)-1-d}(r_d))$.
    
    Let us assume that $B$ and $\partial B$ intersect $Q$, $Q_{d-1}$, and $R$ transversally. Then $\partial(Q \res D_1) = R \res D_1 + Q \res \partial D_1$ and $\partial Q_{d-1} = R \res Y_{d-1} + Q_{d-1} \res \partial Y_{d-1}$. It can be shown that $(Q - Q_{d-1}) \res B$ is a relative cycle, because
    \begin{align*}
        & \partial(Q \res B) - \partial(Q_{d-1} \res B)
        \\
        ={}& R \res B + Q \res \partial B - R \res B - Q_{d-1} \res \partial B
        \\
        ={}& Q \res \partial B - Q_{d-1} \res \partial B,
    \end{align*}
    and $\partial B \subset \partial Y_d$. Let us find a small relative filling of this relative cycle. Observe that $\partial B \subset \partial N_{r_{d-1}}(X_{d-1}) \cap \partial N_{r_d}(X_d)$, so by \cref{eq:InvariantFilling_InvariantFillingDoubleSliceMass,eq:InvariantFilling_FillingDoubleSliceMass}, $\M(Q \res \partial B - Q_{d-1} \res \partial B) \leq C\varepsilon^{\frac{m-d}m}$. Moreover, $Q \res \partial B - Q_{d-1} \res \partial B$ is a null-homologous cycle because both $Q \res \partial B$ and $Q_{d-1} \res \partial B$ are homologous to $F \res \partial B$ by the induction hypothesis. Hence, by the isoperimetric inequality (this holds because $\frac{m-d}k \geq 1$ and so we can apply the deformation theorem with a grid size less than $r_d$)), $Q \res \partial B - Q_{d-1} \res \partial B$ can be filled by a chain of mass at most $C\varepsilon^{\frac{k(m-d)}{(k-1)m}} \leq C\varepsilon^{\frac{m-d}m}$.
    
    Adding that filling to $Q \res B - Q_{d-1} \res B$, we get a cycle of mass at most $C\varepsilon^{\frac{m-d}m}$ by \cref{eq:InvariantFilling_FillingSliceMass} and \cref{eq:InvariantFilling_InvariantFillingSliceMass} from the $d-1$ case. Since this cycle lies in $D_1$ which is nearly 1-bilipschitz to $S^{n+1}(\D^{d-n-1 } \times \D^{p(n+1)-1-d}(r_d))$, it can be coned off by a chain in $D_1$ of mass at most $C\varepsilon^{\frac{m-d}m}$. Add this to $Q_{d-1}$ and the images of $Q \res D_1$ to get $Q_d$. One can verify that $Q_d$ is a relative filling of $R \res Y_d$ in $Y_d$, and that \cref{eq:InvariantFilling_FillingMass,eq:InvariantFilling_InvariantFillingSliceMass} hold.

    It remains to check that $Q_d$ is relatively homologous to $F \res Y_d$. This is true because $Q_d - F \res Y_d$ can be filled by $G_d$ which is the sum of $G_{d-1} + G \res D_1$ with the cone of $G_{d-1} \res B + G \res B$ inside $D_1$.

    Therefore we now have a $\ZZ_p$-invariant chain $Q_{p(n+1)-1}$ in $\Sp{p(n+1)-1}_r$ such that $\partial Q_{p(n+1)-1} - R$ is supported in $\Delta_r$, and by \cref{eq:InvariantFilling_InvariantFillingSliceMass}, $\M(\partial Q_{p(n+1)-1} - R) \leq C\varepsilon^{\frac{m-p(n+1)+1}m} + \M(R \res \Delta_r)$.
\end{proof}

To formally define $\cyc$, we must specify the homeomorphism $h : \Sp{p(n+1)-1} \setminus \Delta \to L_n \times \D^{n+1}$ from \cref{eq:SphereQuotientLensBundle}, and estimate its Jacobian. It is covered by a smooth embedding $f : \Sp{p(n+1)-1} \setminus \Delta \to \Sp{(p-1)(n+1)-1} \times \D^{n+1}$ that is defined using a real analogue of the Discrete Fourier Transform Matrix. Let $\theta = 2\pi/p$, and consider the \emph{real discrete Fourier transform matrix} $F$ defined as follows. When $p = 2$, define $F = \frac1{\sqrt2}\left[ \begin{smallmatrix}
    1 &  1 \\
    1 & -1
\end{smallmatrix} \right]$. When $p > 2$, define
\begin{equation*}
    F = \begin{bmatrix}
    u & v_1 & w_1 & v_2 & w_2 & \cdots & v_{\frac{p-1}2} & w_{\frac{p-1}2}
\end{bmatrix}^T,
\end{equation*}
whose rows are, for $j = 1,\dotsc, p-1$,
\begin{equation}
    \label{eq:FourierBasis}
\begin{aligned}
    u &= \sqrt{\frac1p}(1,\dotsc,1)
    \\
    v_j &= \sqrt{\frac2p}(1, \cos(j\theta), \cos(2j\theta), \dotsc, \cos((p-1)j\theta))
    \\
    w_j &= \sqrt{\frac2p}(0, \sin(j\theta), \sin(2j\theta), \dotsc, \sin((p-1)j\theta)).
\end{aligned}
\end{equation}

Let $F^\perp$ denote $F$ with the first row removed, and $F^\perp_m = F^\perp \otimes I_m$, and $F_m = F \otimes I_m$, where $\otimes$ denotes the Kronecker product. Then $f$ is defined by the formula
\begin{equation}
    \label{eq:CycProdHomeoLensBundle_DFT}
    f(x_1,\dotsc,x_p) =  \left( \frac{F^\perp_{n+1} x}{\norm{F^\perp_{n+1} x}} , \frac{x_1 + \dotsb + x_p}p \right).
\end{equation}

Write $\Sp{p(n+1)-1}_\varepsilon = \Sp{p(n+1)-1} \setminus \Delta_\varepsilon$, and let $f_\varepsilon = f|_{\Sp{p(n+1)-1}_\varepsilon}$. In particular, the image of $f_\varepsilon$ is $\Sp{(n+1)(p-1)-1} \times B_{\sqrt{1 - \varepsilon^2}}(0)$.

\begin{lemma}
    \label{lem:StretchingAwayFromDiagonal}
    Let $\tilde\Delta = \{(x,\dotsc,x)~|~x \in \RR^{n+1}\}$. Consider any $v \in \Delta_\varepsilon$ for sufficiently small $\varepsilon > 0$. Then if the domain of $Df_v : \RR^{p(n+1)} \to \RR^{p(n+1)}$ is decomposed as $\RR^{p(n+1)} = \tilde\Delta \oplus \tilde\Delta^\perp$ and the codomain is given some permutation of coordinates, then $Df_v = \frac1{\sqrt{p}} I_{n+1} \oplus \frac1\varepsilon I_{(p-1)(n+1)}$.
\end{lemma}
\begin{proof}
    $f$ extends onto a composite map
    \begin{equation*}
        \RR^{p(n+1)} \setminus \tilde\Delta \xrightarrow{F_{n+1}} \RR^{p(n+1)} \setminus \RR^{n+1} \times \{0\}^{(p-1)(n+1)} \xrightarrow{(x,y) \mapsto \left( \frac{y}{\norm{y}}, \frac{x}{\sqrt{p}} \right)} \Sp{(n+1)(p-1)-1} \times \RR^{n+1}.
    \end{equation*}
    Observe that $F_{n+1}$ is an isometry that maps $\partial \Delta_\varepsilon$ to the set $\{(x,y)~|~x \in \RR^{n+1}, y \in \RR^{(p-1)(n+1)}, \norm{x} = \sqrt{1 - \varepsilon^2}, \norm{y} = \varepsilon\}$ for sufficiently small $\varepsilon$. The lemma statement follows.
\end{proof}

    \begin{lemma}
        \label{lem:OrthoCompSumOfSubspaces}
        Let $A, B$ be vector subspaces of $\RR^n$. Then $(A + B)^\perp = A^\perp \cap B^\perp$.
    \end{lemma}
    \begin{proof}
        Observe that $(A + B)^\perp \subset A^\perp$ and $(A + B)^\perp \subset B^\perp$, so $(A + B)^\perp \subset A^\perp \cap B^\perp$. For the reverse inclusion, let $v \in A^\perp \cap B^\perp$. Then for all $u + w \in A + B$, where $u \in A$ and $w \in B$, $v \perp u$ and $v \perp w$ so $v \perp w + u$.
    \end{proof}

    \begin{lemma}
        \label{lem:PlaneProdProjDiagRank}
        Let $\Pi_1,\dotsc,\Pi_p$ be subspaces of $\RR^{n+1}$. Then the rank of the orthogonal projection $\pi : \Pi_1 \times \dotsb \times \Pi_p \to \tilde\Delta$ is at least $\frac1p(\dim\Pi_1 + \dotsb + \dim\Pi_p)$. Equality is attained when $\Pi_1 = \dotsb = \Pi_p$.
    \end{lemma}
    \begin{proof}
        Let $d = \dim\Pi_1 + \dotsb + \dim\Pi_p$.
        
        By \cref{lem:OrthoCompSumOfSubspaces},
        \begin{align*}
            \rank \pi
            &= d - \dim(\Pi_1 \times \dotsb \times \Pi_p \cap \tilde\Delta^\perp)
            \\
            &= d - p(n+1) + \dim(\tilde\Delta + \Pi_1^\perp \times \dotsb \times \Pi_p^\perp))
            \\
            &=d - p(n+1) + \underbrace{\dim\tilde\Delta}_{n+1} + \underbrace{\dim(\Pi_1^\perp \times \dotsb \times \Pi_p^\perp)}_{p(n+1)-d)} - \dim(\tilde\Delta \cap \Pi_1^\perp \times \dotsb \times \Pi_p^\perp))\\
            &= n + 1 - \dim(\tilde\Delta \cap \Pi_1^\perp \times \dotsb \times \Pi_p^\perp)).
        \end{align*}
        Thus it remains to prove that $\dim(\tilde\Delta \cap \Pi_1^\perp \times \dotsb \times \Pi_p^\perp)) \leq n + 1 - d/p$. By the rank-nullity theorem, this is equivalent to proving that the dimension of the image of the orthogonal projection $\tilde\Delta \to \Pi_1 \times \dotsb \times \Pi_p$ is at least $d/p$. To show this, note that the Pigeonhole Principle implies that some $\Pi_i$ must have dimension at least $d/p$. Consider the composite projection $\tilde\Delta \to \Pi_1 \times \dotsb \times \Pi_p \to \Pi_i$.
        The image of this composite projection is surjective because for each vector $v \in \Pi_i$, $(v,\dotsc,v)$ projects to $v$

        When $\Pi_1 = \dotsb = \Pi_p$, by applying a linear isomorphism to $\RR^{n+1}$, we may assume that $\Pi_i = \RR^m \times \RR^{n+1-m}$ for some fixed $m$. Then the orthogonal projection $\tilde\Delta \to \Pi_1 \times \dotsb \times \Pi_p$ fixes $(e_i,\dotsc,e_i)$ when $i \leq m$ and kills it otherwise. Thus the rank of this projection is exactly $m$.
    \end{proof}

\begin{theorem}
    \label{thm:CartPowerMassNearDiagonal}
    Let $R \in \Z{k}(\Sp{n})$ be a polyhedral cycle, where $k \geq 1$. Then for sufficiently small $r$, there exists a constant $C$ such that
    \begin{align*}
        \M(f_\sharp(R^p \res \partial\Delta_r)) \leq C\M(R) && \text{and} && \M(f_\sharp(R^p \res \Delta_r)) \leq Cr\M(R).
    \end{align*}
\end{theorem}
\begin{proof}
    Let the grid size be $\lambda$. Let $k$-cells of $R$ as $e^k_1, \dotsc, e^k_q$. Let $e^k_i$ appear with multiplicity $a_i$ in $R$. (Note that $\abs{a_i} \leq p$.) Thus $\spt R^p = (\spt R)^p$ is a $pk$-dimensional cubical complex whose top-dimensional cells are of the form $e^k_{i_1} \times \dotsb \times e^k_{i_p}$.

    For each $0 \leq d \leq k$, let $K_d = 4^{k-d+1}$ and let $V_d$ (resp. $\frac12V_d$) be the $K_d\varepsilon$-neighbourhood (resp. $\frac12K_d\varepsilon$-neighbourhood), in the $\ell^\infty$ metric, of the $d$-skeleton of $R$. Note that for all $1 \leq d \leq k$, $K_{d-1} = 4K_d$.
    
    let $e^d_1,\dotsc,e^d_{q_d}$ denote the $d$-cells of $R$, for $0 \leq d \leq k$. Let $a_i$ denote the multiplicity of $e^k_i$ in $R$. Observe that for all $(x_1,\dotsc,x_p) \in \partial\Delta_\varepsilon$, any $x_j$ and $x_k$ must be distance at most $2\varepsilon$ apart. Then
    \begin{lemma}
        \label{lem:NbhdDiffCuboids}
        For each $1 \leq d \leq k$ and for sufficiently small $\varepsilon$, $\cl{V_d \setminus \frac12V_{d-1}}$ is a disjoint union of cuboids, each cuboid containing $e^d_i \setminus \frac12V_{d-1}$ for a unique $i$. The cuboid $E^d_i$ consists of points at $\ell^\infty$ distance at most $K_d\varepsilon$ from $e^d_i$, and whose orthogonal projection onto $e^d_i$ lies in $\cl{e^d_i \setminus \frac12V_{d-1}}$.

        Furthermore, the cuboids are further than distance  $2\varepsilon$ away from each other.
    \end{lemma}
    \begin{proof}
        First we show that those ``cuboid'' sets defined above are actually cuboids. Without loss of generality, and by the invariance of the $\ell^\infty$ metric under coordinate permutations and translations, we may consider $e^d_1$ and assume that it is $[0,\delta]^d \times \{0\}^{n-d}$. Then the ``cuboid sets'' are $[\frac12K_{d-1}\varepsilon,\delta - \frac12K_{d-1}\varepsilon]^d \times [-K_d\varepsilon, K_d\varepsilon]^{n-d}$. These are indeed cuboids.

        Next we show that the union of the cuboids is indeed $\cl{V_d \setminus \frac12V_{d-1}}$. To prove the inclusion, first note that the union of the cuboids is certainly contained in $V_d$. Each cuboid is also disjoint from the interior of $\frac12V_{d-1}$. To see this, again use the above translation and change of coordinates; then since each point in the cuboid orthogonally projects onto $e^d_1$ to give a point in $[\frac12K_{d-1}\varepsilon,\delta - \frac12K_{d-1}\varepsilon]^d \times \{0\}^{n-d}$, its first $d$ coordinates must be in the interval $[\frac12K_{d-1}\varepsilon,\delta - \frac12K_{d-1}\varepsilon]$, which means that the point itself must be at $\ell^\infty$ distance at least $\frac12K_{d-1}\varepsilon$ from $\partial e^d_1$.

        To prove the reverse inclusion, consider any $x \in \cl{V_d \setminus \frac12V_{d-1}}$. Then $x$ must lie within $\ell^\infty$ distance $K_d\varepsilon$ from some $e^d_i$. Do the same coordinate transformations as before to get $e^d_i = [0,\delta]^d \times \{0\}^{n-d}$, so that $x \in [-K_d\varepsilon,\delta + K_d\varepsilon]^d \times \{-K_d\varepsilon,K_d\varepsilon\}^{n-d}$. Since $x$ must be further than $\ell^\infty$ distance $\frac12K_{d-1}\varepsilon = 2K_d\varepsilon$ from every face of $e^d_i$, for each $1 \leq j \leq d$ we must have that $x$ lies outside \begin{multline*}
            [-2K_d\varepsilon,\delta + 2K_d\varepsilon]^{j-1} \times ([-2K_d\varepsilon,2K_d\varepsilon] \cup [\delta - 2K_d\varepsilon, \delta + 2K_d\varepsilon])
            \\
            \times [-2K_d\varepsilon,\delta + 2K_d\varepsilon]^{d-j} \times [-2K_d\varepsilon,2K_d\varepsilon]^{n-d}.
        \end{multline*}
        Therefore the $j^\text{th}$ coordinate of $x$ must lie in $[\frac12K_{d-1}\varepsilon, \delta - \frac12K_{d-1}\varepsilon]$. Since this is true for each $j$, we have that the projection of $x$ onto $e^d_i$ lies outside $V_{d-1}$.
        
        Finally, we show that the cuboids are at distance at least $2\varepsilon$ away from each other. For that we need the following lemma about cubic lattices:
        \begin{lemma}
            \label{lem:CubesTouchingWithoutContainment}
            Consider the standard cubic lattice of edge length $\delta$ in $\RR^n$. Suppose that the cube $c = [0,\delta]^r \times \{0\}^{n-r}$ intersects but does not contain another cube $c'$. Then either for some $i \leq r$ the $i^\text{th}$ factor of $c'$ is $[-\delta, 0]$ or $[\delta, 2\delta]$, or for some $i > r$ the $i^\text{th}$ factor of $c'$ is $[-\delta,0]$ or $[0,\delta]$.
        \end{lemma}
        \begin{proof}
            Assume otherwise for the sake of contradiction. Then $c' = A_i \times \dotsb \times A_r \times \{(x_{r+1},\dotsc,x_n)\}$, where since $c$ intersects $c'$, each $A_i$ is either $[0,\delta]$ or a point that has to be either 0 or $\delta$, and each $x_i = 0$. That implies that $c' \subset c$, giving a contradiction.
        \end{proof}
        
        Suppose for the sake of contradiction that, without loss of generality, the cuboids associated with distinct $d$-cells $e^d_1$ and $e^d_2$ contain points $x_1$ and $x_2$ respectively such that $\norm{x_1 - x_2} \leq 2\varepsilon$. Then the triangle inequality implies that $e^d_1$ and $e^d_2$ are within $\ell^\infty$ distance $(2K_d + 2)\varepsilon$ apart, which for sufficiently small $\varepsilon$ implies that $e^d_1$ and $e^d_2$ must touch. \Cref{lem:CubesTouchingWithoutContainment} implies that, without loss of generality via a reflection, some factor of $e^d_2$ is $[-\delta, 0]$. Thus that coordinate of $E^d_2$ must be at most $-\frac12K_{d-1}\varepsilon = -2 \cdot K_d\varepsilon$, while that coordinate of $E^d_1$ must be at least $-K_d\varepsilon$. Hence the distance between $E^d_1$ and $E^d_2$ is at least $K_d\varepsilon > 2\varepsilon$, giving a contradiction.
    \end{proof}
    
    \begin{lemma}
        \label{lem:CartPowerNearDiag_BlockDecomp}
        For sufficiently small $\varepsilon$, 
        \begin{equation*}
            (\spt R)^p \cap \partial \Delta_\varepsilon \subset \bigcup_{\substack{
                1 \leq i_1, \dotsc, i_p \leq q_d
                \\
                0 \leq d \leq k
                \\
                e^d_l \subset e^k_{i_1} \cap \dotsb \cap e^k_{i_p}
            }}
            (E^d_l)^p \cap e^k_{i_1} \times \dotsb \times e^k_{i_p}.
        \end{equation*}
    \end{lemma}
    \begin{proof}
        For any $x = (x_1,\dotsc,x_p) \in (\spt R)^p \cap \partial\Delta_\varepsilon$, choose the largest $d$ such that some $x_i$ lies in $V_d \setminus V_{d-1}$. This implies that every $x_j$ also lies in $V_d$. In addition, $x_i$ must lie in some $E^d_l$ by \cref{lem:NbhdDiffCuboids}, and $x_i$ is further than distance $K_{d-1}\varepsilon$ away from the $(d-1)$-skeleton of $R$. This implies that every $x_j$ must be further than distance $\frac12K_{d-1}\varepsilon$ away from the $(d-1)$-skeleton of $R$, because $\frac12K_{d-1}\varepsilon = 2 \cdot 4^{k-d+1}\varepsilon > 2\varepsilon$. Thus $x_j$ lies in $V_d \setminus \frac12V_{d-1}$. Together with \cref{lem:NbhdDiffCuboids} and the fact that $x_i$ and $x_j$ must be within distance $2\varepsilon$ of each other, this implies that $x_j \in E^d_l$. Moreover, $x_j$ must lie in some $e^k_{i_j}$. Consequently, $x \in (E^d_l)^p \cap e^k_{i_1} \times \dotsb \times e^k_{i_p}$.
        
        It remains to show that $e^d_l \subset e^k_{i_j}$ for all $j$. Assume for the sake of contradiction that this is false for some $e^k_{i_j}$. Nevertheless, $e^k_{i_j}$ has a point $x_j$ that is also contained in $E^d_l$, so it is within distance $K_d\varepsilon$ of a point in $\cl{e^d_l \setminus \frac12V_{d-1}}$. For sufficiently small $\varepsilon$ this implies that $e^k_{i_j}$ intersects but does not contain $e^d_l$. Perform coordinate transformations until $e^k_{i_j} = [0,\lambda]^k \times \{0\}^{n-k}$. By \cref{lem:CubesTouchingWithoutContainment} and without loss of generality due to a reflection, $e^d_l$ has some factor that is $[-\lambda,0]$. However, that would imply that that coordinate of $E^d_2$ would be at most $-\frac12K_{d-1}\varepsilon = -2 \cdot K_d\varepsilon$, so actually $e^k_{i_j}$ and $E^d_l$ have to be further than distance $K_d\varepsilon$ apart, giving a contradiction. 
    \end{proof}

    Let $\Pi_j$ be the tangent space of $e^k_{i_j}$. (Assume that the cells are nearly flat so we can identify all the tangent spaces.)
    \begin{lemma}
        \label{lem:CycProdMapJacobian}
        For any $1 \leq i_1, \dotsc, i_p \leq q$ and $v$ in the interior of $e^k_{i_1} \times \dotsb \times e^k_{i_p}$,
        \begin{equation*}
            J_vf|_{e^k_{i_1} \times \dotsb \times e^k_{i_p} \cap \partial\Delta_\varepsilon} = \frac{(\sqrt{p})^{pk - \dim(\Pi_1 \times \dotsb \times \Pi_p \cap \tilde\Delta^\perp)}}{\varepsilon^{\dim(\Pi_1 \times \dotsb \times \Pi_p \cap \tilde\Delta^\perp) - 1}} \leq \frac{p^{pk/2}}{\varepsilon^{\dim(\Pi_1 \times \dotsb \times \Pi_p \cap \tilde\Delta^\perp) - 1}}.
        \end{equation*}
    \end{lemma}
    \begin{proof}
        From \cref{lem:StretchingAwayFromDiagonal}, only the stretching in the direction of $\tilde\Delta^\perp$ matters, where $\tilde\Delta^\perp$ denotes the orthogonal complement to $\tilde\Delta$ at $v$. The result follows.
    \end{proof}
    
    $R^p$ has a finite number of tangent $pk$-dimensional planes at each point. Each tangent $pk$-plane is (affinely) spanned by a product of $k$-cells $e^k_{i_1} \times \dotsb \times e^k_{i_p}$. For each $v \in D$, let $S_v$ denote the sphere of radius $\varepsilon$ around $v$ in the orthogonal complement to $\tilde\Delta$ at $p$.
    
    \begin{lemma}
        \label{lem:MassCubeSliceNearDiagonal}
        For each $1 \leq i_1, \dotsc, i_p \leq q$, suppose that $e^k_{i_1} \cap \dotsb \cap e^k_{i_p}$ contains some $e^d_l$ for some  and $0 \leq d \leq k$. Then
        \begin{equation*}
            \M((E^d_l)^p \cap e^k_{i_1} \times \dotsb \times e^k_{i_p} \cap \partial\Delta_\varepsilon) \leq C\lambda^d\varepsilon^{pk-d-1}
        \end{equation*}
    \end{lemma}
    \begin{proof}
        The cells share at least one vertex, and assuming that the cells are small, and by applying a rotation of $\Sp{n}$, we may assume that $e_1 \in e^k_{i_1} \cap \dotsb \cap e^k_{i_p}$ where $e_1 = (1,0,\dotsc,0)$. By applying a diffeomorphism with bilipschitz constant close to 1, we may assume that the cells are affine subsets of $\RR^{n+1}$ that are orthogonal to $e_1$ and that $\Delta$ can be identified with the part of its span $\tilde\Delta$ near $(e_1,\dotsc,e_1)$, whose linear part is spanned by the vectors $d_i = (e_i,\dotsc,e_i)$ for $2 \leq i \leq n+1$. By a further rotation of $\Sp{n}$, we may assume that the edges of the cubical structure are parallel to the standard basis vectors.
    
        $(E^d_l)^p \cap e^k_{i_1} \times \dotsb \times e^k_{i_p}$ is a $pk$-dimensional cuboid whose edges in the direction of $(E^d_l)^p$ have length $\lambda$, and whose other edges have length $(k-d)\varepsilon$. whose affine span is the $pk$-plane $\Pi$.
    
        We will apply the coarea formula for the orthogonal projection $\pi : (E^d_l)^p \cap e^k_{i_1} \times \dotsb \times e^k_{i_p} \cap \partial\Delta_\varepsilon \to \Delta$. Thus we may study the differential of the projection, $D_v\pi$, for $v = (v_1,\dotsc,v_p) \in e^k_{i_1} \times \dotsb \times e^k_{i_p}$: $\ker D_v\pi = \Pi_1 \times \dotsb \times \Pi_p \cap \tilde\Delta^\perp$. Thus the fibers of $\pi$ are the intersections of the domain with the planes $\{v + \ker D_v\pi\}$. In particular, this implies that $\dim(\im \pi) = pk - \dim(\Pi_1 \times \dotsb \times \Pi_p \cap \tilde\Delta^\perp)$.

        Let the tangent spaces of $e^d_l$ be $W$. When $D_v\pi$ is restricted to $W^p$, \cref{lem:PlaneProdProjDiagRank} implies that its image has dimension $d$. 
        
        Since $\pi$ is 1-Lipschitz, $\vol(\pi((e^d_l)^p)) \leq C\lambda^d$. Since every point in $(E^d_l)^p \cap e^k_{i_1} \times \dotsb \times e^k_{i_p}$ is distance at most $CK_d\varepsilon$ from $(e^d_l)^p$, the image of $\pi$ is an affine set of dimension $pk - \dim(\Pi_1 \times \dotsb \times \Pi_p \cap \tilde\Delta^\perp)$ that is within distance $(k-d)\varepsilon$ from a flat $d$-dimensional disk of volume at most $C\lambda^d$, so $\vol(\im \pi) \leq C\lambda^d\varepsilon^{pk - d - \dim(\Pi_1 \times \dotsb \times \Pi_p \cap \tilde\Delta^\perp)}$. Now, since each fiber of $\pi$ is part of a sphere of radius $\varepsilon$ and dimension $\dim(\Pi_1 \times \dotsb \times \Pi_p \cap \tilde\Delta^\perp) - 1$, the volume of the fiber is at most $C\varepsilon^{\dim(\Pi_1 \times \dotsb \times \Pi_p \cap \tilde\Delta^\perp) - 1}$. Therefore the coarea inequality implies that
        \begin{equation*}
            \M((E^d_l)^p \cap e^k_{i_1} \times \dotsb \times e^k_{i_p} \cap \partial\Delta_\varepsilon) \leq C\lambda^d\varepsilon^{pk-d-1}.
        \end{equation*}
    \end{proof}
    
    Then by the area formula,
    \begin{align*}
        \M(f_\sharp(R^p \res \partial\Delta_\varepsilon))
        &= \int_{\spt R^p \cap \partial\Delta_\varepsilon} \theta_{R^p}(v) Jf|_{\spt R^p \cap \partial\Delta_\varepsilon}  \,dv
        \\
        (\text{\cref{lem:CartPowerNearDiag_BlockDecomp}}) &\leq C\sum_{\substack{
                1 \leq i_1, \dotsc, i_p \leq q_d
                \\
                0 \leq d \leq k
                \\
                e^d_l \subset e^k_{i_1} \cap \dotsb \cap e^k_{i_p}
            }} \int_{(E^d_l)^p \cap e^k_{i_1} \times \dotsb \times e^k_{i_p} \cap \partial\Delta_\varepsilon} J_vf|_{e^k_{i_1} \times \dotsb \times e^k_{i_p} \cap \partial\Delta_\varepsilon} \,dv
        \\
        (\text{\cref{lem:CycProdMapJacobian}}) &\leq p^{p+k/2} \sum_{\substack{
                1 \leq i_1, \dotsc, i_p \leq q_d
                \\
                0 \leq d \leq k
                \\
                e^d_l \subset e^k_{i_1} \cap \dotsb \cap e^k_{i_p}
            }} \frac{\M((E^d_l) \cap e^k_{i_1} \times \dotsb \times e^k_{i_p} \cap \partial\Delta_\varepsilon)}{\varepsilon^{\dim(\Pi_1 \times \dotsb \times \Pi_p \cap \tilde\Delta^\perp) - 1}},
        \\
        (\text{\cref{lem:MassCubeSliceNearDiagonal}})
        &\leq C \sum_{\substack{
                1 \leq i_1, \dotsc, i_p \leq q_d
                \\
                0 \leq d \leq k
                \\
                e^d_l \subset e^k_{i_1} \cap \dotsb \cap e^k_{i_p}
            }} 
        \lambda^d\varepsilon^{pk - d - \dim(\Pi_1 \times \dotsb \times \Pi_p \cap \tilde\Delta^\perp)}
        \\
        (\text{\cref{lem:PlaneProdProjDiagRank}}) &\leq  C \sum_{\substack{
                1 \leq i_1, \dotsc, i_p \leq q_d
                \\
                0 \leq d \leq k
                \\
                e^d_l \subset e^k_{i_1} \cap \dotsb \cap e^k_{i_p}
            }} 
        \lambda^d\varepsilon^{k - d}
        \\
        &\leq C \sum_{\substack{
                1 \leq i_1, \dotsc, i_p \leq q_d
                \\
                0 \leq d \leq k
                \\
                e^d_l \subset e^k_{i_1} \cap \dotsb \cap e^k_{i_p}
            }} 
        \lambda^k.
    \end{align*}
    Each summand corresponds to $p$ $k$-cells and a $d$-cell in their intersection, so those $k$-cells must all share some vertex. Thus the number of summands is bounded by the number of vertices in $R$ times the number of $k$-cells around a vertex, times the number of lower-dimensional cells in a $k$-cell. This is all bounded by a constant $C(n)$ times the number of $k$-cells in $R$, which is at most $\M(R)/\lambda^k$. Altogether, this implies that $\M(f_\sharp(R^p \res \partial\Delta_\varepsilon)) \leq C\M(R)$.
    
    $\M(f_\sharp(R^p \res \Delta_r))$ can be bounded using the coarea inequality by integrating  $\M(f_\sharp(R^p \res \partial\Delta_s))$ for $0 \leq s \leq r$.
\end{proof}

\begin{proposition}
    \label{prop:CycProdMap_Existence_PolyCycles}
    For each mod $p$ polyhedral cycle $\mathcal{R} \in \Z{k}(\Sp{n})$, let $R$ be one of its representatives mod $p$. Then the following hold:
    \begin{enumerate}
        \item The sequence of integral chains $S_1, S_2, \dotsc$, where $S_i = f_\sharp(R^p \res \Sp{p(n+1)-1}_{1/i})$, converges in the flat topology to an integral chain $S$ in $\Sp{(n+1)(p-1)-1} \times \D^{n+1}$ that is $\ZZ_p$-invariant. Moreover,
        \begin{equation}
            \label{eq:CycProdMap_BoundsOnMassAndBoundary}
        \begin{aligned}
            \M(S) &\leq \frac{\M(R)^p}{r_0^{pk}} + Cr_0\M(R)
            \\
            \M(\partial S) &\leq C\M(R)
        \end{aligned}
        \end{equation}

        \item The image of $S$ under the covering map $\Sp{(n+1)(p-1)-1} \times \D^{n+1} \to L \times \D^{n+1}$ is $pT$ for some flat chain $T$ in $L \times \D^{n+1}$.

        \item $[T]$ is a mod $p$ relative cycle in $\Z{pk}(L \times \D^{n+1}, L \times \partial\D^{n+1})$.
    \end{enumerate}
\end{proposition}
\begin{proof}
    This result is immediate for 0-cycles as they are formal sums of points. Henceforth we assume that $k \geq 1$. First let us establish that each $S_i$ is an integral current mod $p$. By the Closure Theorem \cite[p.~432]{Federer_GMT}, it suffices to show that $\M(S_i)$ and $\M(\partial S_i)$ are finite. The former is true because $R^p \res \Sp{p(n+1)-1}_{1/i}$ has finite mass and over its support, $f$ has Jacobian bounded by some power of $i$ by \cref{lem:CycProdMapJacobian}. The latter is true because $\partial\Delta_{1/i}$ intersects $R^p$ transversally so that the intersection has finite mass, and the Jacobian of $f$ has a uniform bound over the intersection.

    The sequence $S_i$ is Cauchy in the mass metric, and thus also in the flat metric: Fix any integer $k$, and consider any $i > j > k$. Then \cref{thm:CartPowerMassNearDiagonal} implies that $\M(S_i - S_j) \leq C\M(R)/k$. To apply the Compactness Theorem for integral currents mod $p$ \cite[p.~432]{Federer_GMT} to the Cauchy sequence $S_i$, we need uniform bounds on $\M(S_i)$ and $\M(\partial S_i)$. Let $r_0$ be the threshold for ``sufficiently small'' in \cref{thm:CartPowerMassNearDiagonal}. We may drop initial terms of the sequence until every $1/i < r_0$. Then \cref{thm:CartPowerMassNearDiagonal} and \cref{lem:StretchingAwayFromDiagonal} imply that
    \begin{equation*}
        \M(S_i) \leq \M(f_\sharp(R^p \res \Sp{p(n+1)-1}_{r_0})) + \M(f_\sharp(R^p \res \Delta_{r_0})) \leq \frac{\M(R)^p}{r_0^{pk}} + Cr_0\M(R).
    \end{equation*}
    Additionally, \cref{thm:CartPowerMassNearDiagonal} implies that $\M(\partial S_i) = \M(f_\sharp(R^p \res \partial\Delta_{1/i})) \leq C\M(R)$. Thus the Compactness Theorem guarantees that $S_i$ converges in the flat topology to an integral chain in $\I{k}(L_n \times \D^{n+1}; \Fl)$.
    
    It can be verified that the $S_i$'s are restrictions of $S$. Since each of the $S_i$'s are $\ZZ_p$-invariant, $S$ is also $\ZZ_p$-invariant. As a result, the image of $S$ under the covering map is $pT$ for some flat chain $T$ in $L \times \D^{n+1}$. It can be verified that $T$ is a mod $p$ relative cycle.
\end{proof}

Let $\cyc$ denote the map from \cref{prop:CycProdMap_Existence_PolyCycles} that maps polyhedral cycles to $\Z{k}(\Sp{n})$ to $\Z{pk}(L \times \D^{n+1}, L \times \partial\D^{n+1})$. Let $\D^m_r$ denote a closed $m$-disk of radius $r$.

\begin{proposition}
    \label{prop:CycProdMap_Continuous}
    $\cyc$ is continuous over the space of polyhedral cycles. It also extends to a continuous map $\cyc : \Z{k}(\Sp{n}) \mapsto \Z{pk}(L \times \D^{n+1}, L \times \partial\D^{n+1})$.
\end{proposition}
\begin{proof}
    For any $0 < \varepsilon < 1$ and $\mu > 0$, let $\delta = \varepsilon^{5pn}/(\mu + 1)^{6pn} < 1$. Consider any polyhedral cycles $R_1, R_2 \in \Z{k}(\Sp{n})$ within flat distance $\delta$ of each other, such that $\M(R_1), \M(R_2) < \mu$. Without loss of generality, assume that they are within the same grid size. By the isoperimetric inequality, $R_1 - R_2$ is filled by a chain $Q$ of mass at most $C(\delta + \delta^{\frac{k+1}k}) \leq C\delta$.

    Expanding $R_1^p = (R_2 + \partial Q)^p$ using the binomial theorem gives a sum of monomials that, other than $R_2^p$ and $(\partial Q)^p$, can be groups of $p$ terms so that each group consists of different cyclic permutations of the same term. Consider any such group, whose sum is a $\ZZ_p$-invariant cycle $G$. Each term in the group must have at least one factor of $\partial Q$, so one of the terms must be $\partial Q \times A_2 \times \dotsb \times A_p = \partial(Q \times A_2 \times \dotsb \times A_p)$, where each $A_i$ is either $\partial Q$ or $R_2$. Thus $G$ can be filled by $\sum_\text{cyc} Q \times A_2 \times \dotsb \times A_p$, a chain of mass at most $C\delta\mu^{p-1}$. Similarly, $(\partial Q)^p$ can be filled by $Q \times (\partial Q^{p-1}$ which is a chain of mass at most $C\delta\mu^{p-1}$. Summing over all such groups $G$ and $(\partial Q)^p$, we obtain that $R_1^p - R_2^p$ is an invariant cycle with a filling by a chain of mass at most $C\delta\mu^{p-1}$.

    By \cref{thm:InvariantFilling} for $m = 5pn$, there exists a $\ZZ_p$-invariant chain $S$ in $\Sp{p(n+1)-1}$ of mass at most $C(\delta \mu^{p-1})^{3/5}$ such that $\M(R_1^p - R_2^p - \partial S) \leq C(\delta \mu^{p-1})^{3/5} + \M((R_1^p - R_2^p) \res \Delta_r)$, where $r = 2pn(\delta \mu^{p-1})^{1/5pn}$. Pushing all these chains by $f_\sharp$, \cref{lem:StretchingAwayFromDiagonal} implies that that the Jacobian of $f|_{\spt S}$ is at most $1/r^{k+1} \leq C(\delta \mu^{p-1})^{2/5}$. Thus $\M(f_\sharp(S)) \leq C\delta \mu^{p-1} \leq C\varepsilon^{5pn}(\frac{\mu}{\mu+1})^{p-1}/(\mu + 1)^{6pn-p+1} \leq C\varepsilon$. Similarly, if we let $\tilde{T}_i$ be the lift of $\cyc(R_i)$ over the universal cover, then
    \begin{align*}
        \M(\tilde{T}_1 - \tilde{T}_2 - \partial S) &\leq C\varepsilon + \M(f_\sharp((R_1^p - R_2^p) \res \Delta_r))
        \\
        (\text{\cref{thm:CartPowerMassNearDiagonal}}) &\leq C\varepsilon + Cr\mu
        \\
        &= C\varepsilon + C(\delta \mu^{p-1})^{1/5pn}\mu
        \\
        &= C\varepsilon + C\varepsilon \left(\frac{\mu^{p-1}}{(\mu+1)^{6pn}}\right)^{1/5pn}\mu
        \\
        &\leq C\varepsilon + C\varepsilon(\mu + 1)^{\frac{p - 1 + 5pn - 6pn}{5pn}}
        \\
        &\leq C\varepsilon.
    \end{align*}
    Therefore, $\Fl(\cyc(R_1), \cyc(R_2)) \leq C\varepsilon/p$.

    So far we have established that $\cyc$ is uniformly continuous in the flat metric as a map from the polyhedral cycles in $\Z{k}(\Sp{n}; \Fl)^{< \mu} = \{T \in \Z{k}(\Sp{n}; \Fl) : \M(T) < \mu\}$ to $\I{pk}(L_n \times \D^{n+1}; \Fl)$. Moreover, the closure of the image is compact because of \cref{eq:CycProdMap_BoundsOnMassAndBoundary} and the Compactness Theorem for integral currents mod $p$ \cite[p.~432]{Federer_GMT}. Thus $\cyc$ extends to a continuous map $\Z{k}(\Sp{n}; \Fl)^{< \mu} \to \Z{pk}(L_n \times \D^{n+1}, L_n \times \partial\D^{n+1}; \Fl)$. The extensions must be consistent for every value of $\mu$, giving a continuous map $\cyc : \Z{k}(\Sp{n}; \Fl) \to \Z{pk}(L_n \times \D^{n+1}, L_n \times \partial\D^{n+1}; \Fl)$. The proposition then follows.
\end{proof}

\subsection{Geometric representations of the Bockstein homomorphisms}
\label{sec:Bockstein}

Now we are in a position to prove \cref{thm:Bockstein}. Observe that the map $b$ in \cref{eq:Bockstein} is equal to $\cyc : \Z0(\Sp{n}) \to \Z0(L_n \times \D^{n+1}, L_n \times \partial\D^{n+1})$ composed with the map that projects every relative 0-cycle in $L_n \times \D^{n+1}$ onto the second factor $\D^{n+1}$. Combined with \cref{prop:CycProdMap_Continuous}, this implies that $b$ is continuous.

\begin{proof}[Proof of \cref{thm:Bockstein}]
    Let us first handle the case where $p > 2$. Since $H^{n+1}(K(\ZZ_p,n), \ZZ_p) \iso \ZZ_p$, it suffices to prove that $b^*(\iota_{n+1})$ evaluates to the same value as $\beta(\iota_n)$ on some generator of $H_{n+1}(\Z[rel]0(\D^{n+1}); \ZZ_p)$.
    
    Let us first prove this theorem for $n = 1$. Let $X$ be the mapping cone of the map $\Sp1 \to \Sp1$ defined by $w \mapsto w^p$. Define the map $g : X \to \Z0(\Sp1)$ as follows. Parametrize the points in $X$ by $(t,w)$ where $w \in \Sp1 \subset \C$ and $t \in I$, such that all points $(0,w)$ are identified and each point $(1,w)$ is identified with $(1, \zeta w)$, where $\zeta = e^{2\pi i/p}$. Define
    \begin{equation}
        \label{eq:Bockstein_MappingConeEmbedding}
        g(t,w) = \{we^{2\pi i tk/p}~|~ {-(p-1)/2} \leq k \leq (p-1)/2\}.
    \end{equation}
    This is well-defined because $g(0,w) = 0$ for all $w$, and $g(1,w)$ is $w$ times the roots of unity, so $g(1,\zeta w) = g(1,w)$.

    $g$ represents the cohomology class in $g^*(\iota_1) \in H^1(X)$, and it evaluates on the 1-skeleton of $X$, denoted by $X^1$, by the Almgren isomorphism: gluing $g|_{X^1}$ into exactly the fundamental homology class of $\Sp1$, i.e. $1 \in H_1(\Sp1) \iso \ZZ_p$. (Orientation is preserved.) Let us compute $\beta g^*(\iota_1) \in H^2(X)$. Thus we consider the map of short exact sequences among cellular cochains,
    \begin{equation*}
    \begin{tikzcd}
        0 \rar & C^1(X; \ZZ_p) \rar{p}\dar & C^1(X; \ZZ_{p^2}) \rar\dar & C^1(X; \ZZ_p) \rar\dar & 0
        \\
        0 \rar & C^2(X; \ZZ_p) \rar{p} & C^2(X; \ZZ_{p^2}) \rar & C^2(X; \ZZ_p) \rar & 0
    \end{tikzcd}
    \end{equation*}
    We need to go from the top-right group to the bottom-left group by the snake lemma. $g^*(\iota_1)$ is represented by a cochain $\psi$ in the top-right group that evaluates to 1 on the 1-cell $e^1 \subset X$. That cochain is the image of another cochain $\varphi \in C^1(X; \ZZ_{p^2})$ that evaluates to $1 \in \ZZ_{p^2}$ on $e^1$. Then $d\varphi$ evaluates on the 2-cell $e^2$ of $X$ to $\varphi(pe^1) = p\varphi(e^1) = p \in \ZZ_{p^2}$. Dividing by $p$ gives a cochain that evaluates to $1 \in \ZZ_p$ on $e^2$. Therefore $\beta g^*(\iota_1)$ has to evaluate to 1 on $e^2$.

    Let us now evaluate $g^*b^*(\iota_2)$ on $e^2$. It is the same as evaluating $\iota_2$ on the homology class $[b \circ g] \in H_2(\Z[rel]0(\D^2))$ represented by the family of cycles $b \circ g : X \to \Z[rel]0(\D^2)$. To evaluate the fundamental cohomology class $\iota_2$ on $[b \circ g]$, observe that over the 1-skeleton of $X$, which is $\Sp1$, $b \circ g$ consists of points that are a uniform distance away from $\partial\D^2$. This is because each point in $b \circ g(X^1)$ is the barycentre of $p$ points $x_1, \dotsc, x_p$ that are not all the same, and at least one of those points, which we may assume to be $x_1$ without loss of generality, is at least a distance of $2\pi/p$ (along $\partial\D^2$) away from the others. The barycenter is equal to $\frac{p-1}p y + \frac1p x_1$, where $y = \frac1{p-1}(x_2 + \dotsb + x_p)$. Moreover, the arc of points on $\partial\D^2$ at that distance away from $x$, has a convex hull $K$ that contains $y$. Since $\frac{p-1}p K + \frac1p x$ is a uniform distance from $\partial\D^2$, the same is true of the barycenter.

    Therefore we can contract $b \circ g(X^1)$ to the origin, and get an equivalent map that is constant over $X^1$, which descends to a map $\tilde{b} : X/X^1 \he S\Sp1 \he \Sp2 \to \Z[rel]0(\D^2)$. Thus we may see $\tilde{b}$ as an element $[\tilde{b}] \in \pi_2(\Z[rel]0(\D^2))$ which is sent to $[b \circ g]$ by the Hurewicz homomorphism. For each $w \in \Sp1$, $\tilde{b}(-,w) : [0,1] \to \Z[rel]0(\D^2)$ for fixed $w$ and $t$ going from 0 to 1 traces out $m$ paths from $w$ to the origin. Thus gluing $\tilde{b}$ according to the Almgren isomorphism means sweeping each of those paths into a copy of $\D^2$ in the opposite orientation, eventually giving us $(-m)$ times of the fundamental homology class of $(\D^2, \partial\D^2)$. Thus we want to compute $-m$ modulo $p$.
    
    $m$ is the number of points in $\cyc(\text{sum of $p$ distinct points})$, which is the number of orbits of size $p$ in $(\ZZ_p)^p$ under the cyclic shift action, which is $(p^p - p)/p = p^{p-1} - 1 \equiv -1 \pmod{p}$. Therefore the homology class $[b \circ g]$ is the image under the Hurewicz homomorphism of $[\tilde{b}]$ which corresponds to $-m \equiv 1 \in H_2(\D^2, \partial\D^2)$ via the Almgren isomorphism. By the definition of fundamental cohomology class \cite[p.~3]{MosherTangora_CohomOps}, $\iota_2([b \circ g]) = [\tilde{b}] = 1$.

    Therefore $g^*b^*(\iota_2) = \beta g^*(\iota_1) = g^*\beta(\iota_1) \neq 0$. This implies that $g^*$ is not the zero map on $H^2$, which means it must be an isomorphism on $H^2$. Thus $\beta(\iota_1) = b^*(\iota_2)$. Therefore $b$ represents $\beta$. This proves the theorem for $n = 1$.

    For $n \geq 2$, define the map $G : S^{n-1}X \to \Z0(S^{n-1}\Sp1) \he \Z0(\Sp{n})$ by $G(s,x) = \{s\} \times g(x)$. $G$ represents the cohomology class $G^*(\iota_n) \in H^n(X)$, and it evaluates on the $n$-skeleton $(S^{n-1}X)^n \he S^{n-1}\Sp1$ by the Almgren isomorphism: gluing $G|_{(S^{n-1}X)^n}$ gives exactly the fundamental homology class of $\Sp{n}$, i.e. $1 \in H_n(\Sp{n}) \iso \ZZ_p$. A similar snake lemma argument as before shows that $\beta G^*(\iota_n)$ evaluates to 1 on the $(n+1)$-cell $e^{n+1}$ of $S^{n-1}X$.

    Same as before, we can evaluate $G^*b^*(\iota_{n+1})$ by evaluating $\iota_{n+1}$ on the homology class $[b \circ G] \in H_{n+1}(\Z[rel]0(\D^{n+1}))$ represented by the family of cycles $b \circ G : S^{n-1}X \to \Z[rel]0(\D^{n+1})$. Like before, we can contract $(b \circ G)|_{(S^{n-1}X)^n}$ to $I^{n-1} \times \{0\} \subset S^{n-1}\D^2 \he \D^{n+1}$ (the origins of every slice) and then push the stack of origins to the side, into $\partial\D^{n+1}$. Then the same arguments as before show that $\iota_{n+1}([b \circ G]) = 1$. Therefore $G^*b^*(\iota_{n+1}) = \beta G^*(\iota_n) = G^*\beta(\iota_n) \neq 0$. This implies that $G^*$ is not the zero map on $H^{n+1}$, which means it must be an isomorphism on $H^{n+1}$. Thus $\beta(\iota_n) = b^*(\iota_{n+1})$. Therefore $b$ represents $\beta$.

    For $p = 2$, the proof is largely the same with a few modifications. Observe that the mapping cone $X$ is now $\RP^2$, and replace \cref{eq:Bockstein_MappingConeEmbedding} with the following formula:
    \begin{equation*}
        g(t,w) = \{we^{-i\pi t/2}, we^{i\pi t/2}\}.
    \end{equation*}
    The rest of the proof follows similarly.
\end{proof}

\section{Geometric Representations of the Steenrod Powers}
\label{sec:SteenrodPowers}

Our goal for this section is to formalize and prove \cref{thm:SteenrodPowers}. It can be rigorously stated as follows:

\begin{theorem}
    \label{thm:SteenrodPowers_Formal}
    For each prime $p$, $i \geq 0$, and $0 \leq k < n$. Let $m = n - k$. Then $P^i : H^m(-) \to H^{m + 2i(p-1)}(-)$ and $\beta \circ P^i : H^m(-) \to H^{m+2i(p-1)+1}(-)$ have Brown representatives $\Phi_{pk+m+2i(p-1)} \circ \cyc$ and $\Phi_{pk+m+2i(p-1)+1} \circ \cyc$ respectively, where each $\Phi_q$ is a gluing of some piecewise smooth Brown representative $g_q$ for a generator of $H^q(L_n \times \D^{n+1}, L_n \times \partial\D^{n+1})$.
\end{theorem}

At a high level, our proof can be thought of as a combination of gluing with the construction of the Steenrod powers using equivariant cohomology.

\subsection{The equivariant cohomology characterization of the Steenrod powers}
\label{sec:SteenrodPowersEquivCohom}

Let us give a brief overview of the equivariant cohomology construction of the Steenrod powers. We will use a minor modification of this construction as presented in \cite{Hatcher_AlgTop}. For a topological space $X$, its $p$-fold smash product $X^{\wedge p}$ also admits an action of $\ZZ_p$ by cyclic permutation. Let us construct a \emph{homotopy quotient}, a homotopy-theoretic analogue of $X^{\wedge p}/\ZZ_p$ with ``nicer'' properties. Let $d$ be either $\infty$ or an odd integer, and consider the free $\ZZ_p$ action on the contractible infinite-dimensional sphere $\Sp\infty$ corresponding to the covering space $\Sp{d} \to \Lens_p^d$. This gives a free $\ZZ_p$ action on $\Sp{d} \times X^{\wedge p}$ by acting on each factor simultaneously. The quotient by this action is denoted by $\Sp{d} \times_{\ZZ_p} X^{\wedge p}$. When $d = \infty$, it is called the homotopy quotient of $X^{\wedge p}$. The cohomology groups $H^*(\Sp\infty \times_{\ZZ_p} X^{\wedge p})$ are the \emph{equivariant cohomology groups} of $X^{\wedge p}$ with respect to the $\ZZ_p$ action.

The connection between the Steenrod powers and equivariant cohomology can be appreciated by adapting an analogy offered by Steenrod: for each $\alpha \in H^n(X)$, the cup power $\alpha^p$ can be expressed as the pullback of the $p$-fold cross product $\alpha^{\otimes p} = \alpha \otimes \dotsb \otimes \alpha \in H^{pn}(X^p)$ over the diagonal map $X \hookrightarrow X^p$ that sends $x$ to $(x,\dotsc,x)$. We may think of the cross product $H^n(X) \to H^{pn}(X^p)$ as an ``external operation'' as its output does not stay within $H^*(X)$. In contrast, the cup square $H^n(X) \to H^{pn}(X)$ is an ``internal operation.'' Thus in this case, the internal operation is defined by first applying the external operation, then pulling back over a diagonal map.

In the same vein, the Steenrod powers $P^i : H^n(X) \to H^{n + 2i(p-1)}(X)$ are internal operations that are defined by first applying an external operation, $H^n(X) \to H^{pn}(\Sp\infty \times_{\ZZ_p} X^{\wedge p})$, and then pulling back over a ``diagonal map'' $\Lens_p^\infty \times X \hookrightarrow \Sp\infty \times_{\ZZ_p} X^p$. The diagonal map is defined by sending $(t,x)$ to $(\tilde{t},x,\dotsc,x)$, where $\tilde{t}$ is any preimage of $t$ under the covering map $\Sp\infty \to \Lens_p^\infty$.

\begin{theorem}[{\cite[pp.~152--155]{AdemMilgram_CohomFiniteGroups}, \cite[p.~504]{Hatcher_AlgTop}, \cite[p.~112]{Steenrod_CohomOps}}]
    \label{thm:ExternalOperation_Cyclic}
    For each $n \geq 1$, there is a unique family of operations $\Gamma : H^n(X) \to H^{pn}(\Sp{\infty} \times_{\ZZ_p} X^{\smsh p})$ for each cell complex $X$ that satisfies the following properties:
    \begin{enumerate}
        \item Let $i$ be an inclusion of a fiber in the fiber bundle
        \begin{equation}
            \label{eq:HomotopyQuotientFiberBundle}
            X^{\wedge p} \to \Sp{\infty} \times_{\ZZ_p} X^{\wedge p} \to \Lens_p^{\infty},
        \end{equation}
        which is induced by the projection onto the first factor. Then $\Gamma(\alpha)$ pulls back over $X^p \to X^{\smsh p} \xrightarrow{i} \Sp\infty \times_{\ZZ_p}X^{\smsh p}$ to $\alpha^{\otimes p} \in H^{pn}(X^p)$.

        \item Naturality: For every continuous map $f : X \to Y$ and $\alpha \in H^n(Y)$, let $\Gamma(f)$ denote the induced map $\Sp{\infty} \times_{\ZZ_p} X^{\smsh p} \to \Sp{\infty} \times_{\ZZ_p} Y^{\smsh p}$. Then $\Gamma(f)^*\Gamma(\alpha) = \Gamma(f^*\alpha)$.

        \item $\Gamma(\alpha \smile \beta) = \Gamma(\alpha) \smile \Gamma(\beta)$.
    \end{enumerate}

    Furthermore, the Steenrod powers are determined by the classes $\Gamma(\alpha)$ as follows. Choose generators $\omega_k \in H^k(\Lens_p^{\infty})$ that are dual to the $k$-cells in $\Lens_p^{\infty}$. Then $\Gamma(\alpha)$ pulls back along the diagonal inclusion $\Lens_p^{\infty} \times X \hookrightarrow \Sp{\infty} \times_{\ZZ_p} X^{\smsh p}$ to $\sum_k \omega_{pn-n-k} \otimes \theta_k(\alpha)$.
    \begin{itemize}
        \item When $p = 2$, $\Sq^i(\alpha) = \theta_i(\alpha)$.

        \item When $p = 2m+1$ is odd,
        \begin{align*}
            P^i(\alpha) &= (-1)^{i + mn(n+1)/2}(m!)^n\theta_{2i(p-1)}(\alpha)
            \\
            \beta P^i(\alpha) &= (-1)^{i + mn(n+1)/2+1}(m!)^n\theta_{2i(p-1)+1}(\alpha),
        \end{align*}
        and all other $\theta_k(\alpha)$ vanish.
    \end{itemize}
\end{theorem}

\subsection{An overview of the proof of \cref{thm:SteenrodPowers}}

From \cref{thm:ExternalOperation_Cyclic} $\Sq^i(\iota_m)$ is  related to $\Gamma(\iota_m)$ via the K\"unneth formula. In particular, $\sum_i \omega_{n-i} \otimes \Sq^i(\iota_m)$ has Brown representative
\begin{equation*}
\begin{aligned}
    & \RP^\infty \times \Z{k}(\Sp{n}) \to \Z{2k}(\Sp{2n})
    \\
    & (t,T) \mapsto F(t,T,T) = \bigcup_{x \in \{t\} \times_{\ZZ_2} T^{\smsh 2}} f(x).
\end{aligned}
\end{equation*}

We proved a result that expresses the K\"unneth formula geometrically in terms of Brown representatives. It implies that $\Sq^i(\iota_m)$ has the following Brown representative. Let $H$ denote the upper hemisphere of $\Sp{n-i}$:
\begin{equation}
    \label{eq:SteenrodSquareByGluing_Intuitive}
\begin{aligned}
    & \Z{k}(\Sp{n}) \to \Z{2k+n-i}(\Sp{2n})
    \\
    & T \mapsto F(H \times \{(T,T)\}) = \bigcup_{x \in H \times_{\ZZ_2} T^{\smsh 2}} f(x).
\end{aligned}
\end{equation}

We proved that the map in \cref{eq:SteenrodSquareByGluing_Intuitive} factors through $\cyc$. To prove this, observe that $H \times_{\ZZ_2} T^{\smsh 2} \subset \Sp{n-i} \times_{\ZZ_2} (\Sp{n})^{\smsh 2}$. It is useful to view $\Sp{n-i} \times_{\ZZ_2} (\Sp{n})^{\smsh 2}$ as a ``fiber bundle'' over $(\Sp{n})^{\smsh 2}/\ZZ_2$ induced by projection onto the second factor.

\subsection{Brown representatives of the images of $\Gamma$}
\label{sec:BrownRepImageExtOp}

By the naturality property in \cref{thm:ExternalOperation_Cyclic} and Brown representability, $\Gamma$ is determined by $\Gamma(\iota_n)$ for all $n$. First we will find Brown representatives for $\Gamma(\alpha)$, for the generators $\alpha$ of $H^n(\Sp{n})$. These representatives are maps $f_n : \Sp{\infty} \times_{\ZZ_p} (\Sp{n})^{\smsh p} \to \Z0(\Sp{pn})$, and we will construct them so that they are piecewise smooth when restricted to the subspace $\Sp{\bar{n}} \times_{\ZZ_p} (\Sp{n})^{\smsh p}$, where $\bar{n} = 2pn+1$. This subspace is a closed manifold so we can consider its space of cycles, and $f_n$ restricted to this manifold will have a gluing. From this gluing we can derive a Brown representative not exactly for $\Gamma(\iota_n) \in H^{pn}(\Sp\infty \times_{\ZZ_p} \Z{k}(\Sp{n+k})^{\smsh p})$, but for its image under the restriction homomorphism \begin{equation}
    H^{pn}(\Sp{\infty} \times_{\ZZ_p} X^{\smsh p}) \xrightarrow{\iso} H^{pn}(\Sp{\bar{n}} \times_{\ZZ_p} X^{\smsh p}),
\end{equation}
for $X = \Z{k}(\Sp{n+k})$. For general spaces $X$, the fact that this homomorphism is an isomorphism can be deduced from the Serre spectral sequences of the fiber bundle in \cref{eq:HomotopyQuotientFiberBundle} and the analogous bundle for $\Sp{\bar{n}} \times_{\ZZ_p} X^{\wedge p}$. Denote the images of classes $\Gamma(\alpha)$ under this isomorphism by $\Gamma^{\bar{n}}(\alpha)$. $\Gamma^{\bar{n}}$ is also natural in the sense that maps $f : X \to Y$ give rise to maps $\Gamma^{\bar{n}}(f) : \Sp{\bar{n}} \times_{\ZZ_p} X^{\wedge p} \to \Sp{\bar{n}} \times_{\ZZ_p} Y^{\wedge p}$.

Henceforth, by $K(\ZZ_p,n)$ we will mean a cell complex whose $n$-skeleton is $\Sp{n}$.

\begin{proposition}
    \label{prop:ExternalOperation_Sphere}
    For any $m \geq 1$, let $a : \Sp{m} \hookrightarrow K(\ZZ_p,m)$ denote the inclusion of the $m$-skeleton, which represents a generator of $H^m(\Sp{m})$, that we denote by $\alpha$. Then for some $\varepsilon > 0$, $\Gamma(\alpha)$ is represented by a continuous map $f_m : \Sp{\infty} \times_{\ZZ_p} (\Sp{m})^{\smsh p} \to \Z0(\Sp{pm})$ such that:
    \begin{itemize}
        \item $f_m|_{\Sp{\bar{m}} \times_{\ZZ_p} (\Sp{m})^{\smsh p}}$ is piecewise smooth.
    
        \item $f_m(t,x_1,\dotsc,x_p) = 0$ whenever $t$ lies in the $(p-1)m-1$-skeleton $(\Sp{\infty})^{(p-1)m-1}$ and $(x_1, \dotsb, x_p)$ is within distance $\varepsilon$ of the image of the diagonal $\Delta \subset (\Sp{m})^p$ in $(\Sp{m})^{\smsh p}$.
    \end{itemize}

    In addition, for topological spaces $X$ and $Y$, consider the map
    \begin{equation*}
        R : \Sp{\infty} \times_{\ZZ_p} (X \times Y)^{\smsh p} \to (\Sp{\infty} \times_{\ZZ_p} X^{\smsh p}) \times (\Sp{\infty} \times_{\ZZ_p}Y^{\smsh p}),
    \end{equation*}
    defined by $R(t,((x_i,y_i))_i) = ((t,(x_i)_i), (t,(y_i)_i))$. Then for any $m, n \geq 1$, $f_m$, $f_n$ and $f_{m+n}$ are related through the following diagram that commutes up to homotopy:
    \begin{equation}
        \label{eq:ExternalOperation_Sphere_CupProd}
    \begin{tikzcd}
        \Sp\infty \times_{\ZZ_p} (\Sp{m} \times \Sp{n})^{\smsh p} \arrow[d,"R"'] \rar{\varphi} & \Sp\infty \times_{\ZZ_p} (\Sp{m+n})^{\smsh p} \arrow[dd,"f_{m+n}"]
        \\
        (\Sp\infty \times_{\ZZ_p} (\Sp{m})^{\smsh p}) \times (\Sp\infty \times_{\ZZ_p} (\Sp{n})^{\smsh p}) \arrow[d,"f_m \times f_n"']
        \\
        \Z0(\Sp{pm}) \times \Z0(\Sp{pn}) \arrow[r,"\chi"'] & \Z0(\Sp{p(m+n)})
    \end{tikzcd}
    \end{equation}
\end{proposition}
\begin{proof}
    Let $\Delta \subset (\Sp{m})^{\smsh p}$ denote the diagonal, and consider the $(pm-1)$-dimensional cell complex $\tilde\Delta = \Sp{(p-1)m-1} \times_{\ZZ_p} \Delta$. Let $a : \Sp{m} \hookrightarrow K(\ZZ_p,m)$ be the inclusion of the $m$-skeleton, which generates $H^m(\Sp{m})$. Then $\Gamma(a)$ actually restricts to an embedding on $\tilde\Delta$. Observe that $\Lens_p^\infty \times \{*\} \cup \tilde\Delta$ can be formed from $\Lens_p^\infty \times \{*\}$ by attaching cells of dimension at most $pm-1$, which does not change the $pm$-dimensional cohomology group. Thus $\Gamma(\iota_m)$ still pulls back to zero over the inclusion $\Lens_p^\infty \times \{*\} \cup \tilde\Delta \hookrightarrow \Sp\infty \times_{\ZZ_p} K(\ZZ_p,m)^{\wedge p}$.

    In other words, we can represent $\Gamma(\iota_m)$ by a map $h : \Sp\infty \times_{\ZZ_p} K(\ZZ_p,m)^{\wedge p} \to \Z0(\Sp{pm})$ that becomes nullhomotopic when restricted to $\Lens_p^\infty \times \{*\}\cup \tilde\Delta$. By the homotopy extension property, we may assume that $h$ actually vanishes over $\Lens_p^\infty \times \{*\} \cup \tilde\Delta$. In fact, we can assume that $h$ vanishes over the union of $\Lens_p^\infty \times \{*\}$ and the closed $\varepsilon$-neighbourhood of the diagonal, for sufficiently small $\varepsilon$.
    
    By naturality, $\Gamma(\alpha) = \Gamma(a^*\iota_m) = \Gamma(a)^*\Gamma(\iota_m)$. Thus by Brown representability, we can represent $\Gamma(\alpha)$ by $f_m = h \circ \Gamma(a)$.
    
    Let us now prove that the diagram \eqref{eq:ExternalOperation_Sphere_CupProd} commutes. From our construction of $f_m$ and $f_n$, we know that $(f_m)^*(\iota_{pm}) = \Gamma(\alpha)$ and $(f_n)^*(\iota_{pn}) = \Gamma(\beta)$, where $\alpha : \Sp{m} \hookrightarrow K(\ZZ_p, m)$ and $\beta : \Sp{n} \hookrightarrow K(\ZZ_p, n)$ are the respective inclusions of the $m$-skeleton and $n$-skeleton.

        The maps in the left column compose into a map $\Sp\infty \times_{\ZZ_p} (\Sp{m} \times \Sp{n})^{\smsh p} \to \Z0(\Sp{pm}) \times \Z0(\Sp{pn})$ that is the pair of maps $(f_m \circ \Gamma(\pr_1), f_n \circ \Gamma(\pr_2))$. Therefore the composite of the left column and bottom row represents the cup product
        \begin{align*}
            \Gamma(\pr_1)^* (f_m)^* (\iota_{pm}) \smile 
            \Gamma(\pr_2)^* (f_n)^* (\iota_{pn})
            &= \Gamma(\pr_1)^* \Gamma(\alpha) \smile \Gamma(\pr_2)^* \Gamma(\beta)
            \\
            \text{(by naturality)} &= \Gamma(\pr_1^*(\alpha)) \smile \Gamma(\pr_2^*(\beta))
            \\
            \text{(by (3))} &= \Gamma(\pr_1^*(\alpha) \smile \pr_2^*(\beta))
            \\
            &= \Gamma(\alpha \times \beta).
        \end{align*}

        On the other hand, the composite of the top row and the right column represents $\Gamma(q)^*(f_{m+n})^*(\iota_{p(m+n)}) = \Gamma(q)^*\Gamma(\gamma)$, where $q$ is the quotient map of the smash product, and $\gamma : \Sp{m+n} \hookrightarrow K(\ZZ_p,m+n)$ is the inclusion of the $(m+n)$-skeleton. By naturality, $\Gamma(q)^*\Gamma(\gamma) = \Gamma(q^*(\gamma))$, but $\gamma \circ q$ represents $\alpha \times \beta$. The lemma statement follows by Brown representability.
\end{proof}

Let $\Phi_m$ denote a gluing for $f_m|_{\Sp{\bar{m}} \times_{\ZZ_p} (\Sp{m})^{\smsh p}}$, which exists by \cref{thm:PiecewiseSmoothGluing_Existence}.

\begin{proposition}
    \label{prop:ExternalOperation_CycleSpace}
    For any $n \geq 1$ and $k \geq 0$, $\Gamma^{\bar{n}}(\iota_n) \in H^{pn}(\Sp{\bar{n}} \times_{\ZZ_p} \Z{k}(\Sp{n+k})^{\smsh p})$ is represented by a map
    \begin{equation}
        \label{eq:ExternalOperation_CycleSpace}
        F_k^n : \Sp{\bar{n}} \times_{\ZZ_p} \Z{k}(\Sp{n+k})^{\smsh p} \to \Sp{\bar{n}} \times_{\ZZ_p} \Z{pk}((\Sp{n+k})^{\smsh p}) \to \Z{pk}(\Sp{\bar{n}} \times_{\ZZ_p} (\Sp{n+k})^{\smsh p}) \xrightarrow{\Phi_{n+k}^{pk}} \Z{pk}(\Sp{p(n+k)}),
    \end{equation}
    where the first map is induced by the Cartesian product in the second variable and the second map is $(t,A) \mapsto \{t\} \times A$.

    Furthermore, for any cell complex $X$ and cohomology class $\alpha \in H^n(X)$ that is represented by a map $a : X \to \Z{k}(\Sp{n+k})$, $\Gamma^{\bar{n}}(\alpha) = \Gamma^{\bar{n}}(a)(F_k^n)^*(\iota_{pn})$.
\end{proposition}
\begin{proof}
    In order to prove the proposition, the naturality of $\Gamma^{\bar{n}}$ implies that it suffices to prove that when $X$ is a cellular model for $K(\ZZ_p,n)$ with $\Sp{n}$ as its $n$-skeleton, then $\Gamma^{\bar{n}}(\iota_n) = \Gamma^{\bar{n}}(\iota_n)^*(F_k^n)^*(\iota_{pn})$. For any cell complex $X$ and $\alpha \in H^n(X)$, define the operation $\tilde{\Gamma}^{\bar{n}}(\alpha) = \Gamma^{\bar{n}}(\alpha)^*(F_k^n)^*(\iota_{pn})$ where Brown representability lets us view $\alpha$ as a map $X \to K(\ZZ_p,n)$.
    
    We verify that $\tilde{\Gamma}^{\bar{n}}$ satisfies properties (1)--(3) in \cref{thm:ExternalOperation_Cyclic}. The naturality property (2) follows directly from the definition and Brown representability. (1): We prove that $j^*\tilde{\Gamma}^{\bar{n}}(\iota_n) = \iota_n^{\otimes p}$, where $j : K(\ZZ_p,n)^{\smsh p} \hookrightarrow \Sp{\bar{n}} \times_{\ZZ_p} K(\ZZ_p,n)^{\smsh p}$ is an inclusion of the fiber. Consider another inclusion of a fiber, $i : \Z{k}(\Sp{n+k})^{\smsh p} \to \Sp{\bar{n}} \times_{\ZZ_p} \Z{k}(\Sp{n+k})^{\smsh p}$. By the definition of $F_k^n$, the composite map
    \begin{align*}
        & \Z{k}(\Sp{n+k})^{\smsh p} \xrightarrow{i} \Sp{\bar{n}} \times_{\ZZ_p} \Z{k}(\Sp{n+k})^{\smsh p} \xrightarrow{F_k^n} \Z{pk}(\Sp{p(n+k)})
        \\
        ={}& \Z{k}(\Sp{n+k})^{\smsh p} \xrightarrow{(A_1,\dotsc,A_p) \mapsto \{*\} \times A_1 \times \dotsb \times A_p} \Z{pk}(\Sp{\bar{n}} \times_{\ZZ_p} (\Sp{n+k})^{\smsh p}) \xrightarrow{\Phi_{n+k}^{pk}} \Z{pk}(\Sp{p(n+k)})
        \\
        ={}& \Z{k}(\Sp{n+k})^{\smsh p} \xrightarrow{(-)^{\smsh p}} \Z{pk}((\Sp{n+k})^{\smsh p}) \xrightarrow{(i_{n+k})_\sharp} \Z{pk}(\Sp{\bar{n}} \times_{\ZZ_p} (\Sp{n+k})^{\smsh p}) \xrightarrow{\Phi_{n+k}^{pk}} \Z{pk}(\Sp{p(n+k)}),
    \end{align*}
    where $(-)^{\smsh p}$ denotes the $p$-fold smash product. Observe that the composition law in \cref{prop:GluingProperties} implies that $f_{n+k} \circ i_{n+k}$ has a gluing, namely $\Phi_{n+k}^{pk} \circ (i_{n+k})_\sharp$.
    
    Property (1) of $\Gamma$ implies that $i_{n+k}^*(f_{n+k})^*(\iota_{p(n+k)}) = i_{n+k}^*\Gamma^{\bar{n}}(\alpha) = \alpha^{\otimes p}$ (where $\alpha : \Sp{n+k} \hookrightarrow K(\ZZ_p, n+k)$). Since $f_{n+k} \circ i_{n+k}$ represents $\alpha^{\otimes p}$, we can homotope it until it is the inclusion of the smash product, $(\Sp{n+k})^{\smsh p} \xrightarrow{h} \Z0(\Sp{p(n+k)})$. Note that the identity homomorphism is a gluing of $h$. Therefore, by \cref{lem:HomotopicPiecewiseSmoothChainHomotopy}, $\Phi_{n+k}^{pk} \circ (i_{n+k})_\sharp$ is homotopic to the identity map.
    
    To summarize, $F_k^n \circ i$ is homotopic to the $p$-fold smash product map $\Z{k}(\Sp{n+k})^p \to \Z{pk}(\Sp{p(n+k)})$, which represents the $p$-fold cross product by \cref{prop:CrossProduct_BrownRep}. That is, $F_k^n \circ i$ represents $\iota_n^{\otimes p}$.

    Now consider the corresponding commutative diagram for $X = K(\ZZ_p,n)$, induced by a map $\iota_n : K(\ZZ_p,n) \to \Z{k}(\Sp{n+k})$ that vanishes at the basepoint:
    \begin{equation*}
    \begin{tikzcd}
        K(\ZZ_p,n)^p \arrow[d,"j"'] \arrow[r,"\iota_n^{\times p}"] & \Z{k}(\Sp{n+k})^p \arrow[d,"i"]
        \\
        \Sp{\bar{n}} \times_{\ZZ_p} K(\ZZ_p,n)^p \arrow[r,"\Gamma^{\bar{n}}(\iota_n)"'] & \Sp{\bar{n}} \times_{\ZZ_p} \Z{k}(\Sp{n+k})^p \arrow[r,"F_k^n"'] & \Z{pk}(\Sp{p(n+k)})
    \end{tikzcd}
    \end{equation*}
    Since the top row is a weak homotopy equivalence, $j^*\tilde{\Gamma}^{\bar{n}}(\iota_n) = j^*\Gamma^{\bar{n}}(\iota_n)^*(F_k^n)^*(\iota_{pn}) = (\iota_n^{\times p})^*i^*(F_k^n)^*(\iota_{pn}) = \iota_n^{\otimes p}$, which demonstrates (1).

    (3): Let us prove that the following diagram commutes up to homotopy:
    \begin{equation}
        \label{eq:ExternalOperation_CycleSpace_CupProd}
    \begin{tikzcd}
        \Sp{\bar{n}} \times_{\ZZ_p} (\Z{j}(\Sp{m+j}) \times \Z{k}(\Sp{n+k}))^p \arrow[d,"R"'] \arrow[r,"\id \times_{\ZZ_p} \chi^p"] & \Sp{\bar{n}} \times_{\ZZ_p} \Z{j+k}(\Sp{m+n+j+k})^p \arrow[dd,"F_{j+k}^{m+n}"]
        \\
        (\Sp{\bar{n}} \times_{\ZZ_p} \Z{j}(\Sp{m+j})^p) \times (\Sp{\bar{n}} \times_{\ZZ_p} \times \Z{k}(\Sp{n+k})^p) \arrow[d,"F_j^m \times F_k^n"']
        \\
        \Z{pj}(\Sp{p(m+j)}) \times \Z{pk}(\Sp{p(n+k)}) \arrow[r,"\chi"'] & \Z{p(j+k)}(\Sp{p(m+n+j+k)})
    \end{tikzcd}
    \end{equation}
    
    Given $x \in \Sp{b}$ and $y \in \Sp{d}$, let $x \wedge y$ denote the image of $(x,y)$ under the smash product map $\Sp{b} \times \Sp{d} \to \Sp{b+d}$. Given $X \in \Z{a}(\Sp{b})$ and $Y \in \Z{c}(\Sp{d})$, let $X \wedge Y \in \Z{a+c}(\Sp{b+d})$ denote the image of $X \times Y$ under the smash product map. Starting with $(t,((X_i,Y_i))_i)$ in the top left corner, if we go clockwise then we get
    \begin{align}
        & F_{j+k}^{m+n} \circ (\id \times_{\ZZ_p} \chi^p)(t,((X_i,Y_i))_i)
        \notag\\
        ={}& F_{j+k}^{m+n}(t, (X_i \wedge Y_i)_i) \notag\\
        ={}& \Phi_{m+n+j+k}^{p(j+k)}(\{t\} \times X_1 \wedge Y_1 \times \dotsb \times X_p \wedge Y_p). \label{eq:ExternalOperation_CupProd_Clockwise}
    \end{align}
    Consider the map $\psi : \Sp{\bar{n}} \times_{\ZZ_p} (\Z{j}(\Sp{m+j}) \times \Z{k}(\Sp{n+k}))^p \to \Z{p(j+k)}(\Sp{\bar{n}} \times_{\ZZ_p} (\Sp{m+j} \times \Sp{n+k})^p)$ defined by $\psi(t,((X_i,Y_i))_i) = \{t\} \times X_1 \times Y_1 \times \dotsb \times X_p \times Y_p$. Then the clockwise direction is equal to $\Phi_{m+n+j+k}^{p(j+k)} \circ \varphi_\sharp \circ \psi$, where $\varphi$ comes from \cref{eq:ExternalOperation_Sphere_CupProd}. By the composition law in \cref{prop:GluingProperties}, $f_{m+n+j+k} \circ \varphi$ has a gluing, namely $\Phi_{m+n+j+k} \circ \varphi_\sharp$.
    
    On the other hand, going counterclockwise gives us
    \begin{align}
        & \chi \circ (F_j^m \times F_k^n) \circ R(t,((X_i,Y_i))_i)
    \notag\\
        ={}& \chi(F_j^m(t,(X_i)_i), F_k^n(t,(Y_i)_i))
    \notag\\
        ={}& \Phi_{m+j}^{pj}(\{t\} \times X_1 \times \dotsb \times X_p) \wedge \Phi_{n+k}^{pk}(\{t\} \times Y_1 \times \dotsb \times Y_p). \label{eq:ExternalOperation_CupProd_Counterclockwise}
    \end{align}
    The composition law in \cref{prop:GluingProperties} implies that $\chi \circ (f_{m+j} \times f_{n+k}) \circ R$ has a gluing, namely $\chi \circ (\Phi_{m+j} \times \Phi_{n+k}) \circ R_\sharp$.

    However, \cref{prop:ExternalOperation_Sphere} implies that the diagram in \cref{eq:ExternalOperation_Sphere_CupProd} commutes up to homotopy. By \cref{lem:HomotopicPiecewiseSmoothChainHomotopy}, this means that $\Phi_{m+n+j+k} \circ \varphi_\sharp$ is homotopic to $\chi \circ (\Phi_{m+j} \times \Phi_{n+k}) \circ R_\sharp$. Therefore the diagram in \cref{eq:ExternalOperation_CycleSpace_CupProd} commutes up to homotopy.
    
    Now suppose that two cohomology classes of $X$ are represented by maps $\alpha : X \to \Z{j}(\Sp{m+j})$ and $\beta : X \to \Z{k}(\Sp{n+k})$. The following diagram also commutes:
    \begin{equation*}
    \begin{tikzcd}
        \Sp{\bar{n}} \times_{\ZZ_p} X^p \arrow[d,equal] \arrow[r,"\id \times_{\ZZ_p} \Delta^p"] & \Sp{\bar{n}} \times_{\ZZ_p} (X \times X)^p \arrow[d,"R"] \arrow[r,"\Gamma(\alpha \times \beta)"] & \Sp{\bar{n}} \times_{\ZZ_p} (\Z{j}(\Sp{m+j}) \times \Z{k}(\Sp{n+k}))^p \arrow[d,"R"]
        \\
        \Sp{\bar{n}} \times_{\ZZ_p} X^p \arrow[r,"\Delta"'] & (\Sp{\bar{n}} \times_{\ZZ_p} X^p) \times (\Sp{\bar{n}} \times_{\ZZ_p} X^p) \arrow[r,"\Gamma^{\bar{n}}(\alpha) \times \Gamma(\beta)"'] & (\Sp{\bar{n}} \times_{\ZZ_p} \Z{j}(\Sp{m+j})^p) \times (\Sp{\bar{n}} \times_{\ZZ_p} \Z{k}(\Sp{n+k})^p)
    \end{tikzcd}
    \end{equation*}
    Together, the last two diagrams imply that $\tilde{\Gamma}^{\bar{n}}(\alpha \smile \beta) = \tilde{\Gamma}^{\bar{n}}(\alpha) \smile \tilde{\Gamma}^{\bar{n}}(\beta)$. This proves (3).
\end{proof}

\begin{theorem}
    \label{thm:CyclicFactorInExternalOperation}
    Let $\iota_n$ denote a generator of $H^n(\Z{m}(\Sp{n+m}))$. Then for all integers $k$ such that $1 \leq 2k \leq n$, $P^k(\iota_n) = \cyc^*(\alpha_{2k(p-1)})$ and $\beta P^k(\iota_n) = \cyc^*(\alpha_{2k(p-1)+1})$ for some nonzero classes $\alpha_j \in H^{n+j}(\Z{pm}(L_{n+m} \times \D^{n+m+1}, L_{n+m} \times \partial\D^{n+m+1}))$ where $L_{n+m}$ is the lens space from \cref{eq:SphereQuotientLensBundle}.

    In fact, $\alpha_j$ has a Brown representative $\Psi_j$ which is a gluing of a map $h_j : L_{n+m} \times \D^{n+m+1} \to \Z{pn-n-j}(\Sp{p(n+m)})$ that vanishes over all $(x_1,\dotsc,x_p)$ for which some $x_i = *$.
\end{theorem}
\begin{proof}
    Consider the maps $f_n$ and $F_k^n$ from \cref{prop:ExternalOperation_Sphere,prop:ExternalOperation_CycleSpace}.
    
    Let $j \geq 1$. By \cref{prop:ExternalOperation_CycleSpace}, the class $\theta_j(\iota_n)$ from \cref{thm:ExternalOperation_Cyclic} can be derived by applying the K\"unneth formula to the following composite map,
    \begin{equation}
        \Lens_p^{\bar{n}} \times \Z{m}(\Sp{n+m}) \hookrightarrow \Sp{\bar{n}} \times_{\ZZ_p} \Z{m}(\Sp{n+m})^{\smsh p} \xrightarrow{F_m^n} \Z{pm}(\Sp{p(n+m)}).
    \end{equation}

    \begin{lemma}
        $\theta_j$ has a Brown representative which is the map $\Z{m}(\Sp{n+m}) \to \Z{p(m+n) - n - j}(\Sp{p(n+m)})$ defined by $T \mapsto F_m^n(\tilde{C} \times \{(T,\dotsc,T)\})$ where $\tilde{C}$ is a lift, over the covering space $\Sp{\bar{n}} \to \Lens_p^{\bar{n}}$, of a cellular $(pn - n - j)$-cycle $C$ in $\Lens_p^{\bar{n}}$ that represents the appropriate generator in $H_{pn - n - j}(\Lens_p^{\bar{n}})$.
    \end{lemma}
    \begin{proof}
        Note that $F_m^n(\tilde{C} \times \{(T,\dotsc,T)\}) = \Phi_{n+m}(\tilde{C} \times_{\ZZ_p} T^{\smsh p})$. Now find a map $K \to \Z{m}(\Sp{n+m})$ where $K$ is a high skeleton of $K(\ZZ_p,n)$, and extend it to a map $Y \to \Z{m}(\Sp{n+m})$ where $Y$ is a tubular neighbourhood of $K$ embedded in a Euclidean space of large dimension. Homotope this map $Y \to \Z{m}(\Sp{n+m})$ until it is piecewise smooth using \cref{thm:PiecewiseSmoothApprox}. Now apply \cref{cor:KunnethFormula}.
    \end{proof}
    Since $j \geq 1$, we have $pn - n - j \leq (p-1)(n+m)-1$, that is, $\tilde{C} \subset (\Sp{\bar{n}})^{(p-1)(n+m)-1}$. Together with \cref{prop:ExternalOperation_Sphere}, this implies that for all $x \in \Delta$, we have $\Phi_{n+m}(\tilde{C} \times \{x\}) = 0$. We will show that this map $\Z{m}(\Sp{n+m}) \to \Z{p(n+m) - n - j}(\Sp{p(n+m)})$ factors into some composite map
    \begin{equation*}
        \Z{m}(\Sp{n+m}) \xrightarrow{\cyc} \Z{pm}(L_{n+m} \times \D^{n+m+1}, L_{n+m} \times \partial\D^{n+m+1}) \to \Z{p(n+m) - n - j}(\Sp{p(n+m)}).
    \end{equation*}

    We are looking at the map
    \begin{equation}
        \tilde{C} \times \{(T,\dotsc,T)\} \hookrightarrow \Sp{\bar{n}} \times_{\ZZ_p} \Z{m}(\Sp{n+m})^{\smsh p} \xrightarrow{F_m^n} \Z{pm}(\Sp{p(n+m)}).
    \end{equation}

    Recall the homeomorphism $h : (\Sp{p(n+m+1)-1} \setminus \Delta)/\ZZ_p \to L_{n+m} \times \itr \D^{n+m+1}$ from \cref{eq:SphereQuotientLensBundle}. Applying \cref{cor:KunnethFormula} to the above map gives
    \begin{align*}
        & F_m^n\left(\tilde{C} \times \left \{(T, \dotsc, T)\right\}\right)
        \\
        ={}& \Phi_{n+m}(\tilde{C} \times_{\ZZ_p} T^p)
        \\
        \eqnote{\cref{prop:ExternalOperation_Sphere}} ={}& \Phi_{n+m}(\tilde{C} \times_{\ZZ_p} (T^p \setminus \Delta)),
        \intertext{and observing using \cref{eq:CycProdMap_Intuitive} that $T^p \setminus \Delta = \gamma^{-1}(h^{-1}(\cyc(T)))$, where $\gamma : \Sp{p(n+m+1)-1} \setminus \Delta \to (\Sp{p(n+m+1)-1} \setminus \Delta)/\ZZ_p$ is the covering map, and thus $T^p \setminus \Delta$ is the sum of translates $a \cdot D$ of a fundamental domain $D$ by all $a \in \ZZ_p$,}
        ={}& 
        \Phi_{n+m} \left(\tilde{C} \times_{\ZZ_p} \sum_{a \in \ZZ_p} (a \cdot D) \right)
        \\
        ={}& 
        \Phi_{n+m} \left(\sum_{a \in \ZZ_p} (a \cdot \tilde{C}) \times_{\ZZ_p} D \right)
        \\
        ={}& 
        \Phi_{n+m} (\Sp{pn-n-j} \times_{\ZZ_p} D).
    \end{align*}
    
    Therefore for $1 \leq j \leq (p-1)n$, $\theta_j(\iota_n)$ has a Brown representative which is the composite map
    \begin{equation*}
        \Z{m}(\Sp{n+m}) \xrightarrow{\cyc} \Z{pm}(L_{n+m} \times \D^{n+m+1}, L_{n+m} \times \partial\D^{n+m+1}) \xrightarrow{\Psi_j} \Z{p(n+m)-n-j}(\Sp{p(n+m)}),
    \end{equation*}
    where $\Psi_j$ is a gluing of the desired map $h_j : L_{n+m} \times \D^{n+m+1} \to \Z{pn-n-j}(\Sp{p(n+m)})$ defined by $h_j(x) = \Phi_{n+m}(\Sp{pn-n-j} \times_{\ZZ_p} \{y\})$, where $y$ is any one of the $p$ points in $\gamma^{-1}(h^{-1}(x))$. $h_j$ does not depend on the choice of $y$ because of the equivalence relation involved in $\times_{\ZZ_p}$, and because the $\ZZ_p$ action on $\Sp{\bar{n}}$ is orientation-preserving when $p$ is odd and the orientation is irrelevant when $p = 2$.

    Finally, for all $x = (x_1,\dotsc,x_p)$ such that some $x_i = *$, $f_{n+m}(-,x) = 0$ by \cref{prop:ExternalOperation_Sphere}, so $h_j(x) = 0$.
\end{proof}

Note that we could also have defines $h_j$ in terms of the map $\Lambda^{n+m}_{pn-n-j}$ from the following proposition.

\begin{proposition}
    \label{prop:ExternalOperation_Sphere_OrbitSpaceMap}
    Consider the map $f_m : \Sp\infty \times_{\ZZ_p} (\Sp{m})^{\smsh p} \to \Z0(\Sp{pm})$ from \cref{prop:ExternalOperation_Sphere}. Then for each $1 \leq k \leq (p-1)m - 1$ there is a map $\Lambda^m_k : (\Sp{pm}/\hat\Delta)/\ZZ_p \to \Z{k}(\Sp{pm})$, where $\Sp{pm} = (\Sp{m})^{\wedge p}$, defined by $\Lambda^m_k(x_1,\dotsc,x_p) = \Phi_m(\Sp{k} \times_{\ZZ_p} \{(x_1,\dotsc,x_p)\})$. 
    
    Furthermore, given that $(\Sp{pm}/\hat\Delta)/\ZZ_p \he \Th(L \times \RR^{m+1})$ where $L$ is a lens space by \cref{lem:SphereQuotientSuspLens_Real}, $\Lambda^m_k$ represents the image of $\omega_{(p-1)m-1-k} \in H^{(p-1)m-1-k}(L)$ under the Thom isomorphism.
\end{proposition}
\begin{proof}
    Let us check that $\Lambda_k^m$ is well-defined and continuous. Let $\Delta_\varepsilon$ denote the open tubular neighbourhood of radius $\varepsilon$ around the diagonal $\Delta \subset (\Sp{m})^{\smsh p}$. Write $(\Sp{m})^{\smsh p}_\varepsilon = (\Sp{m})^{\smsh p} \setminus \Delta_\varepsilon$. We can define a fiber bundle $\xi : \Sp{k} \times_{\ZZ_p} (\Sp{m})^{\smsh p}_\varepsilon \to \Sp\infty \times_{\ZZ_p} (\Sp{m})^{\smsh p}_\varepsilon \to \Sp{pm}_\varepsilon/\ZZ_p$ by projection onto the second factor. Observe that $\Lambda_k^m$ vanishes over $\Delta_\varepsilon/\ZZ_p$ and otherwise is equal to $\xi_*f_m|_{\Sp{k} \times_{\ZZ_p} (\Sp{m})^{\smsh p}_\varepsilon}$. Thus $\Lambda_k^m$ is continuous.

    Consider the map $f : \Sp{k} \times_{\ZZ_p} \Sp{pm} \hookrightarrow \Sp\infty \times_{\ZZ_p} \Sp{pm} \xrightarrow{f_m} \Z0(\Sp{pm})$. Since $k \leq (p-1)m - 1$, $f$ vanishes over $\Lens_p^k \times \Delta$. Then $f$ represents a cohomology class in $H^{pm}(\Sp{k} \times_{\ZZ_p
    } \Sp{pm}_\varepsilon, \Sp{k} \times_{\ZZ_p
    } \partial\Sp{pm}_\varepsilon)$.
    
    Now consider the fiber bundle $\Sp{k} \to \Sp{k} \times_{\ZZ_p
    } N \xrightarrow{\pi} N/\ZZ_p$ where $\pi$ is induced by projection onto the second coordinate. (This is a fiber bundle because $\ZZ_p$ acts freely on $N$.) $\pi$ is also a map of pairs $\pi : (\Sp{k} \times_{\ZZ_p
    } N, \Sp{k} \times_{\ZZ_p
    } \partial N) \to (N/\ZZ_p, \partial N/\ZZ_p)$.

    By excision, $\Lambda^m_k$ represents the same relative cohomology class in $H^{pm-k}(N/\ZZ_p, \partial N/\ZZ_p)$ as the cohomology pushforward $\pi_*f$ as defined in \cref{prop:GeomRep_PushforwardCohom}. Now we have to compute (relative) Poincar\'e duality in $\Sp{k} \times_{\ZZ_p
    } N$ and $N/\ZZ_p$, and the homology pushforward $\pi_*$. Thus we need to compute this pushforward map.

    We can exploit the nice geometry around the fixed points of the action. The proof of \cref{lem:SphereQuotientSuspLens_Real} shows that $N$ is homeomorphic to $\Sp{(p-1)m - 1} \times \D^{m+1}$, where $\ZZ_p$ acts freely on $\Sp{(p-1)m - 1}$ with quotient $L = \Lens_p(1^m, \dotsc, (\frac{p-1}2)^m)$. Moreover, $\ZZ_p$ acts trivially on the fibers, so in fact $\Sp{k} \times_{\ZZ_p
    } N \he (\Sp{k} \times_{\ZZ_p} \Sp{(p-1)m - 1}) \times \D^{m+1}$. The action preserves the boundary so $\Sp{k} \times_{\ZZ_p
    } \partial N \he (\Sp{k} \times_{\ZZ_p} \Sp{(p-1)m - 1}) \times \partial\D^{m+1}$. Now consider the following commutative diagram:
    \begin{equation*}
    \begin{tikzcd}
        H^{pm}((\Sp{k} \times_{\ZZ_p} \Sp{(p-1)m - 1}) \times \D^{m+1}, (\Sp{k} \times_{\ZZ_p} \Sp{(p-1)m - 1}) \times \partial\D^{m+1}) \rar{\pi_*} \arrow[d,"\text{Thom}"'] & H^{pm-k}(L \times \D^{m+1}, L \times \partial\D^{m+1}) \arrow[d,"\text{Thom}"]
        \\
        H^{(p-1)m-1}(\Sp{k} \times_{\ZZ_p} \Sp{(p-1)m - 1}) \arrow[r,"\text{pushforward}"',"(\pr_2)_*"] & H^{(p-1)m-1-k}(L)
    \end{tikzcd}
    \end{equation*}
    To be more specific, a map $h : (\Sp{k} \times_{\ZZ_p} \Sp{(p-1)m - 1}) \times \D^{m+1} \to \Z{a}(\Sp{b})$ that vanishes on the boundary of its domain corresponds under the Thom isomorphism to a map $h' : (\Sp{k} \times_{\ZZ_p} \Sp{(p-1)m - 1}) \to \Z{a+m+1}(\Sp{b})$ defined by $h'(x) = h(x \times \D^{m+1})$. Therefore restricting $h'$ via the inclusion $\{*\} \times \Sp{(p-1)m-1} \hookrightarrow (\Sp{k} \times_{\ZZ_p} \Sp{(p-1)m - 1})$ gives the fiber integral of the restriction of $h$ via the inclusion $(\{*\} \times \Sp{(p-1)m-1} \times \D^{m+1}, \{*\} \times \Sp{(p-1)m-1} \times \partial\D^{m+1}) \hookrightarrow  (\Sp{k} \times_{\ZZ_p} \Sp{(p-1)m - 1}) \times \D^{m+1}, (\Sp{k} \times_{\ZZ_p} \Sp{(p-1)m - 1}) \times \partial\D^{m+1})$. This induces a map $H^{b-a}(\Sp{(p-1)m-1} \times \D^{m+1}, \Sp{(p-1)m-1} \times \partial\D^{m+1}) \to H^{b-a-(m+1)}(\Sp{(p-1)m-1})$ that is the inverse of the suspension isomorphism. Since $f$ is supposed to restrict to the fundamental class over each copy of $\Sp{pm}$, the Thom isomorphism sends it to cohomology class $\phi \in H^{(p-1)m-1}(\Sp{k} \times_{\ZZ_p} \Sp{(p-1)m - 1})$ that restricts to the fundamental class over each copy of $\Sp{(p-1)m-1}$.
    
    Consider now the Serre spectral sequence of the fiber bundle $\Sp{k} \times_{\ZZ_p} \Sp{(p-1)m - 1} \to \Sp{k}$, whose $E_2$ page looks like this:
    \begin{equation*}
    \begin{tikzcd}
        (p-1)m-1 & \ZZ_p & \ZZ_p & \cdots & \ZZ_p
        \\
        \\
        \\
        0 & \ZZ_p & \ZZ_p & \cdots & \ZZ_p
        \\
          & 0 & 1 & \cdots & k
    \end{tikzcd}
    \end{equation*}

    All differentials on this and future pages are zero because their domain or codomain is 0. Thus the $\phi$ must correspond to the only class in $E_2 = E_\infty$ that restricts to the fundamental class on a copy of $\Sp{(p-1)m - 1}$, namely $1 \in E_2^{0,(p-1)m-1}$.

    To compute the pushforward of $\phi$ along the map $(\pr_2)_* : H^{(p-1)m-1}(\Sp{k} \times_{\ZZ_p} \Sp{(p-1)m - 1}) \to H^{(p-1)m-1-k}(L)$, we evaluate the pushforward on the generator $\alpha \in H_{(p-1)m-1-k}(L)$ that corresponds to $1 \in \ZZ_p$. By the computations in the proof of \cref{prop:GeomRep_PushforwardCohom}, if $\gamma \in H^k{L}$ is the Poincar\'e dual of that homology generator $\alpha$, then $((\pr_2)_*)\phi(\alpha) = (\phi \smile \pr_2^*\gamma)[\Sp{k} \times_{\ZZ_p} \Sp{(p-1)m - 1}]$. We have $\gamma = 1 \in \ZZ_p$; then we have to identify where $\pr_2^*\gamma$ is on the $E_2$ page of the above spectral sequence.

    Consider the ``diagonal inclusion'' $\delta : \Lens_p^k \hookrightarrow \Sp{k} \times_{\ZZ_p} \Sp{(p-1)m - 1}$ defined by $[x] \mapsto [(x,x)]$. Note that ${\pr_2} \circ \delta$ is the inclusion of the $k$-skeleton so $\delta^*\pr_2^*\gamma = 1 \in \ZZ_p$. Moreover, ${\pr_1} \circ \delta$ is the identity map so $1 \in \ZZ_p \iso H^k(\Lens_p^k)$ pulls back over $\pr_1$ to $1 \in \ZZ_p \iso E_2^{k,0}$. Thus $\pr_2^*\gamma = 1 \in \ZZ_p \iso E_2^{k,0}$.

    Therefore the cup product of $1 \in E_2^{0,(p-1)m-1}$ and $1 \in E_2^{k,0}$ is $1 \in E_2^{k,(p-1)m-1}$. Thus $((\pr_2)_*\phi)(\alpha) = (\phi \smile \pr_2^*\gamma)[\Sp{k} \times_{\ZZ_p} \Sp{(p-1)m - 1}] = 1$. Thus $(\pr_2)_*\phi = \omega_{(p-1)m-1-k}$.

    Hence, $\pi_*f$ is the Thom isomorphism of $\omega_{(p-1)m-1-k}$, which is the same class represented by $\Lambda_k^m$.
\end{proof}

This brings us to a proof of the formal version of \cref{thm:SteenrodPowers}.

\begin{proof}[Proof of \cref{thm:SteenrodPowers_Formal}]
    Consider the classes $\alpha_j$ from \cref{thm:CyclicFactorInExternalOperation} and their Brown representatives $\Psi_j$. From \cref{Thm:AlmgrenIsoFMetric} and the Thom Isomorphism Theorem we know that $\Z{k}(L_n \times \D^{n+1}, L_n \times \partial\D^{n+1}) \whe \prod_{j = n+1-k}^{p(n+1)-1-k} K(\ZZ_p, j)$. Since each $\alpha_j$ has a Brown representative that is a gluing, \cref{prop:GluingAsBrownRep} implies that actually $\alpha_j = a\iota_j$ for some $a \in \ZZ_p$. \Cref{prop:ExternalOperation_Sphere_OrbitSpaceMap} implies that $a \neq 0$. Thus the Steenrod powers and Bockstein homomorphisms have Brown representatives as given in the theorem statement.
\end{proof}

\subsection{Brown representatives for the Steenrod squares}

\begin{proof}[Proof of \cref{thm:SteenrodSquares}]
    From \cref{thm:SteenrodPowers} we know that $\Sq^i$ has a Brown representative which is the composition of $\cyc$ with the standard gluing of a Brown representative $g_q$ of a class in $H^*(\RP^n \times \D^{n+1}, \RP^n \times \partial\D^{n+1})$, where $q = n+k+i$.
    
    To write $\cyc$ down explicitly, we will use the formula from \cref{eq:CycProdHomeoLensBundle_DFT} for the homeomorphism $h$ from \cref{eq:SphereQuotientLensBundle}. Since $p = 2$, $F = \frac1{\sqrt2}\left[ \begin{smallmatrix}
        1 &  1 \\
        1 & -1
    \end{smallmatrix} \right]$. The formula for $\cyc$ in \cref{eq:CycProdMap_Intuitive} says that to evaluate $\cyc(V \cap \Sp{m})$, we look at the we take unordered pairs of distinct points $x,y \in V \cap \Sp{m}$, and look at the locus of the pairs $(\vspan\{x - y\}, \frac{x+y}2)$ in $\RP^n \times \D^{n+1}$.
    
    $g_q$ can be constructed using \cref{lem:ProjLensCohomGens,lem:ThomSpace_BrownRep,cor:CupProd_BrownRep}, starting from the observation that the class represented by $g_q$ in the Thom space $(\RP^n \times \D^{n+1})/(\RP^n \times \partial\D^{n+1})$ corresponds, under the Thom isomorphism, to the generator of $H^{k+i-1}(\RP^n)$. Thus we may choose $g_q$ to be the map
    \begin{align*}
        & g_q : \RP^n \times \D^{n+1} \to \Z{n(k+i-1)}(\D^{(n+1)(k+i)}, \partial\D^{(n+1)(k+i)})
        \\
        & g_q(\ell,v) = ((\ell^\perp)^{k+i-1} \times \{v\}) \cap \D^{(n+1)(k+i)}.
    \end{align*}
    Combining the formulas for $\cyc$ and $g_q$ gives us the formula for $\sq^i$ in \cref{eq:SteenrodSquare}.

    Now suppose that $V \cap \Sp{n}$ is a planar cycle, where $V$ is a $(k+1)$-dimensional affine subspace of $\RR^{n+1}$. Let $\mathcal{L}$ be the set of lines in $\RR^{n+1}$ through the origin that are parallel to $V$. Then it suffices to prove that
    \begin{equation*}
        \cyc(V \cap \Sp{n}) = \bigcup_{\ell \in \mathcal{L}} \{\ell\} \times (\ell^\perp \cap V \cap \D^{n+1}).
    \end{equation*}

    ($\subset$): For any unordered pair $x, y \in V \cap \Sp{m}$ such that $x \neq y$, let $\ell = \vspan\{x - y\} \in \RP(V_0)$. Then since $V \cap \D^{m+1}$ is a convex set containing $x$ and $y$, it also contains $\frac{x+y}2$. Observe that $\ell^\perp$ is orthogonal to the chord $xy$. The Pythagorean theorem implies that $\ell^\perp$ must bisect the chord, thus $\frac{x+y}2 \in \ell^\perp$.
    
    ($\supset$): Choose any $\ell \in \RP(V_0)$ and $z \in \ell^\perp \cap V \cap \D^{m+1}$, so that $\norm{z} < 1$. Let $\ell'$ be the line parallel to $\ell$ that passes through $z$. $\ell'$ should intersect $V \cap \Sp{m}$ in two distinct points, $x$ and $y$. Thus $\ell = \vspan\{x - y\}$. Since $\ell' \perp \ell^\perp \ni z$, the Pythagorean theorem implies that $\norm{x - z}^2 + \norm{z}^2 = \norm{y - z}^2 + \norm{z}^2 = 1$, so in fact $z = \frac{x+y}2$.
\end{proof}

\subsection*{Acknowledgements}

The author is grateful to their academic advisor Alexander Nabutovsky for suggesting this project, and to Bruno Staffa for many enlightening conversations about $\delta$-localized families of cycles. The author has also appreciated many useful conversations with Alexander Nabutovsky, Regina Rotman, Yevgeny Liokumovich, Bruno Staffa, and Fedor Manin. The author would like to thank Fedor Manin for hosting them at the University of California, Santa Barbara in 2024, and the Hausdorff Research Institute for Mathematics in the University of Bonn for its hospitality during the 2025 Metric Analysis Trimester. Part of this work was completed within those stimulating research environments.

This research was enabled by the financial support from the 2021--2022 Vanier Canada Graduate Scholarship, the 2025--2026 University of Toronto Faculty of Arts and Science Doctoral Excellence Scholarship, and the 2025--2026 Margaret Isobel Elliott Graduate Scholarship.

\appendix

\section{Geometric Measure Theory}
\label{sec:Appendix_GMT}

\subsection{The Standard Definitions of Mod $p$ Integral Currents and Cycles}
\label{sec:StandardIntegralCurrents}

We will show that our definitions from \cref{sec:Definitions} are compatible with the standard definitions of the space of mod $p$ integral currents.

Let $\M^\ext$, $\Fl^\ext$, $\I[ext]k(M;G)$, and $\Z[ext]k(M,N;G)$ denote respectively the extrinsically defined mass functional, flat metric, spaces of integral currents in $M$ and integral relative cycles with coefficients in $G$. These are defined by fixing some Riemannian embedding $e : M \hookrightarrow \RR^n$ which is always possible due to the Nash embedding theorem. This embedding induces an embedding $\tilde\Gr_k(M) \hookrightarrow M \times \Gr_k(\RR^n)$. The standard definition of a varifold is that is a Radon measure on $M \times \Gr_k(\RR^n)$. When $R$ is a rectifiable current, the standard definition of the varifold $\var[ext]{R}$ associated with $R$ is so that for all $A \subset M \times \Gr_k(\RR^n)$,
\begin{equation}
    \label{eq:RectVarifold_Ext_Def}
    \begin{gathered}
        \var[ext]{R}(A) 
        = \int_{\{x \in U ~:~ (x, \Tan^k(\Haus^k \res U,x)) \in A\}} \Theta^k(\meas{R}, -) \,d\Haus^k,
        \\
        \text{where } U 
        = \{x \in \RR^n ~:~ \Theta^k(\meas{R}, x) \geq 1 \}.
    \end{gathered}
\end{equation}

It is clear that our definitions from \cref{sec:Definitions} could have used the coefficient group $\ZZ$ instead of $\ZZ_p$, with $\abs{a}$ being the usual absolute value for $a \in \ZZ$. When we use coefficient group $G$, we denote the sets and groups defined by $I_k(M;G)$, $I_k(M, N; G)$, $\I{k}(M,N;G)$, and $\Z{k}(M,N;G)$.

In what follows, let $G$ be $\ZZ$ or $\ZZ_p$ for any prime $p$. For each Lipschitz singular chain $T = \sum_i a_i\sigma_i$ with coefficients $a_i \in G$, the \emph{current parametrized by $T$} is $\Curr(T) = \sum_i a_i (e \circ \sigma_i)_\sharp (\Delta^k)$, where $\Delta^k$ is treated as a $k$-dimensional current with its standard orientation. When $G = \ZZ$, $\Curr(T)$ is rectifiable due to \cite[4.1.28(3)]{Federer_GMT}, and clearly $\spt \Curr(T) \subset \bigcup_i \im\sigma_i \subset M$. $\partial \Curr(T)$ must also be rectifiable because it is the current parametrized by $\partial T$, thus $\Curr(T)$ is actually an integral current. This gives a map $\Curr : I_k(M; \ZZ) \to \tilde{\mathcal{I}}_k(M; \ZZ)$. When $G = \ZZ_p$, considering congruence classes shows that we have an analogous map $\Curr : I_k(M; \ZZ_p) \to \tilde{\mathcal{I}}_k(M; \ZZ_p)$.

For any $T \in I_k(M,N; \ZZ)$, define $\Curr(T) = \Curr(\tilde{T}) \res (M \setminus N)$ for any $\tilde{T} \in T$.

\begin{lemma}
    \label{lem:Varifold_CompareIntExt}
    For any $T \in I_k(M,N;\ZZ)$ and $A \subset \tilde\Gr_k(M)$, $\var{T}_N(A) = \var[ext]{\Curr(T)}(A)$.
\end{lemma}
\begin{proof}
    Throughout this proof, we will abbreviate ``$\Haus^k$-almost'' to ``almost.'' Let $\sum_i a_i \sigma_i \in T$. Choose a standard simple $k$-vector field $\beta$ on $\Delta^k$ that is the wedge of a positively oriented orthonormal basis. By the proof of \cite[4.1.28(4)]{Federer_GMT}, we have $\Curr(T) = (\Haus^k \res B) \wedge \eta$ where \cite[4.1.25]{Federer_GMT} implies that for almost all $y \in \RR^n$,
    \begin{align*}
        \eta(y) = \sum_i a_i\sum_{x \in \sigma_i^{-1}(y)} \frac{(\bigwedge_k D\sigma_i)(\beta(x))}{J_k\sigma_i} && \text{and} && B = \{y \in \RR^n : \eta(y) \neq 0\}.
    \end{align*}
    
    Let us show that we can replace $B$ with $\bigcup_i \im\sigma_i$. The proof of \cite[4.1.28(4)]{Federer_GMT} implies that $B$ is almost contained in $\bigcup_i \im \sigma_i$. In fact, $\bigcup_i \im \sigma_i$ is also almost contained in $B$: since $T$ is non-overlapping, for almost all $y \in \bigcup_i \im\sigma_i$ there is a unique $x \in \Delta^k$ and a unique $\sigma_i$ such that $\sigma_i(x) = y$. Moreover, for almost all such $y$, the proof of \cite[4.1.25]{Federer_GMT} implies that $\sigma_i$ is differentiable at $x$ and $\im D\sigma_i = \Tan^m(\Haus^k \res \bigcup_i\im\sigma_i, y)$, which by \cite[3.2.19]{Federer_GMT} is a $k$-dimensional vector space for almost all such $y$. Therefore $\eta(y) = a_i(\bigwedge_k D\sigma_i)(\beta(x))/J_k\sigma_i \neq 0$, which by \cite[4.1.25]{Federer_GMT} implies that $\norm{\eta(y)} = \abs{a_i}$. Thus we may write $R = (\Haus^k \res \bigcup_i \im\sigma_i) \wedge \eta$.
    
    For almost all $y \in \bigcup_i\im\sigma_i$, $y$ lies in the image of some unique $\sigma_i$, in which case the proof of \cite[4.1.28]{Federer_GMT} implies that $\Theta^k(\meas{\Curr(T)}, y) = \norm{\eta(y)} = \abs{a_i}$. Therefore \cref{eq:RectVarifold_Ext_Def} simplifies to
    \begin{align*}
        \var[ext]{\Curr(T)}(A \setminus \tilde\Gr_k(N)) 
        &= \int_{\{x \in \bigcup_i \im\sigma_i \setminus N ~:~ (x, \Tan^k(\Haus^k \res \bigcup_i \im\sigma_i,x)) \in A\}} \Theta^k(\meas{\Curr(T)}, -) \,d\Haus^k
        \\
        \eqnote{$T$ is non-overlapping} &= \sum_i \int_{\{x \in \im\sigma_i \setminus N ~:~ (x, \Tan^k(\Haus^k \res  \im\sigma_i,x)) \in A\}} \Theta^k(\meas{\Curr(T)}, -) \,d\Haus^k
        \\
        &= \sum_i \abs{a_i} \underbrace{\Haus^k(\{x \in \im\sigma_i \setminus N ~:~ (x, \Tan^k(\Haus^k \res  \im\sigma_i,x)) \in A\})}_{\var{\sigma_i}_N(A)}
        \\
        &= \var{T}_N(A).
    \end{align*}
\end{proof}

\begin{lemma}
    \label{lem:RepModp_Intrinsic}
    For any $T = \sum_i a_i\sigma_i \in I_k(M;\ZZ_p)$, let $\tilde{T} = \sum_i \abs{a_i}\sigma_i \in I_k(M; \ZZ)$. Then $\M(T) = \M(\tilde{T})$, $\var{T}_N = \var{\tilde{T}}_N$, and $\Curr(\tilde{T})$ is representative mod $p$.
\end{lemma}
\begin{proof}
    The first two claims follow directly from the respective definitions. For the third claim, recall from the proof of \cref{lem:Varifold_CompareIntExt} that for $\Haus^k$-almost all $x \in \bigcup_i \im\sigma_i$, $\Theta^k(\meas{\Curr(\tilde{T})}, x) = \abs{a_i} \leq p/2$ for some $i$. The third claim then follows from \cite[p.~130]{Federer_GMT}.
\end{proof}

\begin{lemma}
    \label{lem:Mass_CompareIntExt_Absolute}
    For any $T \in I_k(M; G)$, $\M(T) = \M^\ext(\Curr(T))$.
\end{lemma}
\begin{proof}
    Let $T = \sum_i a_i\sigma_i$. When $G = \ZZ$,
    \begin{multline*}
        \M(T) = \sum_i \abs{a_i} \Haus^k(\im \sigma_i) = \sum_i \abs{a_i}\var{\sigma_i}_\emptyset(\tilde\Gr_k(M)) = \var{T}_\emptyset (\tilde\Gr_k(M))
        \\
        \eqnote{\cref{lem:Varifold_CompareIntExt}}
        = \var[ext]{\Curr(T)}(\tilde\Gr_k(M)) = \var[ext]{\Curr(T)}(\tilde\Gr_k(\RR^n)) = \M^\ext(T).
    \end{multline*}
    When $G = \ZZ_p$, let $\tilde{T} = \sum_i \abs{a_i}\sigma_i$. Then by \cref{lem:RepModp_Intrinsic} and the $G = \ZZ$ case, $\M(T) = \M(\tilde{T}) = \M^\ext(\Curr(\tilde{T}))$. Therefore, $\Curr(\tilde{T}) \in \Curr(T)$, and thus $\M^\ext(\Curr(T)) = \M^\ext(\Curr(\tilde{T}))$.
\end{proof}

\begin{lemma}
    \label{lem:Mass_CompareIntExt_RelativeNonoverlapping}
    For any $T \in I_k(M,N;G)$, $\M(T) = \var{T}_N(\tilde\Gr_k(M)) = \M^\ext(\Curr(T))$.
\end{lemma}
\begin{proof}
    This follows from \cref{lem:Mass_CompareIntExt_Absolute} and the definition of $\M$ for a relative current.
\end{proof}

\begin{lemma}
    \label{lem:ZeroFlatDistSameCurrent}
    Suppose that $S, T \in I_k(M, N; G)$ satisfy $\Fl(S,T) = 0$. Then $\Curr(S) = \Curr(T)$, $\M(S) = \M(T)$, and $\var{S}_N = \var{T}_N$.
\end{lemma}
\begin{proof}
    First we consider the case where $N = \emptyset$.

    For any $\varepsilon > 0$, choose any $P \in I_{k+1}(M; G)$ and $Q \in I_k(M; G)$ such that $\partial P + Q = S - T$ and $\M(P) + \M(Q) < \varepsilon$. Then $\partial \Curr(P) + \Curr(Q) = \Curr(S) - \Curr(T)$, and by \cref{lem:Mass_CompareIntExt_Absolute}, $\M^\ext(\Curr(P)) + \M^\ext(\Curr(Q)) < \varepsilon$. Thus $\Fl^\ext(\Curr(S), \Curr(T)) \leq \varepsilon$. By letting $\varepsilon \to 0$, we see that $\Fl^\ext(\Curr(S), \Curr(T)) = 0$. Since $\Fl^\ext$ is a metric on the space of rectifiable currents with coefficients in $G$, we conclude that $\Curr(S) = \Curr(T)$.

    Finally, \cref{lem:Mass_CompareIntExt_Absolute} implies that $\M(S) = \M^\ext(\Curr(S)) = \M^\ext(\Curr(T)) = \M(T)$.
\end{proof}

\begin{lemma}
    \label{lem:IntExtEquiv}
    \begin{enumerate}
        \item\label{enum:IntExtEquiv_IntegralCurrentsBilipschitz} $\I{k}(M; G)$ is bilipschitz to $\I[ext]{k}(M; G)$.

        \item When $M$ is closed, $\Z{k}(M;G)$ is bilipschitz to $\Z[ext]{k}(M;G)$.
    \end{enumerate}
\end{lemma}
\begin{proof}
    The embedding $e : M \hookrightarrow \RR^n$ is $L$-bilipschitz for some $L \geq 1$. Let $\N$ be a tubular neighbourhood of $M$, so that the projection map onto $M$ is 2-Lipschitz. (1) and (2) follow from projecting chains and fillings of chains in $\N$ onto $M$.
\end{proof}

\begin{lemma}
    \label{lem:Mass_CompareIntExt_Relative}
    For any $T \in \I{k}(M,N)$, $\M(T) = \M^\ext(\Curr(T))$.
\end{lemma}
\begin{proof}
    This essentially follows from the fact that $\M^\ext$ is lower semicontinuous with respect to $\Fl^\ext$, $\Fl^\ext$ is bilipschitz to $\Fl$ by \cref{lem:IntExtEquiv}(\ref{enum:IntExtEquiv_IntegralCurrentsBilipschitz}), and we defined $\M$ to be lower semicontinuous with respect to $\Fl$ using \cref{eq:MassLowerSemiCtsExtension}.
\end{proof}

\LemCInfinityChainsDense
\begin{proof}
    Choose some embedding $X \hookrightarrow \RR^n$. The key ingredient in this proof is the Approximation Theorem from \cite[4.2.20]{Federer_GMT}, which implies that for any $T \in \I{k}(X)$ and $\varepsilon > 0$, there exists a $k$-chain $S$ in $\RR^n$ that is a linear combination of convex polyhedra so that $S$ is supported within the $\varepsilon$-neighbourhood of $X$, and for some $(1+\varepsilon)$-bilipschitz $C^1$ diffeomorphism $f : \RR^n \to \RR^n$, we have $\M(T - f_\sharp(S)) < \varepsilon$ and $\M(\partial T - \partial f_\sharp(S)) < \varepsilon$. Thus $\Fl(T, f_\sharp(S)) < \varepsilon$. We may smoothen $f$ via a convolution with a $C^\infty$ bump function to get a $C^\infty$ map $f' : \RR^n \to \RR^n$ such that $\Fl(T, f'_\sharp(S)) < \varepsilon$, $\abs{\M(T) - \M(f'_\sharp(S))} < \varepsilon$, and $\abs{\M(\partial T) - \M(\partial f'_\sharp(S))} < \varepsilon$. We may even assume that $f'_\sharp(S)$ intersects each of the given submanifolds of $X$ transversally by performing a small perturbation on $f'$. Now simply let $T_i = f'_\sharp(S)$ for $\varepsilon < \frac1i$.
\end{proof}

\subsection{The inductive limit topology}

\LemContinuityIndLimTop
\begin{proof}
    We will only need to prove the statements for $\Z{}$, as the proof works verbatim for $\I{}$ as well.

    Suppose that $g$ is continuous in the inductive limit topology. Then for any $\mu > 0$ and open set $U \subset Y$, $g^{-1}(U)$ is open in the inductive limit topology which implies that $g^{-1}(U) \cap \Z{k}(M,N)^\mu$ is open in the flat topology on $\Z{k}(M,N)^\mu$. Thus $g$ is continuous in the flat topology when restricted to $\Z{k}(M,N)^\mu$.

    Suppose that $g$ is continuous in the flat topology when restricted to $\Z{k}(M,N)^\mu$ for every $\mu > 0$. Then for any open set $U \in Y$ and $\mu > 0$, $g^{-1}(U) \cap \Z{k}(M,N)^\mu$ is open in the subspace topology of $\Z{k}(M,N)^\mu$. That is, $g^{-1}(U) \cap \Z{k}(M,N)^\mu = V_\mu \cap \Z{k}(M,N)^\mu$ for some $V_\mu \subset \Z{k}(M,N)$ which is open in the flat topology. Thus for any subset $A \subset \Z{k}(M,N)$ of cycles whose masses are less than some $\mu > 0$, we have
    \begin{equation*}
        g^{-1}(U) \cap A = g^{-1}(U) \cap \Z{k}(M,N)^\mu \cap A = V_\mu \cap \Z{k}(M,N)^\mu \cap A = V_\mu \cap A,
    \end{equation*}
    which is open in the subspace topology of $A$. Thus $g^{-1}(U)$ is open in the inductive limit topology. Therefore $g$ is continuous in the inductive limit topology.

    Now we prove the statement about $f$. Assume that $f$ satisfies the stated property. Then for each $\mu > 0$ and the corresponding $\mu'$, $f_\mu$ is continuous with respect to the flat topology in the domain and the inductive limit topology in the codomain, so $f_\mu : \Z{k}(M,N)^\mu \to \Z{k'}(M',N')$ is also continuous in the flat topology in the domain and the inductive limit topology on the codomain. Then it remains to apply the first part of the lemma.
\end{proof}

\subsection{$\delta$-localized families of cycles}

\begin{restatable}[Coarea inequality with boundary mass control]{lemma}{LemCoareaCoverBdryCtrl}
    \label{lem:CoareaCoverBdryCtrl}
    Consider a family of chains $\tau_j \in \I{d_j}(M)$, for $j = 1, \dotsc, k$. Let $B$ be a generalized ball of radius $r$ in $M$. Then for any $0 < s \leq r$ and parameter $0 < \lambda < 1$, $B$ can be replaced by a concentric generalized ball $B_\lambda$ of radius $r_\lambda \in [r, r + s]$ such that
    \begin{align}
        \label{eq:CoareaCoverBdryCtrl}
        \M( \tau_j\res \partial B_\lambda ) &\leq \frac{2k^2}{s}  \M(\tau_j) &\text{and}&&
        \M(\tau \res \partial_{s\lambda/4}B_\lambda) &\leq \lambda k\M(\tau_j).
    \end{align}
\end{restatable}
\begin{proof}
    Let $\alpha = s/r$ and $L = \lceil 1/\lambda \rceil$, and partition the region $(1 + \alpha)B \setminus B$ into $L$ ``annuli'' $A_l = (1 + \frac{l}L\alpha)B \setminus (1 + \frac{l-1}L\alpha)B$. Thus we have
    \begin{equation*}
        \sum_l \M(\tau_j \res A_l) \leq \M(\tau_j)
        \implies \sum_l \sum_j \frac{\M(\tau_j \res A_l)}{\M(\tau_j)}\leq k,
    \end{equation*}
    and by the Pigeonhole Principle, for some $A = A_l$,
    \begin{equation*}
        \sum_j\frac{\M(\tau_j \res A)}{\M(\tau_j)} \leq \frac{k}L \leq \lambda k
        \implies \M(\tau_{j_0} \res A) \leq \sum_j\frac{\M(\tau_j \res A)\M(\tau_{j_0})}{\M(\tau_j)} \leq \lambda k \M(\tau_{j_0}) \text{ for all } j_0.
    \end{equation*}
    Let $A' = (1 + \frac{l - 1/4}L\alpha)B \setminus (1 + \frac{l-3/4}L\alpha)B$. Now for each $j$, we have the coarea inequality
    \begin{align*}
        \int_{r + \frac{l - 3/4}Ls}^{r + \frac{l - 1/4}Ls} \M(\tau_j \res \partial (\tfrac{t}{r}B)) \,dt &\leq \M(\tau_j \res A') 
        \leq \M(\tau_j \res A)
        \leq \frac{k}L\M(\tau_j)
        \\
        \implies \int_{r + \frac{l - 3/4}Ls}^{r + \frac{l - 1/4}Ls} \sum_j \frac{L\M(\tau_j \res \partial (\tfrac{r}{r}B))}{k\M(\tau_j)} \,dt &\leq k.
    \end{align*}
    Therefore for some $r + \frac{l - 3/4}Ls \leq r_L \leq r + \frac{l - 1/4}Ls$, if we choose $B_\lambda = \frac{r_L}{r} B$, then
    \begin{align}
        \sum_j \frac{L\M(\tau_j \res \partial B_\lambda)}{k\M(\tau_j)} &\leq \frac{2Lk}{s}
        \notag\\
        \implies \M(\tau_{j_0} \res \partial B_\lambda)
        \leq \sum_j \frac{\M(\tau_j \res \partial B_\lambda)\M(\tau_{j_0})}{\M(\tau_j)}
        &\leq \frac{2k^2}{s}\M(\tau_{j_0}) \text{ for all } j_0. \label{eq:CoareaCoverBdryCtrl_BallBdry}
    \end{align}
    It remains to observe that $\partial_{s\lambda/4}B_\lambda \subset A$.
\end{proof}

\LemDeltaAdmissibleDoubling
\begin{proof}
    Without loss of generality (a small cost in the constant $c$) we may assume that the elements of $\mathcal{U}_C$ are generalized balls. We may also assume that $\max_C \#\mathcal{U}_C \leq c$.
    
    We will induct on $\dim X$. When $\dim X = 0$ there is nothing to prove. Now assume that the lemma is true for $\dim X = d$. Now let $\dim X = d + 1$, and suppose we are given the initial $\delta$-admissible family $\mathcal{U}_C$. Applying the induction hypothesis to $X^d$ gives us a $c\delta$-admissible family $\mathcal{V}_C$ indexed by cells of dimension at most $d$, which satisfies properties (1)--(3) for those cells, and such that $\#\mathcal{V}_C \leq c$. Now, consider a $(d+1)$-cell $C$ in $X$. We wish to define $\mathcal{V}_C$. For every $d$-face $C'$ of $C$ and $B \in \mathcal{V}_{C'}$, add $2B$ to $\mathcal{U}_C$ to get a family $\mathcal{W}_C$ of generalized balls whose diameters sum to less than $(c + 4dc)\delta < c\delta$. Thus $\max_C \#\mathcal{W}_C \leq c + 2dc \leq c$.

    For each $B \in \mathcal{W}_C$, let $\bar{B} = 2B$. Replace every $B$ with $\bar{B}$ to get families $\bar{\mathcal{W}}_C$, whose elements have diameters summing up to less than $2c\delta \leq c\delta$.  
    
    Consider the following \emph{merging  procedure} for $\bar{\mathcal{W}}_C$. For any $\overline{B_1}, \overline{B_2} \in \bar{\mathcal{W}}_C$ that touch, replace them with $2B$, where $B$ is the smallest generalized ball containing both of them. Keep repeating within $\bar{\mathcal{W}}_C$ until it cannot be repeated. We are left with a family $\bar{\mathcal{V}}_C$ of generalized balls, perhaps fewer in number, whose diameters sum to less than $2^{\#\bar{\mathcal{W}}_C}c\delta \leq c\delta$. Now replace each generalized ball in $\bar{\mathcal{V}}_C$ with the concentric generalized ball of half the radius. The result is the desired family $\mathcal{V}_C$.

    After carrying out the above process for each $(d+1)$-cell $C$, we have defined the collections $\mathcal{V}_C$ for every cell $C$ of positive dimension. The family $\{\mathcal{V}_C\}_{C \subset X^{d+1}}$ is doubling due to the radius-halving step at the end, and the elements of $\mathcal{V}_C$ are added (with radius doubled) to $\mathcal{U}_{C'}$ when $C$ is a codimension-1 face of $C'$. Finally, each element $\mathcal{U}_C$ is contained in some element of $\mathcal{V}_C$ because the sets involved only grow larger or are merged into larger sets. We also have $\max_C \#\mathcal{V}_C \leq c + \max_C \#\mathcal{W}_C \leq c$.
\end{proof}

\PropDiscreteDeltaLocExtension
\begin{proof}
    This follows from the proof of \cite[Proposition~2.7]{GuthLiokumovich_ParamIneq}. That proof inductively constructs a sequence of $c_1\delta$-localized maps $F_j : X^j(q_j)^0 \to \Z{k}(M,N)$ for $0 = q_0 < q_1 < \dotsb < q_n = q_n(n,M,\delta)$, where $F_0 = F$ and $F_j$ extends $F_{j-1}$, and for $N = \partial M$. It can be adapted for the case where $(M,N)$ is a collar pair. We will prove by induction that each $F_j$ satisfies the properties (\ref{enum:DiscreteDeltaLocExtension_MassBound}) and (\ref{enum:DiscreteDeltaLocExtension_Inductive}) with $X^j$ instead of $X^j$ and $X(q_j)$ instead of $\tilde{X}$, and $\chi_{F_j}(r) \leq c_3(r/\delta)(\chi_f(r) + \varepsilon)$ for all $r > 0$. Then $F_n$ will be our desired $F$.

    This is vacuously true for $j = 0$. Assume by induction that it is true for $j-1$. $F_j$ is first defined over the vertices of $X^{j-1}(q_j)$ by $F_j(x) = F_{j-1}(x')$, where $x'$ is a ``closest vertex of $X^{j-1}(q_{j-1})$'' to $x$, in a sense defined formally in the proof. This ensures that (\ref{enum:DiscreteDeltaLocExtension_Inductive}) holds for each cell of positive dimension in $X^{j-1}$.
    
    For each $j$-cell $C$ of $X$, $F_j$ is then extended over the vertices of $C(q_j)$ using a discrete analogue of the method to extend a map $g : \partial C \to \Z{k}(M,N)$ over $C$ by finding a nullhomotopy $\partial C \times I \to \Z{k}(M,\Fl)$ of $g$. More precisely, a map $G : \partial C(q_j)^0 \times \{0,1,\dotsc,3^{q_j-1}\} \to \Z{k}(M,N)$ is constructed so that $G(-,0) = F_j$ and $G(-,3^{q_j})$ is a constant map. A vertex $v$ of $C$ is fixed. To define $G$, \cite[Lemma~2.8]{GuthLiokumovich_ParamIneq} is used to find a family of ``small'' chains $\tau : \partial C(q_j)^0 \to \I{k+1}(M,N)$ such that $\partial\tau(x) = f(v) - F_j(x)$ and $\M(\tau(x)) \leq c_2(n,M,\delta)\varepsilon$. Some coarea family $\{U_1,\dotsc,U_N\}$ for $N \leq 3^{q_j-1}$ is found for $\tau$, which are defined from generalized balls of diameter at most $\delta$. Apply \cref{lem:CoareaCoverBdryCtrl} to $\{F_j(x), \tau(x)\}_{x \in \partial C(q_j)^0}$ (and $\lambda_i$) to replace those generalized balls with new ones, to get some new family which we also call $\{U_i\}$, and which satisfies \cref{eq:DiscreteDeltaLocExtension_Coarea,eq:DiscreteDeltaLocExtension_BdryControl} for $F_j(x)$ and $\tau$ instead of $F(x)$. For each $U_i$, let $y_i$ be a vertex $y$ of $\partial C(q_j)$ that minimizes $\M(F_j(y) \res U_i)$. Then it follows from \cite[(14)]{GuthLiokumovich_ParamIneq} that $G$ has the following formula,
    \begin{align*}
        G(x,t)
        &= f(v) - \partial\left( \tau(x) \res \bigcup_{i = 1}^{N-t} U_i + \sum_{i = N-t+1}^N \tau(y_i) \res U_i \right)
        \\
        &= F_j(x) \res \bigcup_{i = 1}^{N-t} U_i + \sum_{i = N-t+1}^N F_j(y_i) \res U_i - \underbrace{ \tau(x) \res \partial \bigcup_{i = 1}^{N-t} U_i - \sum_{i = N-t+1}^N \tau(y_i) \res \partial U_i}_B.
    \end{align*}

    By \cite[(5)]{GuthLiokumovich_ParamIneq},
    \begin{align*}
        \M(B)
        &\leq \frac{2N \#(\partial C(q_j)^0)}{r_0} \max\{\M(\tau(x)), \M(\tau(y_1),\dotsc, \M(\tau(y_N))\}
        \\
        &\leq c(n,M,\delta)\varepsilon,
    \end{align*}
    which gives \cref{eq:DiscreteDeltaLocExtension}.

    Let us prove that $F_j$ satisfies \cref{eq:DiscreteDeltaLocExtension_Coarea,eq:DiscreteDeltaLocExtension_BdryControl}. Suppose that $F_j(z) = G(x,t)$. Then
    \begin{align*}
        \M(F_j(z) \res \partial U_l) 
        \leq{}& \M(F_j(x) \res \partial U_l) + \sum_{i = N - t + 1}^N \M(F_j(y_i) \res \partial U_l)\\
        +{}& \M(\tau(x) \res \partial U_l) + \sum_{i = N - t + 1}^N \M(\tau(y_i) \res \partial U_l),
    \end{align*}
    and it suffices to recall that $\M(\tau(x)) \leq c_2\varepsilon$ and note that when $x$ is a vertex of $\partial C(q_j)$, the $F_j(x)$'s and $\tau(x)$'s satisfy \cref{eq:DiscreteDeltaLocExtension_Coarea,eq:DiscreteDeltaLocExtension_BdryControl}.

    To prove that $F_j$ satisfies (\ref{enum:DiscreteDeltaLocExtension_MassBound}), for any $F_j(z) = G(x,t)$,
    \begin{align*}
        \M(F_j(z)) 
        &\leq \M\left( F_j(x) \res \bigcup_{i = 1}^{N-1} U_i \right) + \sum_{i = N-t+1}^N \M(F_j(y_i) \res U_i) + \M(B)
        \\
        &\leq \M\left( F_j(x) \res \bigcup_{i = 1}^{N-1} U_i \right) + \sum_{i = N-t+1}^N \M(F_j(x) \res U_i) + c\varepsilon
        \\
        &\leq \M(F_j(x)) + c\varepsilon
        \\
        &\leq \max_{x \in X^{j-1}(q_{j-1})^0} \M(F_{j-1}(x)) + c\varepsilon
        \\
        (\text{induction hypothesis}) &\leq \max_{x \in X^0} \M(f(x)) + c\varepsilon.
    \end{align*}

    After defining $F$ over $\tilde{X}^0$ as above, let us show that it satisfies  (\ref{enum:DiscreteDeltaLocExtension_FlatBound}). Consider any cell $C$ of $X$, and vertex $v$ of $C$, and any $C \cap \tilde{X}^0$. If $x$ is also a vertex of $C$, then the property holds because $f$ is $\varepsilon$-fine. Otherwise, let $C$ be the cell of lowest dimension in $X$ that contains $x$. Then $F(x) = G(x',t)$ is $f(v')$ minus the boundary minus a chain of mass at most $Nc_2\varepsilon \leq c\varepsilon$, where $v'$ is a vertex of $C$. The property holds because $\Fl(f(v), f(v')) < \varepsilon$.
\end{proof}

\subsection{The gluing chain homomorphism}

\PropPiecewiseSmoothGluing
\begin{proof}

    We will prove the composition law, and the others follow from elementary geometric constructions. For any $x \in L$, $\Phi^c \circ \Psi^0(x) = \Phi^c(g(x))$. In addition, every $\Phi^{l+c} \circ \Psi^l$ is continuous. And we have $\partial \circ \Phi^{l+c} \circ \Psi^l = \Phi^{l+c-1} \circ \partial \circ \Psi^l = \Phi^{l+c-1} \circ \Psi^{l-1} \circ \partial$.
\end{proof}

\section{Topological Lemmas about Lens Spaces}
\label{sec:Appendix_TopLemmasLensSpaces}

\subsection{The conventional choice of cohomology generators for lens spaces}
\label{sec:LensSpaceCohomConvention}

In this paper we consider all cohomology to be singular cohomology. On the other hand, for spaces like $E\ZZ_p$ it is convenient to define classes using cellular cohomology, which only makes sense after we have chosen a certain isomorphism $C_n^{cell}(X) \xrightarrow{\iso} H_n(X^n,X^{n-1})$, from which the cellular boundary maps are defined, and which induces an isomorphism $C^n_{cell}(X;G) = Hom(C_n^{cell}(X), G) \xrightarrow{\iso} Hom(H_n(X^n,X^{n-1}), G) \iso H^n(X^n,X^{n-1};G)$. (The last isomorphism comes from the Universal coefficient theorem.)

We could orient each cell using the standard orientation of a disk $\D^n$, but then we have to choose a fundamental class in the singular cohomology $H^n(\D^n,\partial\D^n)$.

Thus to compute the cellular homology and cohomology of $\Lens_p^\infty$, it would be best for us to define a $\ZZ_p$-simplicial complex structure on $\Sp\infty$. Do this inductively, modifying the construction in \cite[p.~145]{Hatcher_AlgTop}. Give $\Sp1$ the simplicial complex structure of a counterclockwise regular $p$-gon. That is, the vertices are the roots of unity $1, \zeta_p, \dotsc, \zeta_p^{p-1}$, where $\zeta_p = e^{2\pi i/p}$, and the edges are $[\zeta_p^j,\zeta_p^{j+1}]$ where $1 \leq j \leq p$. Inductively, assume that $\Sp{2n-1}$ has been given a $\ZZ_p$-simplicial complex structure. Then the structure on $\Sp{2n+1}$ is defined as follows: take the $\Sp1$ in the $(n+1)^\text{th}$ $\C$ factor and give it the counter-clockwise regular $p$-gon simplicial complex structure. Then for any $k$-simplex $[v_0,\dotsc,v_k]$ in $\Sp{2n-1}$, add $p$ $(k+1)$-simplices $[\zeta_p^j, v_0,\dotsc,v_k]$ and $p$ $(k+2)$-simplices $[\zeta_p^j, \zeta_p^{j+1} v_0,\dotsc,v_k]$ for $1 \leq j \leq p$. This structure is a triangulation of each cell, and it is compatible with the (relative) fundamental classes of each cell taken to be simply the sum of the top-dimensional simplices in the cell that have been mentioned so far. In this way, the cellular boundary maps are really multiplication by $p$.

With orientations for each cell chosen in this way, let $\omega_k$ denote the dual to the $k$-cell.

\begin{lemma}
    \label{lem:SphereQuotientSuspLens_Real}
    Let $n \geq 1$. Then thinking of $\Sp{p(n+1)-1}$ as the unit sphere in $(\RR^{n+1})^p$, and considering the $\ZZ_p$ action on $(\RR^{n+1})^p$ by cyclic permutations, $(\Sp{p(n+1)-1}/\tilde\Delta)/\ZZ_p$ is homeomorphic to the Thom space $(L_n \times \D^{n+1})/(L_n \times \partial\D^{n+1})$, where $\tilde\Delta = \Sp{p(n+1)-1} \cap \{(x,\dotsc,x) \in \RR^{p(n+1)}~|~x \in \RR^{n+1}\}$.
\end{lemma}

(Recall that $L_n$ was defined in \cref{eq:SphereQuotientLensBundle}.)

\begin{proof}
    Let us prove this for odd primes $p$ as the case where $p = 2$ is similar. We will show that $\Sp{p(n+1)-1} - \tilde\Delta$ is homeomorphic to a trivial rank $n+1$ real vector bundle over $\Lens_p((1,\dotsc,\frac{p-1}2)^{n+1})$.
    
    $\ZZ_p$ acts on $\Sp{p(n+1)-1} \subset \RR^{p(n+1)}$ by $M_{n+1} = M \otimes I_{n+1}$, where $M$ is the $p \times p$ circulant matrix. We have $M_{n+1} = F_{n+1}D_{n+1}F_{n+1}^{-1}$, where $F_{n+1} = F_\RR \otimes I_{n+1}$ and $D_{n+1} = D_\RR \otimes I_{n+1}$. Since $F_{n+1}$ is orthogonal, $F_{n+1}^{-1}$ restricts to a homeomorphism $\Sp{p(n+1)-1} \to \Sp{p(n+1)-1} \subset \RR^{p(n+1)}$, where both the domain and target spheres are the unit sphere. Moreover, $F_{n+1}^{-1}$ sends $\tilde\Delta$ to $\Sp{p(n+1)-1} \cap V$, where $V = \vspan\{v_{0m}\}_{m = 1}^{n+1}$.
    
    Note that $\dim V = n+1$ and thus $\dim V^\perp = (p-1)(n+1)$. Since every vector subspace of $\RR^{p(n+1)}$ intersects transversally with $\Sp{p(n+1)-1}$, $\Sp{p(n+1)-1} \cap V$ is homeomorphic to $\Sp{n}$ and $\Sp{p(n+1)-1} \cap V^\perp$ is homeomorphic to $\Sp{(p-1)(n+1)-1} = \Sp{pn-n+p-2}$.

    Therefore there is a homeomorphism from $\Sp{p(n+1)-1}$ to the Thom space of the trivial rank $n+1$ real bundle over $\Sp{(p-1)(n+1)-1}$, which takes a point in $\Sp{p(n+1)-1}$, expresses it uniquely as $u + v$ for $u \in V^\perp$ and $v \in V$, then outputs the basepoint if $\norm{v} = 1$, and outputs $(u/\norm{u}, v)$ otherwise.

    $\ZZ_p$ acts via $D_{n+1}$ on the target $\Sp{p(n+1)-1}$, which is the identity on $V$ and rotates each $W_{jm}$ by angle $j\theta$, for $\theta = 2\pi/p$. Hence if we write the coordinates of $V^\perp$ as $(z_{jm})$, where $z_{jm} = x_{jm} + iy_{jm}$ has real and imaginary parts being the real coordinates associated with the respective basis vectors $u_{jm}$ and $w_{jm}$, then $D_{n+1}$ acts on $\Sp{(p-1)(n+1)-1}$ by the map $z_{jm} \mapsto \zeta_p^j z_{jm}$; thus $\Sp{(p-1)(n+1)-1}/\ZZ_p \he \Lens_p((1,\dotsc,\frac{p-1}2)^{n+1}) = L$. $D_{n+1}$ also sends the fibers of the trivial complex bundle to each other while preserving the $v_{0m}$-coordinates, so the quotient of the bundle by the action of $D_{n+1}$ is the trivial rank-$(n+1)$ real vector bundle over $L$, which we can write as $L \times \RR^{n+1}$.

    Finally, $(\Sp{p(n+1)-1}/\tilde\Delta)/\ZZ_p$ is the one-point compactification of $(\Sp{p(n+1)-1} - \tilde\Delta)/\ZZ_p \he L \times \RR^{n+1}$, which is the Thom space of $L \times \RR^{n+1}$.
\end{proof}

\bibliography{References}

\end{document}